\scrollmode
\documentclass[11pt]{amsart}
\usepackage[T1]{fontenc}
\usepackage{amsmath}
\usepackage{mathrsfs}
\usepackage{bbm}
\usepackage{amssymb}
\usepackage{stmaryrd}
\usepackage[nointegrals]{wasysym}
\usepackage[all]{xy}
\usepackage{tensor}
\usepackage{verbatim}
\usepackage{graphicx}
\usepackage{scalerel}
\usepackage{hyperref}
\usepackage{xcolor}
\usepackage{mathtools}
\usepackage{color}
\usepackage{slashed}

\theoremstyle{plain}
\newtheorem{thm}{Theorem}[section]
\newtheorem{lem}[thm]{Lemma}
\newtheorem{prop}[thm]{Proposition}
\newtheorem{cor}[thm]{Corollary}
\newtheorem{fact}[thm]{Fact}

\newtheorem{ques}[thm]{Question}

\newtheorem*{thmA}{Theorem A}
\newtheorem*{thmB}{Theorem B}
\newtheorem*{corC}{Corollary C}

\newtheorem*{thmE}{Theorem E}

\newtheorem*{corF}{Corollary F}
\newtheorem*{thmG}{Theorem G}
\newtheorem*{quesD}{Question D}

\theoremstyle{definition}

\newtheorem*{settWLIT}{Setting [WLIT]}
\newtheorem*{settWLIT+}{Setting [WLIT+]}
\newtheorem*{settP}{Setting [(P)-cl]}
\newtheorem*{settP+}{Setting [(P)-cl+]}
\newtheorem*{settG}{Setting [$\Gamma$]}

\newtheorem{conv}{Convention}[thm]

\theoremstyle{remark}
\newtheorem{rem}{Remark}[thm]

\numberwithin{equation}{section}

\usepackage[autostyle, english = american]{csquotes}
\MakeOuterQuote{"}

\newcommand{\Mat}{\mathrm{Mat}}

\newcommand{\SL}{\mathrm{SL}}

\newcommand{\PSL}{\mathrm{PSL}}

\newcommand{\caF}{\mathcal{F}}

\newcommand{\Aut}{\mathrm{Aut}}
\newcommand{\Sym}{\mathrm{Sym}}

\newcommand{\C}{\mathbb{C}}
\newcommand{\F}{\mathbb{F}}

\newcommand{\Z}{\mathbb{Z}}
\newcommand{\Q}{\mathbb{Q}}
\newcommand{\R}{\mathbb{R}}

\newcommand{\caG}{\mathcal{G}}
\newcommand{\caU}{\mathcal{U}}
\newcommand{\caW}{\mathcal{W}}
\newcommand{\caY}{\mathcal{Y}}

\newcommand{\caE}{\mathcal{E}}
\newcommand{\caO}{\mathcal{O}}

\newcommand{\calP}{\mathcal{P}}

\newcommand{\caV}{\mathcal{V}}

\newcommand{\euE}{\mathscr{E}}

\newcommand{\caL}{\mathcal{L}}

\newcommand{\bbT}{\mathbb{T}}

\newcommand{\Geod}{\mathrm{Geod}}

\newcommand{\calPl}{\calP^{\mathrm{lift}}}

\newcommand{\Nv}{N_{\mathrm{vert}}}
\newcommand{\Ne}{N_{\mathrm{edg}}}

\newcommand{\frp}{\mathfrak{p}}
\newcommand{\frq}{\mathfrak{q}}

\newcommand{\lad}{\Delta=(\Gamma, (X_a), (G(c)))}

\newcommand{\caWr}{\mathcal{W}_{\mathrm{rev}}}

\newcommand{\tchi}{\widetilde{\chi}}

\newcommand{\caZ}{\mathcal{Z}}

\newcommand{\Real}{\mathrm{Re}}

\newcommand{\argu}{\hbox to 1.5ex{\hrulefill}}  

\newcommand{\QG}{\Q[G]}

\newcommand{\til}{~}

\theoremstyle{definition}
\newtheorem{defn}[thm]{Definition}

\newtheorem{ex}[thm]{Example}

\newtheorem{notat}[thm]{Notation}

\usepackage{graphicx}
\usepackage{tikz}
\usepackage{tikz-cd}


\begin{document}
	
\begin{abstract}
We study the double-coset zeta functions for groups acting on trees, focusing on weakly locally $\infty$-transitive or (P)-closed actions. After giving a geometric characterisation of convergence for the defining series, we provide explicit determinant formulae for the relevant zeta functions in terms of local data of the action.
Moreover, we prove that the evaluation of these zeta functions at~$-1$ satisfies the expected identity with the Euler--Poincaré characteristic of the group. 
 Our framework suggests the definition of a new zeta function counting weighted paths in graphs, which recovers some double-coset zeta functions considered in this paper as particular cases.
The behaviour at~$-1$ of the latter zeta function sheds light on a connection with the Ihara zeta function of a weighted graph introduced by A.~Deitmar. 
\end{abstract}

\title[Double-coset zeta functions for groups acting on trees]{Double-coset zeta functions\\ for groups acting on trees}
\author[B.~Marchionna]{Bianca Marchionna}  
\address{ Université catholique de Louvain\\
Research Institute in Mathematics and Physics\\
Chemin du Cyclotron 2, 1348 Louvain-la-Neuve, Belgium}
\email{bianca.marchionna@uclouvain.be}
\date{\today}
	
\maketitle

\section{Introduction}\label{intro}
\subsection*{Background and motivation}
Double cosets play a prominent role in multiple aspects of group theory and beyond. For instance, they are the building blocks of Hecke algebras. Regarded as collections of cosets, they describe spheres with respect to the Weyl distance in the building associated with a Bruhat decomposition of a group~\cite{apv:build}. Strictly related objects, namely the suborbits, are also widely studied in permutation group theory. 
It is common to arrange the suborbit sizes -- provided they are all finite -- in a non-decreasing sequence $(a_n)_{n\geq 1}$ and estimate, for example, the growth of~$a_n$ as a function of $n$ (see, for instance,~\cite{infperm} or~\cite[\S5]{sm:subsiz}). 

The present paper focuses on an alternative approach to the suborbit (or double-coset) growth, recently introduced by I.~Castellano, G.~Chinello and T.~Weigel~\cite{ccw:zeta}. We briefly outline it here in the slightly more general framework we employ. Let $G$ be a group with subgroups $H,K\leq G$ satisfying $|HgK/K|<\infty$ for every $g\in G$. Note that $|HgK/K|$ is the size of the $H$-orbit of $gK$ in the coset space $G/K$. Such a triple $(G,H,K)$ has the \emph{double-coset property} if, for every $n\geq 1$,
\begin{equation}\label{eq:an}
   a_n(G,H,K):=|\{HgK\in H\backslash G/K\,:\,|HgK/K|=n\}|<\infty.
\end{equation} 
If there is some $\alpha\in \R_{\geq 0}$ such that~$a_n(G,H,K)=O(n^\alpha)$ for every~$n$, we say that $(G,H,K)$ has \emph{polynomial double-coset growth}. 

If $(G,H,K)$ has polynomial double-coset growth, then an indirect (but often more convenient) way to study the sequence $(a_n(G,H,K))_{n\geq 1}$ is through the Dirichlet series the sequence generates, that is,
\begin{equation}\label{eq:DCzeta}
    \zeta_{G,H,K}(s):=\sum_{n=1}^{\infty}a_n(G,H,K)\cdot n^{-s}=\sum_{HgK\in H\backslash G/K}|HgK/K|^{-s},
\end{equation}
where $s$ is a complex variable. By~\cite[\S I.1, Theorem~3]{hari}, this series converges in some half-plane $\{s\in \C\mid \Real(s)>\alpha\}$ and the function it determines is called the \emph{double-coset zeta function of $(G,H,K)$}. 

Zeta functions initially arose in number theory, although nowadays they are also an established tool in studying groups, rings and algebras (see M.~du~Sautoy's survey~\cite{mds:intro} for motivation). 
In group theory, the seminal work of F.~Grunewald, D.~Segal, G.~Smith~\cite{gss:zeta} initiated the study of numerous zeta functions associated to finitely generated nilpotent groups or profinite groups. Double-coset zeta functions have been one of the first instances of zeta functions in the class of totally disconnected locally compact (= t.d.l.c.) groups which are possibly neither discrete nor profinite. Their introduction was motivated by an interesting behaviour at $s=-1$. I.~Castellano, G.~Chinello and T.~Weigel~\cite{ccw:zeta} have provided examples of unimodular t.d.l.c.~groups~$G$ with a compact open subgroup $K\leq G$ for which the meromorphic continuation of $\zeta_{G,K,K}(s)$, evaluated at $s=-1$, recovers the Euler--Poincaré characteristic $\widetilde{\chi}_G$ of~$G$ (in the sense of~\cite[\S5]{ccw:zeta}). Namely, one has
\begin{equation}\label{eq:chi-1}
\widetilde{\chi}_G=\zeta_{G,K,K}(-1)^{-1}\mu_K,
\end{equation}
where $\mu_K$ denotes the left Haar measure on~$G$ normalised with respect to~$K$. A pair $(G,K)$ for which~\eqref{eq:chi-1} holds is said to satisfy the \emph{Euler--Poincaré identity}.

The connection between growth series and Euler--Poincaré characteristics is not an isolated phenomenon, and usually it gives rise to surprising connections between two apparently unrelated objects. The reader is referred, for instance, to~\cite{dym:thin,gri:lang,smy:eul} or the introduction of~\cite{ccw:zeta} for some remarkable examples.
One of the main goals of the present paper is to prove that the Euler--Poincaré identity holds in two relevant classes of unimodular t.d.l.c.~groups acting on trees. 
More generally, we present a systematic study of the double-coset zeta functions for groups (not necessarily t.d.l.c.)~acting on trees, including convergence criteria and explicit formulae. 

\begin{conv}\label{conv:intro}
In what follows, every graph $\Gamma=V\Gamma\sqcup E\Gamma$ is meant in the sense of J-P.~Serre~\cite{ser:trees}, and~$V\Gamma$ and~$E\Gamma$ denote the set of vertices and the set of edges of~$\Gamma$, respectively (cf.~Section~\ref{sus:graph}). 

Moreover, we assume that every group action on a tree is \emph{without edge inversions} and the tree has at least one edge (cf.~Section~\ref{sus:actions}). 
Finally, the stabilisers of adjacent vertices are assumed to be \emph{incomparable} with respect to the inclusion.
\end{conv}

\subsection*{Two relevant properties for group actions on trees}\label{act}
In the present paper, we mainly consider group actions on trees with one of the following two properties. 

The first property is \emph{weak local $\infty$-transitivity}. A group action on a tree~$(G,T)$ is weakly locally $\infty$-transitive if, for every $v\in VT$ and every path $\frp$ in the quotient graph $G\backslash T$, the stabiliser of $v$ acts transitively on the set of all geodesics in~$T$ starting at $v$ and lifting $\frp$. This condition, which we introduce in Section~\ref{sus:wlit}, generalises the well-known notion of \emph{local $\infty$-transitivity} (or \emph{boundary $2$-transitivity}). 
The study of local transitivity properties of groups acting on graphs goes back to W.~Tutte~\cite{tutte}. The work of M.~Burger and S.~Mozes~\cite{bumo, bumo2} brought a specific attention to locally $\infty$-transitive actions on trees: indeed, they exploited this condition to construct lattices in products of trees. Later on, N.~Radu~\cite{radu} provided a broad classification result for locally $\infty$-transitive actions on trees, and further insights on those actions have been provided more recently by C.~Reid~\cite{reid}. Our notion of weak local $\infty$-transitivity extends the one of local $\infty$-transitivity to non-arc transitive group actions on trees (cf.~Lemma~\ref{lem:(W)LIT}). 

The second property we consider is \emph{(P)-closedness}, a property that stems from the more general concept of actions on trees with Tits' property~(P). The latter property has been introduced by J.~Tits~\cite[\S4.2]{tit:tree} and is recalled in Section~\ref{sus:Pcl}. A group action on a tree~$(G,T)$ is \emph{(P)-closed} if it has Tits' property~(P) and $G\leq \Aut(T)$ is closed with respect to the permutation topology. Tits' property~(P) play a central role in the theory of groups acting on trees and even beyond. For instance, one may use a (P)-closed action to produce and study properties of simple groups acting on trees (cf.~for instance~\cite[Théorème~4.5]{tit:tree}, \cite[Theorem~7.3]{bew:simp},~\cite[Theorem~1.8]{resm:lad}, and~\cite{cadm}).

We thank one of the anonymous referees to have brought the attention to the following clarification.
\begin{rem}
    Another (equivalent) way to define (P)-closed actions on trees is to require that the group~$G\leq \Aut(T)$ of the relevant action~$(G,T)$ equals its \emph{(P)-closure} (or (P${}_1$)-\emph{closure}). The latter is a suitably \emph{closed} subgroup of~$\Aut(T)$ that has been first introduced by C.~Banks, M.~Elder, and G.~Willis~\cite[\S3]{bew:simp} for locally finite trees under the initial name of \emph{$1$-closure}. The term "$1$-closure" is problematic (cf.~the footnote in~\cite[\S 2.3]{resm:lad}), so it has been successively replaced by "(P)-closure" (or "(P${}_1$)-closure", as there is a more general notion of "(P${}_k$)-closure" for every $k\geq 1$).
    
    In~\cite[\S 2.3]{resm:lad}, C.~Reid and S.~Smith extend the definition of (P)-closure to group actions on arbitrary trees. Their notion agrees with the one in~\cite{bew:simp} if the tree is locally finite, and is equivalent to the definition that we use in the present paper by~\cite[Theorem~2.18]{resm:lad}.
\end{rem}

Remarkably, C.~Reid and S.~Smith~\cite{resm:lad} provide a complete classification of (P)-closed actions on trees by their \emph{local action diagram}, which is a local datum attached to each group action on a tree. In~Section~\ref{sus:lads} we briefly recall it. With this local description, one may study several global properties of the group action through more accessible features of the associated local action diagram. In the present paper, we exploit this approach more than once. 

The reader may find it convenient to keep the following example in mind, as it represents a group action on a tree with both the two properties mentioned before.
\begin{ex}\label{exintro}
  Let $\Gamma$ be a connected graph with a function $\omega\colon E\Gamma\to \Z_{\geq 1}$ (called \emph{edge weight}). Following~\cite[\S 3.5]{car}, the pair $(\Gamma,\omega)$ admits an essentially unique universal cover $(T,\pi)$, which consists of a tree $T$ and a graph epimorphism $\pi\colon T\to \Gamma$ with the following property: for all $a\in E\Gamma$ and $v\in VT$ with $\pi(v)=o(a)$, the number of edges $e\in ET$ with origin~$v$ and $\pi(e)=a$ is exactly $\omega(a)$.
  The \emph{group of deck transformations} of $(\Gamma,\omega)$ is
  \begin{equation*}\label{eq:AutpT}
      \Aut_\pi(T):=\{\varphi\in \Aut(T)\mid \pi\circ\varphi=\pi\},
  \end{equation*}
  where~$\Aut(T)$ denotes the group of automorphisms of $T$. 
  The group~$\Aut_\pi(T)$ is t.d.l.c.~with respect to the permutation topology. Moreover, the $\Aut_\pi(T)$-action on~$T$ is both weakly locally $\infty$-transitive and (P)-closed (cf.~Example~\ref{ex:WLIT}(i) and~\cite[\S4.2]{tit:tree}). 
\end{ex}

\subsection*{Polynomial double-coset growth}\label{mares}
The first step towards the study of double-coset zeta functions is to determine under which conditions a triple $(G,H,K)$ has the double-coset property or polynomial double-coset growth. Here we only consider the case of~$G$ acting on a locally finite tree and we choose~$H,K$ among the stabilisers of vertices or edges of~$T$ (written $H=G_{t_1}$ and $K=G_{t_2}$, for $t_1,t_2\in T=VT\sqcup ET$).
In this setting, we provide the following characterisation.

\begin{thmA}[\protect{cf.~Theorems~\ref{thm:convWLIT} and~\ref{thm:convP}}]\label{thmA}
Let $G$ be a group that acts on a locally finite tree~$T$ as in Convention~\ref{conv:intro} and with finite quotient graph. Assume that the action~$(G,T)$ is weakly locally $\infty$-transitive or (P)-closed. Then the following are equivalent, for all $t_1,t_2\in T$:
\begin{itemize}
    \item[(i)] $(G, G_{t_1}, G_{t_2})$ has the double-coset property;
    \item[(ii)] $(G, G_{t_1}, G_{t_2})$ has polynomial double-coset growth;
    \item[(iii)] for every $e\in ET$, $G_e$ is not contained in any end stabiliser.
\end{itemize}
\end{thmA}
By Theorem~\hyperref[thmA]{A}, both the double-coset property and the polynomial double-coset growth are independent of the choice of $t_1$ and $t_2$. A similar independence has been shown in~\cite[Proposition~6.2]{ccw:zeta}. 

In Theorem~\hyperref[thmA]{A}, if $(G,T)$ is (P)-closed then~(iii) is equivalent to the condition that no end stabilisers are open or, equivalently, that no ends of~$T$ are \emph{locally invariant} (cf.~\cite[Definition~4.12, Proposition~4.23]{resm:lad}). 

A partial version of Theorem~\hyperref[thmA]{A} with more general hypotheses on the action holds and can be found in~Proposition~\ref{prop:conv*}.

\subsection*{Explicit determinant formulae}\label{mero}
The main motivation for considering weakly locally $\infty$-transitive or (P)-closed actions on locally finite trees is the following: provided $H,K$ are either vertex or edge stabilisers, we can count the $(H,K)$-double-cosets (or the $K$-cosets) and compute each size $|HgK/K|$ in terms of convenient \emph{local data} of the action (cf.~Sections~\ref{s:suborb}). 

In the weakly locally $\infty$-transitive case, counting the $(H,K)$-double cosets is rephrased in a more accessible counting of certain paths in the quotient graph $\Gamma$ with a suitably defined weight (cf.~Section~\ref{sus:Gammaw}). This also suggests the definition of a more general Dirichlet series $\caZ_{\Gamma,u_1\to u_2}(s)$ associated to an arbitrary graph~$\Gamma$ with edge weight $\omega\colon E\Gamma\to \Z_{\geq 1}$ and to $u_1,u_2\in \Gamma$ (cf.~Definition~\ref{defn:Zu}). The series $\caZ_{\Gamma,u_1\to u_2}(s)$ recovers $\zeta_{G,G_{t_1},G_{t_2}}(s)$ whenever~$G$ is a group acting weakly locally $\infty$-transitively on a locally finite tree $T$ with quotient graph~$\Gamma$ and standard edge weight $\omega$ (cf.~Section~\ref{sus:Gammaw}), and~$t_1,t_2\in T$ satisfy~$G\cdot t_1=u_1$ and~$G\cdot t_2=u_2$.

In the (P)-closed case, we proceed in a similar manner, except that we count $K$-cosets instead of $(H,K)$-double cosets. This has only a minor impact of $\zeta_{G,H,K}(s)$, as discussed in~Section~\ref{sus:coset}. In contrast to $(H,K)$-double-cosets, the $K$-cosets can be enumerated using suitable weighted paths in the local action diagram associated to the action (cf.~Definitions~\ref{defn:pathD}~and~\ref{defn:W}).

In both cases, we can adapt ideas and techniques from graph theory (e.g., counting paths in graphs by their weight, see for instance~\cite{deit:prim} or~\cite{deika:ihara}) and provide explicit formulae for $\zeta_{G,G_{t_1},G_{t_2}}(s)$, for all $t_1,t_2\in T$.

In the following, we label $\zeta_{G,G_{t_1}, G_{t_2}}(s)$ with a superscript $\bullet\in\{(w),(p)\}$ to distinguish whether $(G,T)$ is weakly locally $\infty$-transitive ($\bullet=(w)$) or (P)-closed ($\bullet=(p)$).
\begin{thmB}[\protect{cf.~Theorems~\ref{thm:meroWLIT}~and~\ref{thm:meroP}}]\label{thmB}
  Let $(G,T)$ be a group action on a locally finite tree as in Convention~\ref{conv:intro} that is weakly locally $\infty$-transitive or (P)-closed. Let $t_1,t_2\in T$ be such that $(G,G_{t_1},G_{t_2})$ has polynomial double-coset growth. Then
  \begin{equation*}
      \zeta_{G,G_{t_1}, G_{t_2}}^\bullet(s)=\frac{\det(I^\bullet-\caE^\bullet(s)+\caU^\bullet_{t_1,t_2}(s))}{\det(I-\caE^\bullet(s))}+\epsilon^\bullet_{t_1}(t_2),
  \end{equation*}
  for explicitly defined square matrices $\caE^{\bullet}(s)$ and $\caU^{\bullet}_{t_1,t_2}(s)$ whose entries are entire functions in~$s\in \C$, and for a determined integer $\epsilon^\bullet_{t_1}(t_2)$. Here $I^\bullet$ denotes the identity matrix of the same dimension as $\caE^\bullet(s)$.
  In particular, $\zeta^\bullet_{G,G_{t_1},G_{t_2}}(s)$ extends to a meromorphic function on $\C$.
\end{thmB}
In Theorem~\hyperref[thmB]{B}, the matrix $\caE^\bullet(s)$ can be interpreted as a weighted adjacency matrix of the local structure in which we count the paths (cf.~Definitions~\ref{defn:Es}~and~\ref{defn:Fs}). The matrix $\caU_{t_1,t_2}^\bullet(s)$ and the integer $\epsilon^{\bullet}_{t_1}(t_2)$ can be regarded as "perturbation data" given by the choice of $t_1$ and $t_2$.

The explicit formulae in~Theorem~\hyperref[thmB]{B} and the fact that $\epsilon_{t_1}(t_1)=0$, for every $t_1\in T$, yield the following:
\begin{corC}\label{corC}
    Under the hypotheses of Theorem~\hyperref[thmB]{B}, let $t\in T$. Then the poles (resp.~zeros) of $\zeta_{G,G_t,G_t}^\bullet(s)$ are all those $s\in \C$ such that~$1$ is an eigenvalue of $\caE^\bullet(s)$ (resp.~$\caE^\bullet(s)-\caU_{t,t}^\bullet(s)$) but not an eigenvalue of $\caE^\bullet(s)-\caU_{t,t}^\bullet(s)$ (resp.~$\caE^\bullet(s)$).
\end{corC}
After Corollary~\hyperref[corC]{C}, the following question arises.
\begin{quesD}\label{quesD}
    Under the hypotheses of Theorem~\hyperref[thmB]{B}, let~$t\in T$. For which $s\in \C$ do the matrices $\caE^\bullet(s)$ and $\caE^\bullet(s)-\caU_{t,t}^\bullet(s)$ have $1$ as an eigenvalue? Provided $\zeta_{G,G_t,G_t}^\bullet(s)$ is an infinite series, what is its abscissa of convergence (that is, the maximal $r\in \R$ such that $\zeta_{G,G_t,G_t}^\bullet(s)$ has a pole at $s=r$)?
\end{quesD}

\subsection*{The behaviour at $s=-1$ and the Euler--Poincaré characteristic}\label{s=-1Eu}
One of the main goals of the paper is the study of the local behaviour at $s=-1$ of the relevant double-coset zeta functions. Unless it is a finite sum, the Dirichlet series in \eqref{eq:DCzeta} does not converge at $s=-1$. This underlines the importance of having a continuation of $\zeta_{G,H,K}(s)$ at least to $s=-1$ (which in our context is provided by Theorem~\hyperref[thmB]{B}) to carry out such an evaluation.

Addressing~\cite[Question~G(b)]{ccw:zeta}, we prove that the Euler--Poincaré identity is also attained in our framework. 
A crucial step towards this goal is to reduce the evaluation at $s=-1$ to the more accessible weakly locally $\infty$-transitive case (cf.~Lemma~\ref{lem:noUD}). More specifically, we can only focus on $\caZ_{\Gamma,u\to u}(-1)^{-1}$ for a finite connected edge-weighted graph $(\Gamma,\omega)$.
After introducing suitable notions of unimodularity and Euler--Poincaré characteristic $\chi(\Gamma,u)$ at $u\in \Gamma$ on $(\Gamma,\omega)$ (cf.~Definition~\ref{defn:chiGamma}), we deduce the following.
\begin{thmE}\label{thmE}
    Let $\Gamma$ be a finite connected non-empty graph with no cycles of length $\geq 2$, and let $\omega\colon E\Gamma\to \Z_{\geq 2}$ be an edge weight satisfying $\omega(a)\geq 3$ or $\omega(\bar{a})\geq 3$, for every $a\in E\Gamma$. If~$(\Gamma,\omega)$ is unimodular, then
    \begin{equation*}
       \chi(\Gamma,u)=\caZ_{\Gamma,u\to u}(-1)^{-1}, \quad\forall\,u\in \Gamma.
    \end{equation*}
\end{thmE}
In Theorem~\hyperref[thmE]{E}, the hypothesis on $\omega$ guarantees that a formula analogous to the one in Theorem~\hyperref[thmB]{B} is applicable to $\caZ_{\Gamma,u\to u}(s)$, for all $u\in \Gamma$.
The proof of Theorem~\hyperref[thmE]{E} is based on some splitting formulae satisfied by $\caZ_{\Gamma,u\to u}(s)^{-1}$, where~$u\in \Gamma$, that we discuss in Section~\ref{s:split}. The assumption that $\Gamma$ has no cycles of length $\geq 2$ guarantees their general applicability. 
By~Theorem~\hyperref[thmE]{E} and Lemma~\ref{lem:noUD}, we deduce the following.
\begin{corF}\label{corF}
    Let $G$ be a unimodular t.d.l.c.~group acting on a locally finite tree $T$ as in Convention~\ref{conv:intro} and with compact open vertex stabilisers. Assume that the quotient graph~$\Gamma$ and the associated standard edge-weight~$\omega\colon E\Gamma\to \Z_{\geq 2}$ (cf.~Section~\ref{sus:Gammaw}) satisfy the hypotheses of Theorem~\hyperref[thmE]{E}. Suppose also that $(G,T)$ is weakly locally $\infty$-transitive or (P)-closed. Then, for every $t\in T$ such that $(G,G_{t}, G_{t})$ has polynomial double-coset growth, we have
    \begin{equation*}
        \widetilde{\chi}_G=\zeta_{G,G_t,G_t}(-1)^{-1}\mu_{G_t}.
    \end{equation*}
\end{corF}

\subsection*{The behaviour at $s=-1$ and the weighted Ihara zeta function}\label{s=-1Ih}
The zeta function $\caZ_{\Gamma,u\to u}(s)$  is not the only growth series that has been considered in a graph-theoretical setting.
Within Dirichlet series, we mention the zeta functions associated to graphs and hypergraphs introduced by T.~Rossmann and C.~Voll~\cite{rovo}, used by them to study the enumeration of conjugacy classes of unipotent groups over finite fields. Within growth series, a well-known power series associated to a graph is its \emph{Ihara zeta function}~\cite{ihara}, which relates closed paths in the graph with the spectrum of its adjacency matrix.
In the present note, we focus on a variation of the Ihara zeta function, namely the \emph{weighted Ihara zeta function}. This is a power series~$Z_{(\Gamma,W)}(x)$ introduced by A.~Deitmar~\cite{deit:prim} 
for every graph $\Gamma$ with transition weight $W$ (cf.~Section~\ref{sus:ih-1} for a brief recap). In~Section~\ref{sus:ih-1}, we provide a canonical way to construct a transition weight $W_{(\Gamma,\omega)}$ on a finite graph $\Gamma$ starting from an edge weight $\omega$ (cf.~Example~\ref{ex:ih}). 

The main result of the section is the following.
\begin{thmG}[\protect{cf.~Theorem~\ref{thm:iha}}]\label{thmG}
Let~$\Gamma$ be a finite connected non-empty graph, and~$\Gamma_1, \Gamma_2$ be subgraphs satisfying~$\Gamma=\Gamma_1\cup \Gamma_2$, such that~$\Gamma_1\cap \Gamma_2$ is a $1$-segment graph, i.e., a Serre graph with two distinct vertices and an edge couple~$\{a,\bar{a}\}$ connecting them, such that $t(a)$ and $o(a)$ are leaves of~$\Gamma_1$ and~$\Gamma_2$, respectively.
Let also $\omega\colon E\Gamma\to \Z_{\geq 2}$ satisfy $\omega(a)\geq 3$ or $\omega(\bar{a})\geq 3$ for every $a\in E\Gamma$. Then,
    \begin{equation}\label{eq:ih1}
   \frac{\caZ_{\Gamma, a\to a}(-1)}{\caZ_{\Gamma_1, a\to a}(-1)\cdot \caZ_{\Gamma_2, a\to a}(-1)}=\frac{1}{\omega(a)\omega(\bar{a})}\cdot \frac{Z_{(\Gamma,W)}(1)}{Z_{(\Gamma_1,W_1)}(1)\cdot Z_{(\Gamma_2, W_2)}(1)}.
\end{equation}
In particular, if $(\Gamma,\omega)$ is unimodular then
\begin{equation}\label{eq:ih2}
    \frac{\chi(\Gamma_1,a)\cdot \chi(\Gamma_2,a)}{\chi(\Gamma,a)}=\frac{1}{\omega(a)\omega(\bar{a})}\cdot\frac{Z_{(\Gamma,W)}(1)}{Z_{(\Gamma_1,W_1)}(1)\cdot Z_{(\Gamma_2, W_2)}(1)}.
\end{equation} 
\end{thmG}
Relating Dirichlet series and power series, even if generated by a common sequence, is usually a non-trivial problem. In our case, the relation established in~Theorem~\hyperref[thmG]{G} has been facilitated by explicit determinant formulae for each of the two types of zeta functions involved. Theorem~\hyperref[thmG]{G} sheds light on the behaviour of the weighted Ihara zeta function $Z_{(\Gamma,W)}(t)$ at $t=1$ which, to the best of the author's knowledge, has been never studied so far. In the same spirit as in Theorem~\hyperref[thmE]{E}, it would be interesting to find an explicit expression of $Z_{(\Gamma,W)}(1)$ in terms of suitable invariants of~$(\Gamma,W)$.

\subsection*{Structure of the paper}\label{struct}
Sections~\ref{s:prel}~and~\ref{s:actions} collect background knowledge for the paper. 
In particular, in Section~\ref{sus:wlit} we introduce the new concept of weakly locally $\infty$-transitive actions on trees. 

In Section~\ref{s:suborb}, we give geometric descriptions of the coset spaces $G/G_t$ and the size $|G_{r}gG_t/G_t|$ of a group $G$ acting on a tree $T$ with respect to vertex or edge stabilisers $G_r,G_t$. This description is furthermore refined if the action is weakly locally $\infty$-transitive (cf.~Section~\ref{sus:subWLIT}) or (P)-closed (cf.~Section~\ref{sus:subP}). 

Section~\ref{s:DCzeta} follows a similar pattern. It begins with general results on the double-coset property and polynomial double-coset growth for groups acting on trees (cf.~Proposition~\ref{prop:conv*}). 
Afterwards, we specialise the discussion to the cases of weakly locally $\infty$-transitive or (P)-closed actions (cf.~Sections~\ref{sus:DCwlit}~and~\ref{sus:DCPcl}, respectively). In these two cases, we characterise the polynomial double-coset growth and give explicit determinant formulae of the relevant double-coset zeta functions. 
 
Section~\ref{s:split} collects some splitting formulae for $\caZ_{\Gamma,u\to u}(s)^{-1}$ which are key for the proofs in the next Section~\ref{s:euler}. 

Finally, in Section~\ref{s:euler} we discuss the behaviour at $s=-1$ for the double-coset zeta functions studied in Sections~\ref{sus:DCwlit}~and~\ref{sus:DCPcl}. In particular, we provide connections with the Euler--Poincaré characteristic of the group (cf.~Section~\ref{sus:eul-1}) and the Ihara zeta function of a weighted graph (cf.~Section~\ref{sus:ih-1}).

\subsection*{Acknowledgments}
The author thanks her Ph.D.~advisors Christopher Voll and Thomas Weigel for continuous support and instructive input throughout the preparation of this paper. She also thanks Ilaria Castellano, Jan Moritz Petschick, José Pedro Quintanilha, and Karthika Rajeev for interesting conversations and several useful comments on an early draught of the article, and Evgenii Khukhro and Roman Sauer for interesting questions that inspired the results in Section~\ref{s:split}. Finally, she is grateful to the two anonymous referees for valuable and extremely instructive suggestions that improved the article. 

The author is a postdoctoral researcher of the Fonds de la Recherche Scientifique -- FNRS. 
The author was supported by DFG TRR 358 –- Project-ID 491392403, and RTG 2229 "Asymptotic Invariants and Limits of Groups and Spaces". She is also a member of GNSAGA, as a part of Istituto Nazionale di Alta Matematica (INdAM).

\section{Notation and preliminaries}\label{s:prel}

\subsection{Generalities}\label{sus:gen} 
If there is no ambiguity, we denote a $1$-point set~$\{x\}$ by the element $x$ itself.
Given a set~$X$, for every subset $A\subseteq X$ denote by~$\mathbbm{1}_A$ the indicator function of $A$. 
For every set $X$, the symmetric group $\mathrm{Sym}(X)$ is always regarded as a topological group with the \emph{permutation topology}, i.e., the topology generated by the local basis at~$1$ given by all possible pointwise stabilisers of finite non-empty subsets of $X$. 

Moreover, let $G$ and $H$ be groups acting on the sets $X$ and $Y$, respectively. Denote by $\sigma_G\colon G\to \Sym(X)$ and $\sigma_H\colon H\to \Sym(Y)$ the homomorphisms induced by the two actions. The $G$-action on $X$ is said to be \emph{permutational isomorphic} to the $H$-action on $Y$ if there are a group isomorphism $\varphi\colon\sigma_G(G)\to \sigma_H(H)$ and a bijection $f\colon X\to Y$ satisfying $f(g\cdot x)=\varphi(g)\cdot f(x)$, for all $x\in X$ and $g\in \sigma_G(G)$.

\subsection{Graphs}\label{sus:graph}
A \emph{graph} (in the sense of J-P.~Serre~\cite{ser:trees}) consists of a set $\Gamma=V\Gamma\sqcup E\Gamma$ partitioned into two subsets $V\Gamma$ and $E\Gamma$ (called the \emph{set of vertices} and the \emph{set of edges} of $\Gamma$, respectively), together with two maps $o,t\colon E\Gamma\to V\Gamma$ (called \emph{origin} and \emph{terminus} maps, respectively) and an involution $\overline{\cdot}\colon E\Gamma\to E\Gamma$ (called \emph{edge inversion}) satisfying $\bar{e}\neq e$ and $o(\bar{e})=t(e)$, for every $e\in E\Gamma$. 
We introduce the following notation for a graph $\Gamma$.
\begin{notat}\label{notat:capitG}
   Given $u\in \Gamma$, we use the associated capital letter $U$ to denote the set $\{u\}$ if $u\in V\Gamma$, and the set $\{u,\bar{u}\}$ if $u\in E\Gamma$.
\end{notat}
An \emph{orientation} in a graph~$\Gamma$ is a set $E\Gamma^+\subseteq E\Gamma$ satisfying $|\{e,\bar{e}\}\cap E\Gamma^+|=1$ for every $e\in E\Gamma$, and $E\Gamma=E\Gamma^+\cup\{\bar{e}\mid e\in E\Gamma^+\}$. A graph $\Gamma$ is \emph{non-empty} if it has at least one vertex, and is \emph{locally finite} if $|o^{-1}(c)|<\infty$ for all $c\in V\Gamma$.

A \emph{subgraph} of a graph $\Gamma$ is a graph~$\Lambda=V\Lambda\sqcup E\Lambda$ with $V\Lambda\subseteq V\Gamma$, $E\Lambda\subseteq E\Gamma$, and whose origin, terminus and inversion maps are obtained from the corresponding ones of~$\Gamma$ by restriction.
A subgraph~$\Lambda$ of~$\Gamma$ is \emph{proper} if~$\Lambda\neq \Gamma$. 

A vertex $v$ in~$\Gamma$ is said to be \emph{terminal} (or a \emph{leaf}) if $|o^{-1}(v)|=1$.
An edge $e$ in~$\Gamma$ with $o(e)=t(e)$ is called \emph{$1$-loop}. A \emph{$n$-bouquet of loops (based at $c$)} is a graph with one vertex $c$ and edge-set $\{a_i,\bar{a}_i\mid 1\leq i\leq n\}$, where each $a_i$ is a $1$-loop starting at $c$.
A \emph{$1$-segment} is a graph $\Gamma$ with two distinct vertices and an edge-couple $\{e,\bar{e}\}$ connecting them.

Given two graphs $\Gamma$ and $\Lambda$, a \emph{graph morphism} is a map $\varphi\colon \Gamma\to \Lambda$ satisfying $\varphi(V\Gamma)\subseteq V\Lambda$, $\varphi(E\Gamma)\subseteq E\Lambda$, $\varphi(o(e))=o(\varphi(e))$ and $\varphi(\bar{e})=\overline{\varphi(e)}$ for every $e\in E\Gamma$. A \emph{graph monomorphism} (resp.~\emph{epimorphism}, \emph{isomorphism}) is a graph morphism which is injective (resp.~surjective, bijective). 
Given a graph~$\Gamma$, let $\Aut(\Gamma)$ be the group of all automorphisms (= self-isomorphisms) of $\Gamma$. We always regard $\Aut(\Gamma)$ as a topological group with the subspace topology induced by $\mathrm{Sym}(\Gamma)$. 

Let $\Gamma$ be a graph. A \emph{path} in~$\Gamma$ is a sequence of vertices and edges $\frp=(v_0,e_1,v_1,\ldots, e_n,v_n)$, $n\geq 0$, with $o(e_i)=v_{i-1}$ and $t(e_i)=v_i$ for every $1\leq i\leq n$. 
Note that paths may contain repeated vertices and edges, and it might happen that $e_{i+1}=\bar{e}_i$ for some~$i$.
We say that $\frp$ starts at $v_0$ (or at $e_1$) and ends at $v_n$ (or at $e_n$), has \emph{reverse path} $\overline{\frp}=(v_n,\bar{e}_n,v_{n-1},\ldots, \bar{e}_1,v_0)$, and length $n$ (written $\ell(\frp)=n$).
If $n\geq 1$, we may without ambiguity specify only the sequence of edges of $\frp$.
If $n=0$, the $1$-term sequence $\caO_{v_0}=(v_0)$ is called the \emph{trivial path at $v_0$}. 
Denote by $\calP_\Gamma$ the set of all paths in~$\Gamma$. Given non-empty subsets $X,Y\subseteq \Gamma$, let $\calP_\Gamma(X\to Y)$ be the set of all paths in~$\Gamma$ starting at some~$x\in X$ and ending at some~$y\in Y$.  
The \emph{product} of two paths $\frp=(e_1,\ldots, e_m)$ and $\frq=(f_1,\ldots,f_n)$ is defined only if $t(e_m)=o(f_1)$ and it is the path $\frp\cdot\frq=(e_1,\ldots,e_m,f_1,\ldots,f_n)$. If $\frp$ is a path starting and ending at the same vertex, denote by $\frp^d$ the $d$-th power of $\frp$ with respect to the product defined before.
A path $\frp$ is \emph{reduced} if either $\ell(\frp)=0$ or $\frp=(e_1,\ldots,e_n)$ and $e_{i+1}\neq \bar{e}_i$ for every $1\leq i\leq n-1$. 
For $n\geq 1$, an $n$-\emph{cycle} is a reduced path $\frp=(e_1, \ldots, e_n)$ with $o(e_1)=t(e_n)$ and $t(e_i)\neq t(e_{j})$ for all $1\leq i,j\leq n$~with~$i\neq j$.  

A graph $\Gamma$ is \emph{connected} if for all $v,w\in V\Gamma$ there is a path from~$v$ to~$w$. A subgraph of $\Gamma$ is a \emph{connected component} if it is a maximal connected subgraph of $\Gamma$. A graph is the disjoint union of all its connected components.
A \emph{tree}~$T$ is a connected graph with no $n$-cycles, for every $n\geq 1$. If~$T$ is a tree and $e\in ET$, then the graph $T\setminus\{e,\bar{e}\}$ has two connected components, $T_e^{+}\ni t(e)$ and $T_{e}^-\ni o(e)$. Set $T_{\geq e}:=T_e^+\sqcup \{e\}$ and $T_{\geq \bar{e}}:=T_e^-\sqcup \{\bar{e}\}$. 
A tree~$T$ is \emph{uniquely geodesic}, i.e., for all $v,w\in VT$ there is a unique reduced path $[v,w]$ from $v$ to $w$, which we call \emph{geodesic from $v$ to $w$}. Recall that $[v,w]$ is the path of minimal length in~$T$ from $v$ to $w$. 
Moreover, given $e,f\in ET$, there is a geodesic $(e_1,\ldots, e_n)$ in~$T$ with $e_1=e$ and $e_n=f$ if, and only if, $f\in T_{\geq e}$. In general, for $t_1,t_2\in T$ we denote by $[t_1,t_2]$ the geodesic from $t_1$
to $t_2$ in~$T$ (if it exists).
Moreover, for non-empty subsets $X,Y\subseteq T$, denote by $\Geod_T(X\to Y)$ the set of all geodesics in~$T$ from some $x\in X$ to some~$y\in Y$.

 A \emph{ray} in a tree is a sequence of edges $(e_i)_{i\in \Z_{\geq 1}}$ such that $o(e_i)\neq o(e_j)$ and $t(e_i)=o(e_{i+1})$, for all $i,j\in \Z_{\geq 1}$ with $i\neq j$. Two rays $(e_i)_{i\in \Z_{\geq 1}}$ and $(f_i)_{i\in \Z_{\geq 1}}$ are \emph{equivalent} if there is $j\in \Z_{\geq 1}$ such that $e_i=f_i$ for every $i\geq j$. Equivalence classes of rays in a tree are called \emph{ends}.

\begin{rem}\label{rem:Gnocyc}
    Let $\Gamma$ be a connected graph without $n$-cycles for every $n\geq 2$. Then $\Gamma$ has a unique maximal subtree: the subgraph $\Lambda$ obtained from $\Gamma$ by removing all its $1$-loops. In particular, for all $v,w\in V\Gamma$ the geodesic $[v,w]=(v=v_0,e_1,v_1,\ldots, e_n,v_n=w)$ in $\Lambda$ is the path of minimal length in $\Gamma$ from $v$ to $w$. Thus, if $v\neq w$ then $v_i\neq v_j$ for all $0\leq i,j\leq n$ with $i\neq j$.
\end{rem}

\section{Group actions on trees}\label{s:actions}
In this section, we introduce the two main classes of group actions on trees considered in this paper. The first is the class of (P)-closed actions on trees (cf.~Section~\ref{sus:Pcl}). In Section~\ref{sus:lads}, we briefly recall a local--to--global approach due C.~Reid and S.~Smith~\cite{resm:lad} to study these kinds of group actions, which will be largely used in the paper. We also add some new vocabulary (cf.~Definition~\ref{defn:pathD}) that will be exploited in Section~\ref{sus:subP} to count geodesics on the tree of those actions. Section~\ref{sus:wlit} introduces the second class we focus on, the one of weakly locally $\infty$-transitive actions on trees. Therein we discuss the connection with locally $\infty$-transitive actions (cf.~Lemma~\ref{lem:(W)LIT}), providing a local characterisation and some explicit examples (cf.~Proposition~\ref{prop:locWLIT} and~Examples~\ref{ex:WLIT}, \ref{ex:nonwlit} and~\ref{ex:nonPcl}).

\subsection{Group actions on trees}\label{sus:actions}
Let~$G$ be a topological group and~$T$ be a tree. A \emph{$G$-action $(G,T)$ on $T$} is a continuous group homomorphism~$G\to \Aut(T)$ (recall that~$\Aut(T)$ has the permutation topology, see Section~\ref{sus:graph}).
The action~$(G,T)$ is said to be \emph{without edge inversions} if~$g\cdot e\neq \bar{e}$ for all $g\in G$ and $e\in ET$.
\begin{conv}\label{conv:action}
    In this paper, all actions on trees are assumed to be \emph{without edge inversions}, and the relevant tree is assumed to have at least one edge.
\end{conv}
Convention~\ref{conv:action} allows us to have, as quotient of the action, a graph in the sense of J-P.~Serre that has at least one edge. These (relatively mild) hypotheses are important for the introduction of certain linear operators associated the quotient graph (cf.~Sections~\ref{sus:DCwlit} and~\ref{sus:DCPcl}), whose role is key in our study of double-coset zeta functions.

\smallskip

Given $t\in T$, denote by $G_t$ the stabiliser of $t$ in~$G$. More generally, for every subset $X\subseteq T$, let $G_X$ denote the pointwise stabiliser of~$X$. If $\frp=(e_i)_{1\leq i\leq n}$ is a path, $G_\frp$ denotes the pointwise stabiliser of the set $\{e_1,\ldots, e_n\}$. If~$\xi$ is an end of~$T$ (regarded as a set of equivalent rays), we denote by $G_\xi$ its setwise stabiliser.

An action $(G,T)$ is \emph{arc-transitive} if $ET=G\cdot e\sqcup G\cdot \bar{e}$ for some (and hence every) $e\in ET$. Moreover, $(G,T)$ is \emph{locally $\infty$-transitive} if, for all $v\in VT$ and $d\geq 0$, the stabiliser $G_v$ acts transitively on $\{\frp\in \Geod_T(v\to T) \mid \ell(\frp)=d\}$ (cf.~\cite[\S 0.2]{bumo}). One checks that locally $\infty$-transitive actions are arc-transitive. Examples of groups admitting a locally $\infty$-transitive action on a tree satisfying Convention~\ref{conv:action} are the $k$-points of simple simply connected algebraic $k$-groups of relative rank $1$, where $k$ is a non-Archimedean local field (cf.~\cite[pp.~91~and~95]{ser:trees}), and the Burger--Mozes universal groups $U(F)$ associated to $2$-transitive groups $F\leq \mathrm{Sym}(\{1,\ldots, d\})$ acting on the barycentric subdivision of a $d$-regular tree (cf.~the lines before~\cite[\S 3.1]{bumo}).

\subsection{The quotient graph and its standard edge weight}\label{sus:Gammaw}
Let $(G,T)$ be a group action on a tree without edge inversions. The \emph{quotient graph} $\Gamma=G\backslash T$ of the action is the graph with $V\Gamma:=G\backslash VT$, $E\Gamma:=G\backslash ET$ and, given $G\cdot e\in E\Gamma$, its origin is $G\cdot o(e)$, its terminus is $G\cdot t(e)$ and its inverse edge is $G\cdot\bar{e}$. One checks that these definitions are independent of the choice of $e$ in $G\cdot e$.

The assignment $\pi\colon t\in T\longmapsto G\cdot t\in \Gamma$ yields a graph epimorphism which is called \emph{quotient map} of $(G,T)$. The map $\pi$ entrywise extends to a map (denoted with the same symbol) from the set of all paths in~$T$ to the set of all paths in~$\Gamma$.

The \emph{standard edge weight} on $\Gamma$ is the map $\omega\colon E\Gamma\to \Z_{\geq 1}\cup\{\infty\}$ defined on every $a\in E\Gamma$ by choosing $v\in VT$ with $\pi(v)=o(a)$ and setting
\begin{equation}\label{eq:w}
    \omega(a):=|\{e\in ET\,:\,  o(e)=v\text{ and }\pi(e)=a\}|.
\end{equation}
In other words, $\omega(a)$ counts how many edges in~$T$ starting at $v$ lift $a$ via $\pi$.
It is straightforward to check that the assignment in~\eqref{eq:w} does not depend on the choice of the vertex $v$. 

Starting from a connected graph $\Gamma$ and a function $\omega\colon E\Gamma\to \Z_{\geq 1}$, one always has a group action on a tree without edge inversions having quotient graph~$\Gamma$ and standard edge weight~$\omega$ (cf.~Example~\ref{exintro}).

\subsection{(P)-closed actions on trees}\label{sus:Pcl}
The study of the group $\Aut_\pi(T)$ as in Example~\ref{exintro} initiated the study of (P)-closed group actions on trees (cf.~\cite{bew:simp}), a class that stems from the more general class of group actions on trees with the Tits' property~(P) (cf.~\cite[\S4.2]{tit:tree}), which we now recall.  

Let~$(G,T)$ be a group action on a tree. Let~$\ell=(v_0,e_1,v_1,\ldots, e_n, v_n, \ldots)$, $1\leq n\leq +\infty$, be a geodesic or a ray in~$T$ with vertices~$(v_n)_{n\geq 0}$ and edges~$(e_n)_{n\geq 1}$. Denote by~$\pi_\ell\colon VT\to \{v_n\mid n\geq 0\}$ the closest point projection of the vertices of~$T$ onto~$\ell$. For every~$n\geq 0$, the pointwise stabiliser~$G_\ell$ induces a permutation group~$(G_\ell)_{n}\leq \Sym(\pi_\ell^{-1}(v_n))$. Therefore, one obtains a group monomorphism
\begin{equation*}
    \Phi_\ell\colon G_\ell\longrightarrow \prod_{n\geq 0}(G_\ell)_n.
\end{equation*}
If~$\Phi_\ell$ is surjective for every~$\ell$, the action~$(G,T)$ is said to have \emph{Tits' property~(P)}. Note that, if~$\ell=(v_0,e,v_1)$ has only one edge~$e$, the function~$\pi_\ell$ maps  every~$v\in VT_{\geq \bar{e}}$ to~$v_0=o(e)$ and every~$v\in VT_{\geq e}$ to~$v_1=t(e)$.  
Hence, it is easily checked that the surjectivity of~$\Phi_\ell$ is equivalent to the requirement 
\begin{equation}\label{eq:propP}
     G_e=G_{T_{\geq\bar{e}}}\cdot G_{T_{\geq e}}.
\end{equation}

Following the characterisation of~\cite[Theorem~2.18]{resm:lad}, we say that the action~$(G,T)$ is \emph{(P)-closed} if it has Tits' property~(P) and~$G\leq \Aut(T)$ is closed with respect to the permutation topology.

The following fact is well-known and is included for completeness.
\begin{lem}\label{lem:propPbis}
  Let $(G,T)$ be a (P)-closed action on a tree and~$(e_1,\ldots, e_n)$ be a geodesic in~$T$ of length $n\geq 2$. Then, for every $k<n$, we have
  \begin{equation}\label{eq:propPbis}
      G_{(e_1,\ldots, e_k)}\cdot (e_{k+1},\ldots, e_n)=G_{e_k}\cdot (e_{k+1},\ldots, e_n). 
  \end{equation}
  \end{lem}
  \begin{proof}
  The inclusion $\subseteq$ is clear. Moreover, we note that $e_1,\ldots, e_{k-1}\in ET_{\geq \bar{e}_k}$ and $e_{k+1},\ldots, e_n\in ET_{\geq e_k}$. Hence, 
  \begin{equation*}
      \begin{split}
          G_{(e_1,\ldots, e_k)}\cdot (e_{k+1},\ldots, e_n) & \supseteq G_{T_{\geq \bar{e}_k}}\cdot (e_{k+1},\ldots, e_n)\\
          =G_{T_{\geq \bar{e}_k}}\cdot G_{T_{\geq e_k}}\cdot (e_{k+1},\ldots, e_n) & = G_{e_k}\cdot (e_{k+1},\ldots, e_n). \qedhere
      \end{split}
  \end{equation*}
\end{proof}

C.~Reid and S.~Smith~\cite{resm:lad} provide a parametrisation of (P)-closed group actions on trees in terms of \emph{local action diagrams}, that we now recall.

\subsection{Local action diagrams and their associated universal groups}\label{sus:lads}
Following~\cite[Definition~3.1]{resm:lad}, a \emph{local action diagram} is a triple $$\Delta=(\Gamma,(X_a)_{a\in E\Gamma},(G(c))_{c\in V\Gamma})$$ consisting of the following data:
\begin{itemize}
    \item[(i)] a connected graph $\Gamma$ in the sense\footnote{I.e., $\Gamma$ is a connected graph in the sense of J-P.~Serre, except that a loop may or may not be equal to its own reverse.} of~\cite[\S 2.1]{resm:lad};
    \item[(ii)] a family of non-empty pairwise disjoint sets $(X_a)_{a\in E\Gamma}$.
    For every $c\in V\Gamma$, set $X_c:=\bigsqcup_{a\in o^{-1}(c)}X_a$ and $X:=\bigsqcup_{a\in E\Gamma}X_a$;
    \item[(iii)] for every $c\in V\Gamma$, a closed subgroup $G(c)$ of $\mathrm{Sym}(X_c)$ whose orbits are given by $G(c)\backslash X_c=\{X_a\}_{a\in o^{-1}(c)}$.
\end{itemize}
A local action diagram~$\lad$ is \emph{trivial} if~$\Gamma$ is a $1$-point graph (i.e., a graph~$\Gamma$ with~$|V\Gamma|=1$ and~$E\Gamma=\emptyset$). 

\smallskip

In Theorem~\ref{thm:UD} we will recall a correspondence -- established by C.~Reid and S.~Smith in~\cite{resm:lad} -- between local action diagrams and (P)-closed actions on trees. In particular, trivial local action diagrams correspond to (P)-closed actions on $1$-point trees, and local action diagrams whose underlying graph has a $1$-loop equal to its own inverse correspond to (P)-closed actions on trees with edge inversions.
Hence, having assumed Convention~\ref{conv:action} for group actions on trees, we assume the following henceforth.
\begin{conv}\label{conv:lad}
   In this paper, all local action diagrams are assumed to be \emph{non-trivial} and their underlying graph is assumed to have \emph{no 1-loops equal to their own inverses} (that is, it is graph as defined in Section~\ref{sus:graph}).
\end{conv}

\begin{notat}\label{notat:capitD}
If there is no ambiguity, we write $\lad$ in place of $\Delta=(\Gamma,(X_a)_{a\in E\Gamma},(G(c))_{c\in V\Gamma})$. 
Moreover, given $u\in \Gamma$, set $X_U:=X_u$ if $u\in V\Gamma$ and $X_U:=X_u\sqcup X_{\bar{u}}$ if $u\in E\Gamma$ (cf.~Notation~\ref{notat:capitG}).
\end{notat}

Local action diagrams can be constructed from a group action on a tree $(G,T)$ without edge inversions as follows: Let $\pi$ be the quotient map on $(G,T)$ and choose a set of representatives $V^*$ of the $G$-orbits on $VT$. Following~\cite[Definition~3.6]{resm:lad}, the \emph{local action diagram associated to $(G,T)$ and $V^*$} is defined as follows:
\begin{itemize}
    \item[(i)] $\Gamma=G\backslash T$ is the quotient graph of $(G,T)$;
    \item[(ii)] for every $a\in E\Gamma$, let $v^*\in V^*$ be such that $\pi(v^*)=o(a)$ and define
    \begin{equation*}\label{eq:Xa}
        X_a:=\{e\in ET \mid o(e)=v^*\text{ and }\pi(e)=a\};
    \end{equation*}
    \item[(iii)] for every $c\in V\Gamma$ with representative $c^*\in V^*$, let $G(c)$ be the closure in~$\mathrm{Sym}(X_c)$ of the permutation group induced by $G_{c^*}$ acting on $X_c$.
\end{itemize}
Note that the standard edge weight on the quotient graph $\Gamma$ is given by $\omega(a)=|X_a|$, for every $a\in E\Gamma$. 
Up to isomorphism of local action diagrams (cf.~\cite[Definition~3.2]{resm:lad}), every group action on a tree~$(G,T)$ has a unique associated local action diagram (cf.~\cite[Lemma~3.7]{resm:lad}). Thus, we refer to \emph{the} local action diagram associated to $(G,T)$. 

Note that two actions might have the same local action diagram (cf.~Example~\ref{ex:lads}(ii)). One of the key results in~\cite{resm:lad} (recalled in Theorem~\ref{thm:UD}) shows that, up to natural notions of isomorphisms, there is a bijective correspondence between (P)-closed actions on trees and local action diagrams.

\begin{ex}\label{ex:lads}
   \begin{itemize}
       \item[(i)] Let $\Gamma$ be a connected graph with an edge weight $\omega\colon E\Gamma\to\Z_{\geq 1}$. Let $T$ and $G=\Aut_\pi(T)$ as in Example~\ref{exintro}. 
       Then the local action diagram $\lad$ associated to $(G,T)$ is given by taking, for every $a\in E\Gamma$, a set $X_a$ of cardinality $\omega(a)$, and by setting
       \begin{equation*}
        G(c):=\{\sigma \in \mathrm{Sym}(X_c)\mid\forall\,a\in o^{-1}(c),\,\sigma(X_a)=X_a\}, 
    \end{equation*}
    for every $c\in V\Gamma$.
        By design, $G(c)\backslash X_c=\{X_a\}_{a\in o^{-1}(c)}$. More precisely, for every $a\in o^{-1}(c)$, the $G(c)$-action on $X_a$ is permutational isomorphic to the action of $\mathrm{Sym}(X_a)$ on the same set.
        
        Note that $G(c)$ is closed in $\Sym(X_c)$. Indeed, if $\sigma\in \mathrm{Sym}(X_c)$ satisfies $\sigma(x)\not\in X_a$ for some $x\in X_a$, then $\sigma\cdot \mathrm{Sym}(X_c)_{x}$ is an open neighbourhood of $\sigma$ in~$\mathrm{Sym}(X_c)$ which is contained in~$\mathrm{Sym}(X_c)\setminus G(c)$. 
       \item[(ii)] Given a prime $p$, denote by $\Q_p$ the field of $p$-adic numbers, by $\Z_p$ the ring of $p$-adic integers, and by $\F_p\simeq\Z_p/{p\Z_p}$ the field of size~$p$. Here below, we report and rephrase some results from~\cite[\S II.1.4]{ser:trees}. 
       
       The action of $G=\SL_2(\Q_p)$ on its Bruhat--Tits tree $T$, which is a $(p+1)$-regular tree, has a $1$-segment as quotient graph $\Gamma$. Set $E\Gamma=\{a,\bar{a}\}$. Moreover, there is $e\in ET$ with $G\cdot e=a$ satisfying $G_{o(e)}=\SL_2(\Z_p)\simeq G_{t(e)}$. 
       Both the $\SL_2(\Z_p)$-action on $\{f\in ET\mid o(f)=o(e)\}$ and the $G_{t(e)}$-action on $\{f\in ET\mid o(f)=o(\bar{e})\}$ are permutational isomorphic to the faithful action of $\PSL_2(\F_p)$ on the projective line $\mathbb{P}^1(\F_p)$.
       Hence, the local action diagram $\Delta=(\Gamma,(X_b)_{b\in E\Gamma},(G(c))_{c\in V\Gamma})$ associated to $(G,T)$ is given by setting $X_a=X_{\bar{a}}=\mathbb{P}^1(\F_p)$ and $G(o(a))=G(t(a))=\PSL_2(\F_p)$.

       The same local action diagram can be obtained by taking~$G$ as the projective special linear group~$\PSL_2(\Q_p)$, and considering its induced action on~$T$. Since the kernel of the $\SL_2(\Q_p)$-action on~$T$ is the centre of~$\SL_2(\Q_p)$, the induced $\PSL_2(\Q_p)$-action on~$T$ is faithful.
\end{itemize}
\end{ex}

We now expand the vocabulary of local action diagrams, introducing some key tools for the discussion.
\begin{defn}\label{defn:pathD}
    Let $\lad$ be a local action diagram as in Convention~\ref{conv:lad}. For $n\geq 1$, an \emph{$n$-path} in~$\Delta$ is a sequence $\xi=(x_1, \ldots, x_n)$ obtained by starting with a path $\frp_\xi=(a_1,\ldots, a_n)$ in $\Gamma$ and selecting, for each $1\leq i\leq n$, an element $x_i\in X_{a_i}$. One says that $\xi$ starts at $x_1$, ends with $x_n$, has \emph{length} $n$ (written $\ell(\xi)=n$), and $\frp_{\xi}$ is called the \emph{underlying path} of $\xi$ in~$\Gamma$. The \emph{$0$-path} $O_c$ in~$\Delta$ at $c\in V\Gamma$ is the empty sequence of elements in~$X$ with the trivial path at $c$ as underlying path in~$\Gamma$. The path $O_c$ has length zero. A \emph{path} in~$\Delta$ is an $n$-path for some $n\geq 0$. Given paths $\xi=(x_1,\ldots, x_m)$ and $\eta=(y_1,\ldots, y_n)$ in~$\Delta$ with $\frp_\xi=(a_1,\ldots, a_m)$ and $\frp_\eta=(b_1,\ldots, b_n)$, the product $\xi\cdot \eta$ is defined only if $t(a_m)=o(b_1)$ and it is the path $(x_1,\ldots, x_{m}, y_1,\ldots, y_n)$ in~$\Delta$ with underlying path $\frp_\xi\cdot \frp_\eta$. Put also $\xi\cdot O_{t(a_m)}=\xi$ and $O_{o(b_1)}\cdot \eta=\eta$. 
    Given a path $\xi$ and a non-empty set of paths $\euE$ in~$\Delta$ such that the product $\xi\cdot \eta$ is defined for every $\eta\in \euE$, let $\xi\cdot \euE=\{\xi\cdot\eta\mid \eta\in\euE\}$. If $\xi=(x)$ has length $1$, we write $x\cdot \euE$ in place of $(x)\cdot \euE$.
    
    A map $\iota\colon X\to X$ is said to be an \emph{inversion} in~$\Delta$ if $\iota(X_a)\subseteq X_{\bar{a}}$ for every $a\in E\Gamma$. 
    A path $\xi$ in~$\Delta$ is \emph{reduced} in~$(\Delta, \iota)$ if either $\ell(\xi)=0$ or $\xi=(x_1, \dots, x_n)$, for some $n\geq 1$, and $x_{i+1}\neq\iota(x_i)$ for every $1\leq i\leq n-1$. Even if $\xi$ is reduced, the underlying path $\frp_\xi$ needs not to be reduced.
Denote by $\calP_{(\Delta, \iota)}$ the set of all reduced paths in~$(\Delta, \iota)$.
For non-empty subsets $X_1, X_2\subseteq X$, let also $\calP_{(\Delta, \iota)}(X_1\to X_2)$ be the collection of all reduced paths in~$(\Delta, \iota)$ starting at some $x_1\in X_1$ and ending at some $x_2\in X_2$.  
\end{defn}
An inversion in a local action diagram is not required to be an involution. In fact, the sizes of $X_a$ and $X_{\bar{a}}$ might differ. The term "inversion" here refers to the edge inversion on the labels in the partition $X=\bigsqcup_{a\in E\Gamma}X_a$.

\smallskip

Let $\lad$ be a local action diagram. Following~\cite[Definition~3.4]{resm:lad}, a \emph{$\Delta$-tree} $(T,\pi,\caL)$ consists of a tree $T$, a graph epimorphism $\pi\colon T\to \Gamma$ and a map $\caL\colon ET\to X$ which restricts, for all $v\in VT$ and $a\in E\Gamma$ with $o(a)=\pi(v)$, to a bijection
  \begin{equation*}\label{eq:Lva}
    \caL_{v,a}\colon\{e\in ET \mid o(e)=v\text{ and } \pi(e)=a\} \longrightarrow X_a.
  \end{equation*}
In particular, for every $v\in VT$, the map $\caL$ restricts to a bijection 
$$\caL_v\colon o^{-1}(v)\longrightarrow X_{\pi(v)}.$$
Note that the definition of a $\Delta$-tree is independent of $(G(c))_{c\in V\Gamma}$.

According to~\cite[\S 3.5]{car}, the pair $(T,\pi)$ is the universal cover of the edge-weighted graph $(\Gamma,\omega)$, where $\omega(a)=|X_a|$ for every $a\in E\Gamma$ (here~$\omega(a)$ is regarded as a cardinal). Therefore, for every two $\Delta$-trees $(T,\pi,\caL)$ and $(T', \pi',\caL')$ there is a graph isomorphism $\varphi\colon T\to T'$ such that $\pi=\pi'\circ \varphi$ (cf.~\cite[Lemma~3.5]{resm:lad}). Moreover, for every $\Delta$-tree we may define $\Aut_\pi(T)$ as in Example~\ref{exintro}.

\subsubsection{The standard $\Delta$-tree associated to $\iota$ and $c_0$}\label{ssus:TD}
Let $\lad$ be a local action diagram. 
Following the proof of~\cite[Lemma~3.5]{resm:lad}, we recall the construction of an explicit family of $\Delta$-trees which plays an important role in the next discussion.

Set an inversion $\iota\colon X\to X$ in~$\Delta$ and $c_0\in V\Gamma$, and define a graph $T=T(\Delta, \iota, c_0)$ as follows: 
  The set of vertices of $T$ is
\begin{equation}\label{eq:VTD}
    VT:=\calP_{(\Delta,\iota)}(X_{c_0}\to X).
\end{equation}
The vertex $v_0=O_{c_0}$ is called the \emph{root of $T$}.
The edges of $T$ are the pairs $(v,w)$ and $(w,v)$ of reduced paths in~$\Delta$ of the form $v=(x_1, \dots, x_n)$ and $w=(x_1, \dots, x_n, x_{n+1})$, for some $n\geq 0$. 
Every edge $(v,w)$ of $T$ as before is said to be a \emph{positive edge}. Denote by $ET^+$ the set of all positive edges of $T$.
The origin, terminus, and inversion maps of $T$ are given by $o(v,w)=v$, $t(v,w)=w$ and $\overline{(v,w)}=(w,v)$, for every $(v,w)\in ET$. For every $(v,w)\in ET^+$ with $w=(x_1, \dots, x_{n+1})$, set $\caL(v,w)=x_{n+1}$ and $\caL(w,v)=\iota(x_{n+1})$. 

\begin{rem}\label{rem:Linv}
    Every $e\in ET(\Delta,\iota,c_0)^+$ satisfies $\caL(\bar{e})=\iota(\caL(e))$. This is generally not true if $e\in ET(\Delta,\iota,c_0)\setminus ET(\Delta,\iota,c_0)^+$, as a given $x\in X$ might differ from $\iota(\iota(x))$.
\end{rem}

More generally, for every path $\frp=(e_1, \dots, e_n)$ in~$T=T(\Delta,\iota,c_0)$ we define
\begin{equation}\label{eq:Lpaths}
    \caL(\frp):=(\caL(e_1), \ldots, \caL(e_n)).
\end{equation}
\begin{rem}\label{rem:TDtree}
By Remark~\ref{rem:Linv}, a path~$\frp=(e_1,\ldots, e_n)$ in~$T$~with~$e_1,\ldots, e_n\in ET^+$ is reduced if and only if $\caL(\frp)$ is reduced in~$(\Delta,\iota)$. 
Moreover, by~\eqref{eq:VTD}, for every $v=(x_1,\ldots, x_n)\in VT$ there is a unique reduced path from $v_0$ to $v$ in $T$, namely $\caO_{v_0}$ if $v=v_0$, and
$(v_0,e_1,v_1,\ldots, e_n,v_n=v)$ with $v_i=(x_1,\ldots, x_i)$ for every $1\leq i\leq n$ otherwise. In particular, $T$ is connected.

More precisely, $T$ is a tree. In fact, if $T$ admits a $n$-cycle $\gamma$ for some $n\geq 1$, there is a vertex $v$ of $\gamma$ such that the reduced path $\frp$ from $v_0$ to $v$ shares no edges with $\gamma$. Up to an edge-relabelling, we may assume that $\gamma=(e_1,\ldots, e_n)$ and $o(e_1)=v$. Then $\frp$ and $\frp\cdot \gamma$ are two distinct reduced paths in $T$ from $v_0$ to $v$, impossible.

Being $T$ a tree, the map $\caL$ in~\eqref{eq:Lpaths} restricts to a bijection
\begin{equation}\label{eq:Lpathsv0}
    \caL^{v_0}\colon\Geod_T(v_0\to T)\longrightarrow \calP_{(\Delta,\iota)}(X_{c_0}\to X).
\end{equation}
\end{rem}

\begin{rem}\label{rem:geoE+}
    By~\eqref{eq:Lpaths}, every edge lying in a geodesic from $v_0$ in $T$ belongs to $ET^+$. Indeed, if $(e_1,\ldots, e_n)\in \Geod_T(v_0\to T)$ and $\caL(e_1,\ldots, e_n)=(x_1,\ldots, x_n)$, then $e_1=(O_{c_0},x_1)$ and $e_i=\big((x_1,\ldots, x_{i-1}), (x_1,\ldots, x_i)\big)$ for every $2\leq i\leq n$.
    
    Moreover, for~$e\in ET^+$ and for every $f\in o^{-1}(t(e))\setminus\{\bar{e}\}$, the path $[v_0,e]\cdot f$ is a geodesic from $v_0$ and thus $f\in ET^+$.
\end{rem}

\smallskip

We define the graph epimorphism $\pi\colon T\to \Gamma$ by putting $\pi(O_{c_0})=c_0$ and, for~$v=(x_1, \ldots, x_n)\in VT$ ($n\geq 1$) having underlying path $(a_1, \ldots, a_n)$ in~$\Gamma$, by $\pi(v)=t(a_n)$. The triple $(T,\pi,\caL)$ is a $\Delta$-tree that we call the \emph{standard $\Delta$-tree associated to $\iota$ and $c_0$}. 

\smallskip

Since the map in~\eqref{eq:Lpathsv0} is a bijection, we deduce the following.
\begin{lem}\label{lem:geodTD}
    Put $T=T(\Delta,\iota,c_0)$ and consider $e\in ET^+$ with $\caL(e)=x$. Then the map $\caL$ in~\eqref{eq:Lpaths} restricts to the following bijection:
    \begin{equation}\label{eq:Lbije}
         \caL^e\colon \Geod_T(e\to T)\longrightarrow \calP_{(\Delta,\iota)}(x\to X).
    \end{equation}
    If in particular $o(e)=v_0$, the map $\caL$ in~\eqref{eq:Lpaths} restricts also to the following bijection:
    \begin{equation}
         \caL^{\bar{e}}\colon \Geod_T(\bar{e}\to T)\longrightarrow \iota(x)\cdot \calP_{(\Delta,\iota)}(X_{c_0}\setminus \{x\}\to X).
    \end{equation}
\end{lem}
\begin{proof}
Write $e=(v,w)$, where $w=(x_1,\ldots, x_m, x)\in \calP_{(\Delta,\iota)}(X_{c_0}\to X)$. Note that $\caL([v_0,v])=(x_1,\ldots, x_m)$.
The bijection $\caL^{v_0}$ in~\eqref{eq:Lpathsv0} restricts to a bijective map
\begin{equation*}
        [v_0,v]\cdot \Geod_T(e\to T)\longrightarrow (x_1,\ldots, x_m)\cdot \calP_{(\Delta,\iota)}(x\to X).
\end{equation*}
This implies that $\caL^e$ is a bijection. Moreover, we observe that
\begin{equation}\label{eq:Geodec}
     \Geod_T(\bar{e}\to T)=\bigsqcup_{f\in o^{-1}(v),\,f\neq e}\bar{e}\cdot \Geod_T(f\to T)
\end{equation}
and
\begin{equation}\label{eq:Pathdec}
    \iota(x)\cdot\calP_{(\Delta,\iota)}(X_{\pi(v)}\setminus\{x\}\to X)=\bigsqcup_{y\in X_{\pi(v)},\,y\neq x}\iota(x)\cdot \calP_{(\Delta,\iota)}(y\to X).
\end{equation}
If $v=v_0$, then $o^{-1}(v)\subseteq ET^+$. From~\eqref{eq:Geodec}, \eqref{eq:Pathdec} and the first part of the statement, we conclude that $\caL^{\bar{e}}$ is bijective.
\end{proof}

\begin{rem}\label{rem:geodEbar}
    For every $e=(v,w)\in ET^+$, one may restrict the map $\caL$ in~\eqref{eq:Lpaths} to a bijection $\caL^{\bar{e}}$ from $\Geod_T(\bar{e}\to T)$ to a suitable set of paths in $(\Delta,\iota)$. One may proceed inductively on $\ell([v_0,v])=l\geq 0$. The case $l=0$ is done by Lemma~\ref{lem:geodTD}. 
    If $l\geq 1$, one assumes the claim true for $l-1$ and observes that $o^{-1}(v)\setminus\{e\}$ has exactly one edge which does not belong to $ET^+$: it is the edge whose reverse $f_0$ is the last edge of $[v_0,v]$. By~Lemma~\ref{lem:geodTD}, for all $f\in o^{-1}(v)\setminus\{e,\bar{f}_0\}$, the map $\caL^f$ as in~\eqref{eq:Lbije} is bijective. Moreover, $\ell([v_0,o(f_0)])=\ell([v_0,v])-1$ and by induction $\caL^{\bar{f}_0}$ is bijective. Using the decomposition in~\eqref{eq:Geodec}, one determines the image of $\Geod_T(\bar{e}\to T)$ via $\caL$ and thus the bijection $\caL^{\bar{e}}$.
\end{rem}

\subsubsection{The universal group $U(\Delta,\mathbb{T})$}\label{ssus:UD}
Let $\Delta$ be a local action diagram with a $\Delta$-tree $\mathbb{T}=(T,\pi,\caL)$. For all $g\in \Aut_\pi(T)$ and $v\in VT$, there is an induced permutation $\sigma(g,v)\colon X_{\pi(v)}\to X_{\pi(v)}$ given by
\begin{equation*}
    \sigma(g,v)(x):=\Big(\caL_{g\cdot v}\circ g\circ(\caL_{v})^{-1}\Big)(x).
\end{equation*}
\begin{defn}[\protect{\cite[Definition~3.8]{resm:lad}}]\label{defn:UD}
  The \emph{universal group} associated to $\Delta$ and $\mathbb{T}$ is 
\begin{equation*}\label{eq:UD}
    U(\Delta, \mathbb{T}):=\{g\in \Aut_\pi(T)\mid\forall\,v\in VT,\,\sigma(g,v)\in G(\pi(v))\}.
\end{equation*}  
\end{defn}
In other words, $U(\Delta,\mathbb{T})$ collects all elements of $\Aut_\pi(T)$ acting on~$o^{-1}(v)$ as a permutation in~$G(\pi(v))$, for every $v\in VT$. 

If $\mathbb{T}$ is the standard $\Delta$-tree associated to $\iota$ and $c_0$ (cf.~Section~\ref{ssus:TD}), we write $U(\Delta,\iota,c_0)$ instead of~$U(\Delta,\mathbb{T})$.

If $T$ is locally finite, the group $U(\Delta,\mathbb{T})$ is a t.d.l.c.~group with respect to the subspace topology from $\Aut(T)$. Indeed, $U(\Delta,\mathbbm{T})$ is a closed subgroup of the t.d.l.c.~group $\Aut(T)$ (cf.~\cite[\S 6]{resm:lad}). If additionally $\Gamma$ is finite, then $U(\Delta,\mathbbm{T})$ is a compactly generated t.d.l.c.~group (cf.~\cite[Proposition~6.5]{resm:lad}).

The group $U(\Delta,\mathbb{T})$ simultaneously generalises the three following notable examples.
\begin{ex}\label{ex:UD}
Let $\lad$ be a local action diagram and $\mathbb{T}=(T,\pi,\caL)$ a $\Delta$-tree.
\begin{itemize}
       \item[(i)] For every $c\in V\Gamma$, let 
       $$G(c):=\{\sigma \in \mathrm{Sym}(X_c)\mid\forall\,a\in o^{-1}(c),\,\sigma(X_a)=X_a\}.$$ 
       Then $U(\Delta,\mathbb{T})=\Aut_\pi(T)$.
       \item[(ii)] Let $\Gamma$ be a $1$-segment with $E\Gamma=\{a,\bar{a}\}$, $c=o(a)$ and $d=t(a)$. Let $G(c)$ and $G(d)$ act transitively on $X_c$ and $X_d$, respectively.
       \begin{itemize}
           \item[(iia)] Set $F:=G(c)$ and assume that $G(d)=C_2$, $|X_c|=k\geq 2$ and $|X_d|=2$. Adapting~\cite[Example~11]{resm:motiv} to actions on trees without edge inversion, the action $(U(\Delta,\mathbb{T}),T)$ is permutational isomorphic to the action of the Burger--Mozes universal group $U(F)$ on the barycentric subdivision $T_k'$ of the $k$-regular tree $T_k$ (cf.~\cite[\S3.2]{bumo}). Here we consider $T_k'$, and not $T_k$, because $U(F)$ acts vertex-transitively (and thus with edge inversions) on $T_k$.
           \item[(iib)] Set $F_1:=G(c)$ and $F_2:=G(d)$. Following~\cite[Example~12]{resm:motiv}, the action $U(\Delta,\mathbb{T})$-action on $T$ is permutational isomorphic to $(U(F_1,F_2),T)$, where $U(F_1,F_2)$ is the group introduced by S.~Smith in~\cite{sm:UFF}.
           \end{itemize}
   \end{itemize}
\end{ex}

The following theorem collects some key results of the work of C.~Reid and S.~Smith~\cite{resm:lad}. It motivates why, throughout the paper, we focus on actions of the form $(U(\Delta,\iota,c_0),T(\Delta,\iota,c_0))$ while considering (P)-closed actions on trees with associated local action diagram $\Delta$.
Note that the result below holds for general group actions on trees and general local action diagrams (even if they do not satisfy Conventions~\ref{conv:action} and~\ref{conv:lad}, respectively).
\begin{thm}\label{thm:UD}
\begin{itemize}
    \item[(i)] \emph{(\cite[Lemma~3.5, Theorem~3.12]{resm:lad})} Let $\Delta$ be a local action diagram. For all $\Delta$-trees $\mathbb{T}=(T,\pi,\caL)$ and $\mathbb{T}'=(T',\pi',\caL')$ and $v\in VT$, $v'\in VT'$ with $\pi(v)=\pi'(v')$, there is a graph isomorphism $\phi\colon T\to T'$ such that $\phi(v)=v'$, $\pi=\pi'\circ \phi$ and $\phi U(\Delta,\mathbbm{T})\phi^{-1}=U(\Delta,\mathbb{T}')$. 
    \item[(ii)] \emph{(\cite[Theorem~3.9]{resm:lad})} Let $\Delta$ be a local action diagram and $\mathbb{T}=(T,\pi,\caL)$ be a $\Delta$-tree. Then the local action diagram associated to $(U(\Delta,\mathbb{T}),T)$ is isomorphic to $\Delta$ (in the sense of~\cite[Definition~3.2]{resm:lad}).
    \item[(iii)] \emph{(\cite[Theorem~3.10]{resm:lad})} Let $(G,T)$ be a (P)-closed action on a tree with associated local action diagram~$\Delta$. Then~$(G,T)$ is permutational isomorphic to~$(U(\Delta,\iota,c_0),T(\Delta,\iota,c_0))$, for every inversion $\iota$ in~$\Delta$ and every $c_0\in V\Gamma$.
\end{itemize}
\end{thm}

\subsection{Weakly locally $\infty$-transitive actions on trees}\label{sus:wlit}
Recall that group actions on trees satisfy Convention~\ref{conv:action} and local action diagrams satisfy Convention~\ref{conv:lad}.

\begin{defn}\label{defn:WLIT}
    Let $(G,T)$ be a group action on a tree with quotient map $\pi\colon T\to \Gamma=G\backslash T$. Then $(G,T)$ is said to be \emph{weakly locally $\infty$-transitive at $v\in VT$} if, for every path $\frp$ in~$\Gamma$ starting at $\pi(v)$, the stabiliser $G_v$ acts transitively on the set
    \begin{equation}\label{eq:geop}
      \{\tilde{\frp}\in \Geod_T(v\to T)\mid \pi(\tilde{\frp})=\frp\}.
    \end{equation}
   Moreover, $(G,T)$ is said to be \emph{weakly locally $\infty$-transitive} if it is weakly locally $\infty$-transitive at every $v\in VT$.
\end{defn}

\begin{rem}\label{rem:WLIT} 
    \begin{itemize}
        \item[(i)] The action $(G,T)$ is weakly locally $\infty$-transitive at $v$ if, and only if, it is weakly locally $\infty$-transitive at every $u\in G\cdot v$.
        \item[(ii)] Let $(G,T)$ be weakly locally $\infty$-transitive at $v$, and consider $\tilde{\frp}=(e_1, \ldots, e_n)\in \Geod_T(v\to T)$ with $\pi(\tilde{\frp})=\frp$. Then, for every $i<n$ the group $G_{(e_1, \ldots, e_i)}$ acts transitively on 
        $$\{\widetilde{\frq}=(f_1,\ldots,f_n)\in\Geod_T(v\to T)\mid \pi(\widetilde{\frq})=\frp\text{ and }\forall\,j\leq i, f_j=e_j\}.$$
    \end{itemize}
\end{rem}

As the name suggests, weakly locally $\infty$-transitive actions generalise locally $\infty$-transitive ones as follows.

\begin{lem}\label{lem:(W)LIT}
  Let $(G,T)$ be a group action on a tree with quotient graph~$\Gamma$. Then $(G,T)$ is locally $\infty$-transitive if, and only if, it is weakly locally $\infty$-transitive and $\Gamma$ is a $1$-segment.
\end{lem}

\begin{proof}
    Let $\pi\colon T\to \Gamma=G\backslash T$ be the quotient map of $(G,T)$. Given $v\in VT$ with $\pi(v)=c$ and for every $d\geq 0$, we observe that
    \begin{equation}\label{eq:setWLIT}
       \{\tilde{\frp}\in \Geod_T(v\to T)\mid\ell(\tilde{\frp})=d\}=\hspace{-0.2cm}\bigsqcup_{\frp\in\calP_{\Gamma}(c\to \Gamma),\atop \ell(\frp)=d}\hspace{-0.2cm}\{\tilde{\frp}\in \Geod_T(v\to T)\mid \pi(\tilde{\frp})=\frp\}.
    \end{equation}
    Note that $\Gamma$ is a $1$-segment if, and only if, $|\{\frp\in\calP_\Gamma(c\to \Gamma)\,:\,\ell(\frp)=d\}|=1$ for all $d\geq 0$ and $c\in V\Gamma$. 
    
    If $(G,T)$ is weakly locally $\infty$-transitive and $\Gamma$ is a $1$-segment, then $G_v$ acts transitively on $\{\tilde{\frp}\in \Geod_T(v\to T)\mid\ell(\tilde{\frp})=d\}$, for all $v\in VT$ and $d\geq 0$, and $(G,T)$ is locally $\infty$-transitive. 
    
    Conversely, if $(G,T)$ is locally $\infty$-transitive, then $(G,T)$ is arc-transitive. More precisely, since $G_v$ acts transitively on $o^{-1}(v)$ for every $v\in VT$, we deduce that $\Gamma$ is a $1$-segment. Indeed, if $\Gamma$ is a $1$-loop, for every $v\in VT$ the group $G_v$ has two orbits on $o^{-1}(v)$. Finally, for all $v\in VT$ and $d\geq 0$, there is a unique path $\frp$ in~$\Gamma$ starting at $\pi(v)$ of length~$d$. By~\eqref{eq:setWLIT}, $(G,T)$ is also weakly locally $\infty$-transitive.
\end{proof}

The next lemma provides a local characterisation of weakly 
 locally $\infty$-transitive (P)-closed actions on trees. An analogous result has already been proved for Burger--Mozes universal groups~\cite[Lemma~3.1.1 and the lines before it]{bumo}.
\begin{prop}\label{prop:locWLIT}
    Let $(G,T)$ be a group action on a tree with associated local action diagram $\Delta=(\Gamma, (X_a),(G(c)))$, and let $v\in VT$ with $\pi(v)=c$. If $(G,T)$ is weakly locally $\infty$-transitive at $v$, then
    \begin{itemize}
        \item[($\diamondsuit$)] for all $a,b\in o^{-1}(c)$ and all~$x\in X_a$, the group $G(c)_x$ acts transitively on~$X_b\setminus \{x\}$.
    \end{itemize}
    Conversely, if $(G,T)$ is (P)-closed, having condition~$(\diamondsuit)$ for all~$c\in V\Gamma$ implies that $(G,T)$ is weakly locally $\infty$-transitive.
\end{prop}
In ($\diamondsuit$), note that $X_b\setminus\{x\}=X_b$ unless $a=b$. Moreover, if $a=b$, condition ($\diamondsuit$) is equivalent to say that $G(c)$ acts $2$-transitively on $X_a$.

\begin{proof}
To prove the first part of the assertion, let $a,b\in o^{-1}(c)$. By the essential uniqueness of the local action diagram associated to $(G,T)$ (cf.~Section~\ref{sus:lads}), we may assume that $X_a,X_b\subseteq o^{-1}(v)$. Since $(G,T)$ is weakly locally $\infty$-transitive at $v$, for every $e\in X_a$ the stabiliser $G_e$ acts transitively on 
$$\{\frp\in \Geod_T(\bar{e}\to T)\mid \pi(\frp)=(\bar{a},b)\}=\{(\bar{e},f)\mid f\in X_{b}\setminus \{e\}\}$$
and thus on $X_b\setminus\{e\}$.

Let now $(G,T)$ be (P)-closed and suppose that ($\diamondsuit$) holds for all~$c\in V\Gamma$.
Without loss of generality, let $G=U(\Delta,\iota,c)$, $T=T(\Delta,\iota,c)$ and $v\in VT$ be the root of~$T(\Delta,\iota,c)$ (so $\pi(v)=c$).
Consider two geodesics in~$T$ with the same image in~$\Gamma$, say $[v,w_1]=(e_1, \dots, e_n)$ and $[v,w_2]=(f_1, \dots, f_n)$ for some $n\geq 1$ and $w_1,w_2\in VT$.
For every $1\leq i\leq n$, since $\pi(e_i)=\pi(f_i)$, there exists $g_i\in G$ such that $f_i=g_i\cdot e_i$. Note that $o(g_i\cdot e_{i+1})=g_i\cdot t(e_i)=t(f_i)=o(f_{i+1})$ and $\pi(g_i\cdot e_{i+1})=\pi(e_{i+1})=\pi(f_{i+1})$. Hence, both $\caL(g_i\cdot e_{i+1})$ and $\caL(f_{i+1})$ belong to $X_{\pi(f_{i+1})}$.

By ($\diamondsuit$) applied to the vertex $c_i=o(\pi(f_{i+1}))$, for every $i<n$ there is $h_i\in G_{\bar{f}_i}$ such that $f_{i+1}=h_ig_i\cdot e_{i+1}$. 
For every $i<n$, set $k_i:=h_ig_i$ and observe that
\begin{equation}\label{eq:eifiki}
   (k_i\cdot e_i,k_i\cdot e_{i+1})=(f_i, f_{i+1}).
\end{equation}
In particular, $k_1=h_1g_1\in G_v$ (as~$h_1,g_1\in G_v$) and $k_i\cdot e_{i+1}=f_{i+1}=k_{i+1}\cdot e_{i+1}$ for every $1\leq i<n-1$. From~$(k_i)_{1\leq i\leq n-1}$, we now construct a sequence~$(\tilde{k}_i)_{1\leq i\leq n-1}$ in~$G_v$ such that~$[v,w_2]=\tilde{k}_{n-1}\cdot [v,w_1]$.

For every $1\leq i<n-1$ write $u_{i+1}:=k_{i}^{-1}k_{i+1}\in G_{e_{i+1}}$, and let $u_{i+1}^-\in G_{T_{\geq \bar{e}_{i+1}}}$ and $u_{i+1}^+\in G_{T_{\geq e_{i+1}}}$ be such that $u_{i+1}=u_{i+1}^-u_{i+1}^+$ (recall that $(G,T)$ is (P)-closed, so~\eqref{eq:propP} holds). For all~$2\leq i<j\leq n-1$, $u_i^+$ has support in~$T_{\geq \bar{e}_i}$ and~$u_j^-$ has support in~$T_{\geq e_j}$, which is disjoint to~$T_{\geq \bar{e}_i}$. Hence, for all~$2\leq i<j\leq n-1$ one has
\begin{equation}\label{eq:commut}
    u_i^+u_j^-=u_j^-u_i^+.
\end{equation}

Let $\tilde{k}_1:=k_1$ and, for every $2\leq i\leq n-1$, define inductively $\tilde{k}_{i}:=\tilde{k}_{i-1}u_{i}^-$. 
For each $2\leq i\leq n-1$, we claim that 
\begin{equation}\label{eq:kikitil}
    k_{i}=\tilde{k}_{i}u_2^+u_3^+\cdots u_{i}^+.
\end{equation}
For $i=2$, one has $k_2=k_1u_2=(\tilde{k}_1u_2^{-})u_2^+=\tilde{k}_2u_2^+$. For~$i\geq 3$, assuming~\eqref{eq:kikitil} true for~$i-1$, we have
\begin{equation*}
\begin{split}
   k_i & =k_{i-1}u_i=(\tilde{k}_{i-1}u_2^+\cdots u_{i-1}^+)(u_i^-u_i^+)\\
   & \stackrel{\eqref{eq:commut}}{=}(\tilde{k}_{i-1}u_i^-)u_2^+\cdots u_i^+=\tilde{k}_iu_2^+u_3^+\cdots u_{i}^+.  
\end{split}
\end{equation*}
Since~$u_2^+,\ldots, u_{i}^+$ fix both $e_i$ and $e_{i+1}$, from~\eqref{eq:eifiki} and~\eqref{eq:kikitil} we have
\begin{equation}\label{eq:tildek}
    (\tilde{k}_{i}\cdot e_i,\tilde{k}_i\cdot e_{i+1})=(k_i\cdot e_i,k_i\cdot e_{i+1})=(f_i,f_{i+1}).
\end{equation}
Moreover, we claim that $\tilde{k}_i$ fixes $v$ for every $1\leq i\leq n-1$. If $i=1$, this is clear. For $i\geq 2$, assume inductively that $\tilde{k}_{i-1}$ fixes $v$. Since $u_i^-$ fixes $T_{\geq \bar{e}_i}\ni v$ pointwise, we have that~$\tilde{k}_i=\tilde{k}_{i-1}u_i^-$ fixes~$v$. In particular, $\tilde{k}_{n-1}\in G_v$ and~\eqref{eq:tildek} yields
\begin{equation*}
    \tilde{k}_{n-1}\cdot w_1=\tilde{k}_{n-1}\cdot t(e_n)=t(\tilde{k}_{n-1}\cdot e_n)=t(f_n)=w_2.
\end{equation*}
   Being $T$ a tree, we conclude that $[v, w_2]=[v,\tilde{k}_{n-1}\cdot w_1]=\tilde{k}_{n-1}\cdot [v, w_1]$. 
\end{proof}

Proposition~\ref{prop:locWLIT} gives a recipe for constructing (P)-closed actions that are weakly locally $\infty$-transitive. We collect some explicit examples below.
\begin{ex}\label{ex:WLIT}
\begin{itemize}
    \item[(i)] Let $(T,\pi)$ be the universal cover of a connected edge-weighted graph $(\Gamma,\omega)$ as in~Example~\ref{exintro}. According to~Example~\ref{ex:UD}(i), one checks that ($\diamondsuit$) is satisfied for every $c\in V\Gamma$. Hence, by Proposition~\ref{prop:locWLIT}, $(\Aut_\pi(T),T)$ is weakly locally $\infty$-transitive. This gives an alternative proof of~\cite[Theorem~3.1 and the comment thereafter]{cawi20}.
    \item[(ii)] Let $U(F)$ be the Burger--Mozes universal group associated to a transitive group $F\leq \mathrm{Sym}(\{1,\ldots,k\})$, $k\geq 2$, acting on the barycentric subdivision $T_k'$ of the $k$-regular tree. 
    According to~Example~\ref{ex:UD}(iia), condition ($\diamondsuit$) of Proposition~\ref{prop:locWLIT} is satisfied for every $c\in V\Gamma$ if, and only if, $F$ is $2$-transitive. 
    Moreover, $U(F)\backslash T_k'$ is a $1$-segment. Hence, $(U(F),T_k')$ is weakly locally $\infty$-transitive if, and only if, it is locally $\infty$-transitive (cf.~Lemma~\ref{lem:(W)LIT}).
    By Proposition~\ref{prop:locWLIT}, we deduce that $(U(F),T_k')$ is locally $\infty$-transitive if, and only if, $F$ is $2$-transitive. This gives an alternative proof of what was observed in the first lines of~\cite[\S3]{bumo}.
    \item[(iii)] Let $U(F_1,F_2)$ be the Smith group associated to two transitive groups $F_1\leq\mathrm{Sym}(\{1,\ldots,k_1\})$ and $F_2\leq\mathrm{Sym}(\{1,\ldots, k_2\})$, $k_1,k_2\geq 2$. According to~Example~\ref{ex:UD}(iib), condition $(\diamondsuit)$ is satisfied for every $c\in V\Gamma$ if, and only if, both $F_1$ and $F_2$ are $2$-transitive. Moreover, the $U(F_1,F_2)$-action on the $(k_1,k_2)$-biregular tree~$T_{k_1,k_2}$ has quotient graph a $1$-segment. Hence, as in~(ii), $(U(F_1,F_2),T_{k_1,k_2})$ is weakly locally $\infty$-transitive if, and only if, it is locally $\infty$-transitive.
    By Proposition~\ref{prop:locWLIT} we conclude that $(U(F_1,F_2),T_{k_1,k_2})$ is locally $\infty$-transitive if, and only if, both $F_1$ and $F_2$ are $2$-transitive.
\end{itemize}
\end{ex}

As one may guess, there are (P)-closed actions that are not weakly locally $\infty$-transitive (and possibly not arc-transitive).
\begin{ex}\label{ex:nonwlit}
For $n\geq 3$, recall that the \emph{dihedral} group~$D_{2n}$ of size~$2n$ is the group of symmetries of an Euclidean regular $n$-gon, say with vertices~$x_1,\ldots, x_n$. Hence, we may consider $D_{2n}$ a subgroup of~$\Sym(\{x_1,\ldots, x_n\})$. Note that $D_{2n}$ acts transitively on~$\{x_1,\ldots, x_n\}$ (think of the rotations in the correspondent $n$-gon). However, for $n\geq 4$, the $D_{2n}$-action on~$\{x_1,\ldots, x_n\}$ is no longer $2$-transitive (as $D_{2n}$ preserves adjacent vertices in the $n$-gon).

Let~$\Delta=(\Gamma, (X_a), (G(c)))$ be any local action diagram with the following property: there is~$c\in V\Gamma$ such that~$o^{-1}(c)=\{a\}$ (so, $c$ is a leaf of~$\Gamma$) and, keeping the notation above, such that $G(c)=D_{2n}\leq \Sym(X_c)$, where~$X_c=X_a=\{x_1,\ldots, x_n\}$ for some $n\geq 4$. 

For example,~$\Gamma$ can be the graph as depicted below.
\begin{equation*}
    \xymatrix{
        \stackrel{c}{\bullet}\ar@/^/[rr]^a && \stackrel{d}{\bullet} \ar@/^/[rr]^{b}\ar@/^/[ll]^{\bar{a}} && \stackrel{e}{\bullet}\ar@/^/[ll]^{\bar{b}}
        }
\end{equation*}
 Moreover, one can take $X_a=X_{\bar{b}}:=\{x_1,\ldots, x_n\}$, $n\geq 4$, $G(c)=G(e):=D_{2n}$ as before, and let~$G(d)$ be any permutation group acting on a finite set~$X_d$ with two distinct orbits, say $X_{\bar{a}}$ and~$X_b$ (e.g., $G(d)=2\Z/4\Z$ acting on~$X_d=\Z/4\Z$ by left translation). 

If~$\bbT=(T,\pi,\caL)$ is any $\Delta$-tree, the action~$(U(\Delta,\bbT),T)$ is clearly (P)-closed but, by Proposition~\ref{prop:locWLIT}, not weakly locally $\infty$-transitive. Moreover, unless~$\Gamma$ is a $1$-segment, the latter action is also not arc-transitive.
\end{ex}

There are also (weakly) locally $\infty$-transitive group actions~$(G,T)$ on trees that are not (P)-closed but $G\leq \Aut(T)$ is closed. The following example is based on~\cite[Example~6.33]{ggt:aut}.
\begin{ex}\label{ex:nonPcl}
     Let~$T=T_{p+1}$ be the $(p+1)$-regular tree with $p$ prime.
     Recall by Example~\ref{ex:lads}(ii) that~$\PSL_2(\Q_p)$ acts faithfully on~$T$, so it can be embedded in~$\Aut(T)$.
     Let~$G$ be the closure of~$\PSL_2(\Q_p)$ in~$\Aut(T)$.
     
     By~\cite[Example~6.33]{ggt:aut} and~\cite[Corollary~6.28]{ggt:aut}, $(G,T)$ is not (P)-closed (that is, following~\cite{ggt:aut}, $G$ does not coincide with its ($P_1$)-closure). However, it is well known that the relevant $\PSL_2(\Q_p)$-action on~$T$ is locally $\infty$-transitive (cf.~\cite[pp.~91 and~95]{ser:trees}). In particular, $(G,T)$ is locally $\infty$-transitive.
\end{ex}

Example~\ref{ex:nonPcl} motivates the following question.
\begin{ques}
    Is there a weakly locally $\infty$-transitive group action on a tree~$(G,T)$ as in Convention~\ref{conv:action} such that~$G\leq \Aut(T)$ is closed, but $(G,T)$ is neither (P)-closed nor arc-transitive?
\end{ques}

\section{Cosets and double-cosets for groups acting on trees}\label{s:suborb}

Given a group acting on a tree, there is a standard geometric characterisation of coset and double-coset spaces (and of double-coset sizes) with respect to vertex and edge stabilisers. We expose it in Section~\ref{sus:cosGeo}. In view of Section~\ref{s:DCzeta}, we rephrase this characterisation in local terms in case that the action is
 weakly locally $\infty$-transitive (cf.~Section~\ref{sus:subWLIT}) or (P)-closed (cf.~Section~\ref{sus:subP}). For weakly locally $\infty$-transitive actions, we can conveniently enumerate double-cosets and compute their sizes in terms of paths in the quotient graph and a suitable weight on them (cf.~Propositions~\ref{prop:dcgeoWLIT}~and~\ref{prop:subWLIT}). In the (P)-closed case, we bypass the problem of enumerating double-cosets in local terms and focus on the associated coset spaces. Indeed, the latter ones can be conveniently described in terms of paths in the local action diagram (cf.~Proposition~\ref{prop:cosPcl}). This also allows us to express the relevant double-coset sizes in local terms (cf.~Proposition~\ref{prop:subsizeP}). We will see in Section~\ref{sus:coset} that counting cosets instead of double-cosets has only little impact on the double-coset zeta functions. 

\smallskip
 
We remind that we are following Convention~\ref{conv:action} and Convention~\ref{conv:lad}.

\subsection{Cosets and geodesics}\label{sus:cosGeo}
Since a tree is uniquely geodesic and by the orbit-stabiliser theorem, we observe what follows.

\begin{fact}\label{fact:cosets}
    Let $(G,T)$ be a group action on a tree and let $v\in VT$, $e\in ET$. For every $t\in VT$, there are $G$-equivariant bijections $\varphi_{v,t}\colon G/G_t\to \Geod_T(v\to G\cdot t)$ and $\varphi_{e,t}\colon G/G_t\to \Geod_T(\{e,\bar{e}\}\to G\cdot t)$ defined as follows:
    \begin{equation*}
      \varphi_{v,t}(gG_t):=[v,g\cdot t]\quad\text{and}\quad\varphi_{e,t}(gG_t):= \left\{
            \begin{array}{cc}
                [e, g\cdot t], & \text{if }{g\cdot t\in T_{\geq e}};\\
                {[\bar{e}, g\cdot t]}, & \text{if }g\cdot t\in T_{\geq \bar{e}}.
            \end{array}
            \right.
    \end{equation*}
    
    Similarly, for every $t\in ET$ there are $G$-equivariant bijections $\varphi_{v,t}\colon G/G_t\to\Geod_T(v\to G\cdot \{t,\bar{t}\})$ and $\varphi_{e,t}\colon G/G_t\to \Geod_T(\{e,\bar{e}\}\to G\cdot\{t,\bar{t}\})$ defined as follows:
    \begin{equation*}
    \begin{split}
        \varphi_{v,t}(gG_t) & :=\left\{
            \begin{array}{cc}
               [v,g\cdot t],  & \text{if }[v,g\cdot t]\text{ exists in }T;\\
               {[v, g\cdot\bar{t}]}, & \text{if }[v,g\cdot \bar{t}]\text{ exists in }T;
            \end{array}
            \right.\quad\text{and}\\ 
            \varphi_{e,t}(gG_t)& :=\left\{
            \begin{array}{cc}
                {[e, g\cdot t]}, & \text{if }[e,g\cdot t]\text{ exists in }T;\\
                {[e, g\cdot \bar{t}]}, & \text{if }[e,g\cdot \bar{t}]\text{ exists in }T;\\
                {[\bar{e}, g\cdot t]}, & \text{if }[\bar{e},g\cdot t]\text{ exists in }T;\\
                {[\bar{e}, g\cdot \bar{t}]}, & \text{if }[\bar{e},g\cdot \bar{t}]\text{ exists in }T.
            \end{array}
            \right.
        \end{split}
    \end{equation*}
\end{fact}

\begin{lem}\label{lem:subsize}
    Let $(G,T)$ be a group action on a tree. According to Fact~\ref{fact:cosets}, for all $t_1,t_2\in T$ we have
    \begin{gather*}\label{eq:sizevet}
        |G_{t_1}gG_{t_2}/G_{t_2}|=|G_{t_1}\cdot \varphi_{t_1,t_2}(gG_{t_2})|.
    \end{gather*}
\end{lem}
\begin{proof}
 For every $g\in G$ we observe that $$|G_{t_1}gG_{t_2}/G_{t_2}|=|G_{t_1}:G_{t_1}\cap gG_{t_2}g^{-1}|.$$ 
 Since $gG_{t_2}g^{-1}=G_{g\cdot t_2}$ and $T$ is uniquely geodesic, the group $G_{t_1}\cap gG_{t_2}g^{-1}$ is the pointwise stabiliser of the geodesic $\varphi_{t_1,t_2}(gG_{t_2})$. Then the orbit-stabiliser theorem yields the claim. 
\end{proof}

\subsection{The case of weakly locally $\infty$-transitive actions on trees}\label{sus:subWLIT}

Let $(G,T)$ be a weakly locally $\infty$-transitive group action on a tree. Denote by $\pi\colon T\to \Gamma=G\backslash T$ the quotient map of $(G,T)$, and let $\omega$ be the standard edge weight on $\Gamma$.
Given $u_1,u_2\in \Gamma$ and $t_1\in \pi^{-1}(u_1)$, let $\calPl_{\Gamma,t_1}(u_1\to u_2)$ be the set of all paths $\frp\in \calP_\Gamma(u_1\to u_2)$ that can be lifted to a geodesic in~$T$ from $t_1$ via $\pi$.

\begin{prop}\label{prop:dcgeoWLIT}
    In the hypotheses before, let $v\in VT$ with $\pi(v)=c$ and $e\in ET$ with $\pi(e)=a$. According to Fact~\ref{fact:cosets}, we have the following bijections for every~$t\in T$ with $\pi(t)=u$:
    \begin{gather*}
        \begin{array}{rcll}
           \Psi_{v,t}\colon G_v\backslash G/G_t & \longrightarrow & \calPl_{\Gamma, v}(c\to U), &  \Psi_{v,t}(G_vgG_t)=\pi(\varphi_{v,t}(gG_t));\\
           \Psi_{e,t}\colon G_e\backslash G/G_t & \longrightarrow &  \calPl_{\Gamma, e}(A\to U), & \Psi_{e,t}(G_egG_t)=\pi(\varphi_{e,t}(gG_t)).
        \end{array}
    \end{gather*}
    Here the sets $A$ and $U$ are as in Notation~\ref{notat:capitG}.
\end{prop}
For all $t_1,t_2\in T$ and all $h\in G_{t_1}$, note that 
$$\pi(\varphi_{t_1,t_2}(hgG_{t_2}))=\pi(h\cdot \varphi_{t_1,t_2}(gG_{t_2}))=\pi(\varphi_{t_1,t_2}(gG_{t_2}))$$
and thus the map $\Psi_{t_1,t_2}$ is well-defined.
\begin{proof}
    We only prove that $\Psi_{v,t}$ is bijective for $t\in VT$, as for the remaining cases one may argue analogously. Let $\pi_{v,t}\colon \Geod_T(v\to G\cdot t)\longrightarrow\calPl_{\Gamma,v}(c\to u)$ be the map defined as
    \begin{equation*}\label{eq:pivt}
        \pi_{v,t}([v, g\cdot t]):=\pi([v, g\cdot t]).
    \end{equation*}
    Clearly, $\pi_{v,t}$ is $G$-equivariant and surjective. Moreover, since the $G$-action on~$T$ is weakly locally $\infty$-transitive, for every $g\in G$ the $\pi_{v,t}$-fibre of $\pi([v,g\cdot t])$ is $G_v\cdot [v,g\cdot t].$
     Thus $\pi_{v,t}$ induces a bijective map
     \begin{equation*}
         \tilde{\Psi}_{v,t}\colon G_v\backslash \Geod_T(v\to G\cdot t)\longrightarrow \calPl_{\Gamma,v}(c\to u).
     \end{equation*}
   Composing the bijection $G_v\backslash G/G_t\to G_v\backslash \Geod_T(v\to G\cdot t)$ induced by $\varphi_{v,t}$ with $\tilde{\Psi}_{v,t}$, we obtain $\Psi_{v,t}$. 
\end{proof}    

In view of Proposition~\ref{prop:subWLIT},
we introduce a weight on the paths of $G\backslash T$ which extends the standard edge weight $\omega$ in $G\backslash T$ and such that
$|G_{t_1}gG_{t_2}/G_{t_2}|$ coincides with the weight of $\Psi_{t_1,t_2}(G_{t_1}gG_{t_2})$, for all $t_1,t_2\in T$.
Such a weight can be defined even in the following more general framework.
\begin{defn}\label{defn:N}
   Let $\Gamma$ be a graph with a function $\omega\colon E\Gamma\to \Z_{\geq 1}$ called \emph{edge weight}. Define two functions $\Ne=\Ne^\omega, \Nv=\Nv^\omega\colon \calP_\Gamma\to \Z_{\geq 0}$ as follows.
   For every path $\frp$ in~$\Gamma$, let
    \begin{equation}\label{eq:Ne}
     \Ne(\frp):=\left\{
     \begin{array}{cl}
         \ell(\frp), & \text{if }\ell(\frp)\leq 1\\
        \displaystyle{\prod_{i=1}^{n-1}\Big(\omega(a_{i+1})-\mathbbm{1}_{\{\bar{a}_i\}}(a_{i+1})\Big),} & \text{if }\frp=(a_1, \ldots, a_n),\,n\geq 2 
     \end{array}
     \right.
    \end{equation} 
    and
    \begin{equation}\label{eq:Nv}
     \Nv(\frp):=\left\{
     \begin{array}{cl}
         1, & \text{if }\ell(\frp)=0\\
        \omega(a_1)\cdot\Ne(\frp), & \text{if }\frp=(a_1, \ldots, a_n),\,n\geq 1 .
     \end{array}
     \right.
    \end{equation}
\end{defn}
\begin{notat}
    For $(a_1,\ldots, a_n)\in \calP_\Gamma$, we write $\Ne(a_1,\ldots, a_n)$ and \linebreak $\Nv(a_1, \ldots, a_n)$ in place of $\Ne((a_1,\ldots,a_n))$ and $\Nv((a_1,\ldots, a_n))$, respectively.
\end{notat}

\begin{rem}\label{rem:Nint}
  Let $\Gamma$ be the quotient graph of a group action on a tree $(G,T)$, and denote by $\omega\colon E\Gamma\to\Z_{\geq 1}$ its standard edge weight.

 For every path $\frp=(a_1, \ldots, a_n)$ in~$\Gamma$ of positive length and for every $e\in ET$ with $\pi(e)=a_1$, one checks that
 \begin{equation*}
      \Ne^\omega(\frp)=|\{\widetilde{\frp}\in \Geod_T(e\to T)\,:\, \pi(\widetilde{\frp})=\frp\}|.
 \end{equation*}
 Similarly, for every path $\frp$ in~$\Gamma$ starting at $c\in V\Gamma$ and for every $v\in VT$ with $\pi(v)=c$, we have
  \begin{equation*}
      \Nv^\omega(\frp)=|\{\widetilde{\frp}\in \Geod_T(v\to T)\,:\, \pi(\widetilde{\frp})=\frp\}|.
  \end{equation*}

 These observations apply if in particular $\Gamma$ is the underlying graph of a local action diagram $\lad$, and $T=T(\Delta,\iota,c_0)$ is a standard $\Delta$-tree (cf.~Section~\ref{ssus:TD}). In this case, given a path $\frp=(a_1,\ldots, a_n)$ in~$\Gamma$ of positive length and for every $e\in ET^+$ with $\pi(e)=a_1$, we have
  \begin{equation}\label{eq:NintP}
      \Ne^{\omega}(\frp)=|\{\xi\in \calP_{(\Delta,\iota)}(\caL(e)\to X)\,:\,\frp_\xi=\frp\}|,
  \end{equation}
  (cf.~Lemma~\ref{lem:geodTD}).
Thus, a reduced path $\xi$ in~$(\Delta,\iota)$ of positive length can be lifted to a geodesic in~$T$ if, and only if, $\Ne^\omega(\frp_{\xi})\geq 1$.
\end{rem}

\begin{rem}\label{rem:lift}
Assume the notation of Remark~\ref{rem:Nint} and suppose that $\omega(a)\geq 2$ for every $a\in E\Gamma$. Then every path $\frp$ can be lifted to a geodesic in~$T$ via $\pi$ (cf.~Remark~\ref{rem:Nint}). In particular, the bijections in Proposition~\ref{prop:dcgeoWLIT} are onto $\calP_\Gamma(c\to U)$ and $\calP_\Gamma(A\to U)$, respectively.
\end{rem}

\begin{prop}\label{prop:subWLIT}
 Let $(G,T)$ be a weakly locally $\infty$-transitive group action on a tree with quotient graph $\Gamma$. Assume that the standard edge weight $\omega$ on~$\Gamma$ takes finite values. 
 Then, for all $g\in G$ and $t_1,t_2\in T$,
 \begin{equation*}
     |G_{t_1}gG_{t_2}/G_{t_2}|=|G_{t_1}:G_{\varphi_{t_1,t_2}(gG_{t_2})}|=\left\{
     \begin{array}{ll}
        \Nv^\omega(\Psi_{t_1,t_2}(G_{t_1}gG_{t_2})),  & \hspace{-0.2cm}\text{if }t_1\in VT;\\
        \Ne^\omega(\Psi_{t_1,t_2}(G_{t_1}gG_{t_2})),  & \hspace{-0.2cm}\text{if }t_1\in ET.
     \end{array}
     \right.
 \end{equation*}
 Here the map $\Psi_{t_1,t_2}$ is as in Proposition~\ref{prop:dcgeoWLIT}.
\end{prop}

\begin{proof}
We may assume that $\varphi_{t_1,t_2}(gG_{t_2})=[t_1, g\cdot t_2]$. Indeed, in the other cases the argument is analogous. Let $\frp:=\Psi_{t_1,t_2}(G_{t_1}gG_{t_2})=\pi([t_1,g\cdot t_2])$, where $\pi$ denotes the quotient map of $(G,T)$ (extended entrywise to all paths).
By Lemma~\ref{lem:subsize} and by weak local $\infty$-transitivity, we deduce that
\begin{gather*}
|G_{t_1}gG_{t_2}/G_{t_2}|=|G_{t_1}\cdot [t_1, g\cdot t_2]|=|\{\widetilde{\frp}\in \Geod_T(t_1\to T)\,:\, \pi(\widetilde{\frp})=\frp\}|.
\end{gather*}
Now Remark~\ref{rem:Nint} applies.
\end{proof}

\subsection{The case of (P)-closed actions on trees}\label{sus:subP}
Let $\Delta$ be a local action diagram and let~$(T=T(\Delta,\iota,c_0),\pi,\caL)$ be the standard $\Delta$-tree associated to an inversion map $\iota$ in~$\Delta$ and a chosen $c_0\in V\Gamma$ (cf.~Section~\ref{ssus:TD}). 
Let $G\leq U(\Delta,\iota,c_0)$ be a subgroup acting on $T$ with local action diagram $\Delta$. For the definition of $U(\Delta,\iota,c_0)$, see Section~\ref{ssus:UD}.

The following proposition rephrases Fact~\ref{fact:cosets} in the language of local action diagrams.
\begin{prop}\label{prop:cosPcl} 
 Let $G\leq U(\Delta,\iota,c_0)$ and $T=T(\Delta,\iota,c_0)$ be as before, and denote by $v_0$ the root of $T$. Consider also $e\in o^{-1}(v_0)$ with $\caL(e)=x$ and $t\in T$ with $\pi(t)=u$. Let $\varphi_{u,t}, \varphi_{e,t}$ be the maps introduced in Fact~\ref{fact:cosets}, and denote by $\caL$ the map defined in \eqref{eq:Lpaths}. Then the following two maps are bijective:
        \begin{gather*}
        \begin{array}{ccl}
            \caL\circ\varphi_{v_0,t}\colon G/G_t & \longrightarrow & \calP_{(\Delta,\iota)}(X_{c_0}\to X_U);\\
    \caL\circ\varphi_{e,t}\colon G/G_t & \longrightarrow & \calP_{(\Delta,\iota)}(x\to X_U)\sqcup \iota(x)\cdot \calP_{(\Delta,\iota)}(X_{c_0}\setminus\{x\}\to X_U). 
        \end{array}
        \end{gather*}   
    For $X_U$ see Notation~\ref{notat:capitD}.
\end{prop}

\begin{proof}
It is a direct consequence of Fact~\ref{fact:cosets} and Lemma~\ref{lem:geodTD}.
\end{proof}

Following a similar strategy to the one in Section~\ref{sus:subWLIT}, we define a weight on paths in~$\Delta$ as follows. 

\begin{defn}\label{defn:W}
    Let $\lad$ be a local action diagram and recall that $X=\bigsqcup_{a\in E\Gamma}X_a$. The \emph{standard weight} on $\Delta$ is the function $\caW\colon X\times X\longrightarrow \Z_{\geq 0}\cup\{\infty\}$ defined, for all $x\in X_a$ and $y\in X_b$ with $a,b\in E\Gamma$, as follows:
\begin{equation*}\label{eq:W}
    \caW(x,y):=\left\{
    \begin{array}{cl}
       |{G(t(a))}_{\iota(x)}\cdot y|,  &  \text{if }t(a)=o(b);\\
       0,  &  \text{otherwise.}
    \end{array}
    \right.
\end{equation*} 
Define also $\caW_{\mathrm{rev}}\colon X\times X\longrightarrow \Z_{\geq 0}\cup\{\infty\}$ as follows, for all $x\in X_a$, $y\in X_b$ with $a,b\in E\Gamma$:
\begin{equation*}
    \caW_{\mathrm{rev}}(x,y):=\left\{
    \begin{array}{cl}
       |G(o(a))_x\cdot y|,  &  \text{if }o(a)=o(b);\\
       0,  &  \text{otherwise.}
    \end{array}
    \right.
\end{equation*}
Moreover, for every sequence $\xi=(x_1,\ldots, x_n)$ of elements of $X$ of length $n\geq 0$, define
\begin{equation}\label{eq:Wpath}
    \caW(\xi):=\left\{
    \begin{array}{cl}
        1, & \text{if }n\leq 1;\\
        \displaystyle{\prod_{i=1}^{n-1}\caW(x_i, x_{i+1}),} & \text{if }n\geq 2.
    \end{array}
    \right.
\end{equation}
\end{defn}
\begin{rem}\label{rem:Wint}
   For $x,y\in X$ note that $\caW(x,y)\neq 0$ if, and only if, $(x,y)$ is a path in $\Delta$. More generally, given a sequence $\xi$ of elements of $X$, we have $\caW(\xi)\neq 0$ if and only if $\xi$ is a path in $\Delta$.
\end{rem}

\begin{notat}
    For a sequence $\xi=(x_1,\ldots, x_n)$, we write $\caW(x_1,\ldots, x_n)$ in place of $\caW((x_1,\ldots, x_n))$.
\end{notat}

\begin{rem}\label{rem:wW}
Assume the hypotheses of the section, let $G=U(\Delta,\iota,c_0)$ and denote by $\omega$ and $\caW$ the standard edge weights on $\Gamma$ and $\Delta$, respectively. For every $e\in ET$ we have
    \begin{equation*}\label{eq:}
        |G_{o(e)}\cdot e|=\omega(\pi(e)).
    \end{equation*}
Moreover, let $e,f\in ET$ with $t(e)=v=o(f)$. If $e\in ET^+$, then $\caL(\bar{e})=\iota(\caL(e))$ (cf.~Remark~\ref{rem:Linv}) and 
$$|G_e\cdot f|=|G_{\bar{e}}\cdot f|=|G(\pi(v))_{\iota(\caL(e))}\cdot\caL(f)|=\caW(\caL(e),\caL(f)).$$
Moreover, if $e\in ET\setminus ET^+$ we have
$$|G_e\cdot f|=|G_{\bar{e}}\cdot f|=|G(\pi(v))_{\caL(\bar{e})}\cdot f|=\caWr(\caL(\bar{e}),\caL(f)).$$
\end{rem}

\begin{ex}\label{ex:preWLIT}
    Let $\lad$ be a local action diagram and consider $a,b\in E\Gamma$ with $t(a)=c=o(b)$. Assume that, for every $x\in X_a$, the group ${G(c)}_{\iota(x)}$ acts transitively on $X_b\setminus \{\iota(x)\}$. Note that $X_b\setminus \{\iota(x)\}=X_b$ unless $b=\bar{a}$.
    Then, for all $x\in X_{a}$ and $y\in X_b\setminus \{\iota(x)\}$, we have
    \begin{equation*}\label{eq:Ww}
       \caW(x,y)=|G(c)_{\iota(x)}\cdot y|=|X_b\setminus\{\iota(x)\}|=w(b)-\mathbbm{1}_{\{\bar{a}\}}(b).
    \end{equation*}
 The transitivity condition before leads back to Proposition~\ref{prop:locWLIT}. For explicit examples satisfying it, see Example~\ref{ex:WLIT}.
\end{ex}

\begin{prop}\label{prop:subsizeP}
    Let $G=U(\Delta,\iota,c_0)$. Consider a geodesic $\frp=(e_1,\ldots, e_n)$ in~$T=T(\Delta,\iota,c_0)$ with $n\geq 1$,  $e_1,\ldots, e_n\in ET^+$ and $\caL(\frp)=(x_1,\ldots, x_n)$. Then, 
    \begin{equation*}
|G_{o(e_1)}:G_{\frp}|=\omega(\pi(e_1))\cdot\caW(\xi)\quad\text{and}\quad |G_{e_1}:G_{\frp}|=\caW(\xi).
    \end{equation*}
\end{prop}

\begin{proof}
    By the orbit stabiliser theorem,
    \begin{equation}\label{eq:orb1}
    \begin{split}
        |G_{o(e_1)}:G_\frp| & =|G_{o(e_1)}: G_{e_1}|\cdot \prod_{i=1}^{n-1}|G_{(e_1, \dots, e_i)}: G_{(e_1, \dots, e_{i+1})}|\\
                           & = |G_{o(e_1)}\cdot e_1|\cdot \prod_{i=1}^{n-1}|G_{(e_1, \dots, e_i)}\cdot e_{i+1}|.
    \end{split}
    \end{equation}
    Similarly, 
    \begin{equation}\label{eq:orb2}
        |G_{e_1} : G_\frp|=\prod_{i=1}^{n-1}|G_{(e_1, \dots, e_i)}\cdot e_{i+1}|.
    \end{equation}
 By~Lemma~\ref{lem:propPbis}, for every $1\leq i\leq n-1$ we have
   \begin{equation}\label{eq:orbEdge}
       G_{(e_1, \dots, e_i)}\cdot e_{i+1}=G_{e_i}\cdot e_{i+1}.
   \end{equation}
Combining~Remark~\ref{rem:wW}, \eqref{eq:orb1}, \eqref{eq:orb2} and~\eqref{eq:orbEdge}, we conclude the claim.
\end{proof}

\begin{cor}\label{cor:subsizP}
Let $G=U(\Delta,\iota,c_0)$ and denote by $v_0$ the root of $T=T(\Delta,\iota,c_0)$. Let $e\in ET\setminus ET^+$ with $t(e)=v_0$, and consider $t\in T$ such that $[v_0,t]$ is defined. Set $\frp_t:=e\cdot[v_0,t]=(e_1=e,e_2,\ldots, e_n)$, $\caL(\bar{e})=x_1$ and $\caL(\frp_t)=(\iota(x_1),x_2,\ldots, x_n)$. Then
    \begin{equation*}
            |G_e:G_{\frp_t}|=
            \left\{
            \begin{array}{cl}
                1, &  \text{if }n=1;\\
                \caWr(x_1,x_2)\cdot\caW(x_2,\ldots, x_n), & \text{otherwise}.
            \end{array}
            \right.
        \end{equation*}
\end{cor}

\begin{proof}
   First, note that
    \begin{equation*}
        |G_{e}:G_{\frp_t}|=|G_e:G_{(e_1,e_2)}|\cdot |G_{(e_1,e_2)}:G_{\frp_t}|.
    \end{equation*}
    By Remark~\ref{rem:wW}, we deduce that
    $$|G_e:G_{(e_1,e_2)}|=|G_{e_1}\cdot e_2|=\caWr(x_1,x_2).$$
    From Remark~\ref{rem:geoE+} observe that $e_2,\ldots, e_n\in ET^+$. Hence, Lemma~\ref{lem:propPbis} and Remark~\ref{rem:wW} yield 
    \begin{equation*}
        \begin{split}
            |G_{(e_1,e_2)}:G_{\frp_t}| & =\prod_{i=2}^{n-1}|G_{(e_1,\ldots, e_i)}:G_{(e_1,\ldots, e_{i+1})}|=\prod_{i=2}^{n-1}|G_{e_i}\cdot e_{i+1}|\\
            & =\prod_{i=2}^{n-1}\caW(x_i,x_{i+1})=\caW(x_2,\ldots, x_n).\qedhere
        \end{split}
    \end{equation*}
\end{proof}

\begin{cor}\label{cor:Wgeod}
    Let $G=U(\Delta,\iota,c_0)$, $T=T(\Delta,\iota,c_0)$ and assume that $|X_a|\geq 2$ for every $a\in E\Gamma$. Then, for each geodesic $\frp=(e_1,\ldots, e_n)$ in~$T$ of length $n\geq 1$, there is $g\in G_{o(e_1)}$ such that $g\cdot e_i\in ET^+$ for every $i\leq n$ and 
    \begin{equation*}
        |G_{e_1}:G_{\frp}|=\caW(\caL(g\cdot \frp)).
    \end{equation*}
\end{cor}
\begin{proof}
    If $e_1\in ET^+$, by Remark~\ref{rem:geoE+} we may take $g=1$ and Proposition~\ref{prop:subsizeP} applies. Assume now that $e_1\in ET\setminus ET^+$ and let $(f_1,\ldots, f_r=\bar{e}_1)$ be the geodesic from the root $v_0$ of $T$ to $\bar{e}_1\in ET^+$. Set $\caL(f_i)=x_i$ for every $1\leq i\leq r$. In particular, $\caL(\bar{e}_1)=x_r$ and then $\caL(e_1)=\iota(x_r)$ (cf.~Remark~\ref{rem:Linv}).
    Since $|X_{\pi(e_1)}|\geq 2$, there is $y\in X_{\pi(e_1)}\setminus \{\iota(x_r)\}$ such that $(x_1,\ldots, x_r,y)$ is a reduced path in~$(\Delta,\iota)$. Then there is $f\in ET^+$ such that $o(f)=o(e_1)=t(\bar{e}_1)$, and $t(f)$ corresponds to $(x_1,\ldots, x_r,y)$ (cf.~\eqref{eq:VTD}). Since both $\caL(f)=y$ and $\caL(e_1)$ belong to $X_{\pi(e_1)}$, there is $g\in G_{o(e_1)}$ such that $f=g\cdot e_1$. By Remark~\ref{rem:geoE+}, every edge of $g\cdot \frp=(g\cdot e_1,\ldots, g\cdot e_n)$ belongs to $ET^+$. Moreover,
    \begin{equation*}
      |G_{e_1}\cdot \frp|=|gG_{e_1}\cdot \frp|=|G_{g\cdot e_1} g\cdot \frp|
    \end{equation*}
    and Proposition~\ref{prop:subsizeP} applies.
\end{proof}

\section{Double-coset zeta functions for groups acting on trees}\label{s:DCzeta}
This section deals with the convergence and explicit formulae for the relevant double-coset zeta functions for groups acting on trees. We introduce a family of properties (labelled with positive integers) on group actions on trees (cf.~Section~\ref{sus:propast}). 
In Proposition~\ref{prop:conv*}, we exploit that one of these properties is satisfied to deduce that the group has polynomial double-coset growth with respect to vertex or edge stabilisers.
The latter result can be refined to a characterisation in case that the action is weakly locally $\infty$-transitive (cf.~Theorem~\ref{thm:convWLIT}) or (P)-closed (cf.~Theorem~\ref{thm:convP}). In these two cases, we also provide explicit formulae for the relevant double-coset zeta functions in terms of the local data introduced in Section~\ref{s:suborb} (cf.~Theorems~\ref{thm:meroWLIT}~and~\ref{thm:meroP}).

\smallskip

We remind that we are following Convention~\ref{conv:action} and Convention~\ref{conv:lad}.

\subsection{From double-cosets to cosets}\label{sus:coset}
Let $G$ be a group and $H,K\leq G$ be subgroups such that $|HgK/K|<\infty$ for every $g\in G$. For each $n\geq 1$, consider $a_n(G,H,K)$ as in \eqref{eq:an} and define
\begin{equation}\label{eq:bn}
    b_n(G,H,K):=|\{gK\in G/K\,:\, |HgK/K|=n\}|.
\end{equation} 
We claim that $b_n(G,H,K)<\infty$ if, and only if, $a_n(G,H,K)<\infty$. Moreover, if $b_n(G,H,K)<\infty$ then
\begin{equation}\label{eq:bnan}
    b_n(G,H,K)=n\cdot a_n(G,H,K).
\end{equation}
To see this, consider the map
$\varphi\colon gK\in G/K\longmapsto HgK\in H\backslash G/K$. For every $g\in G$ we have $\varphi^{-1}(HgK)=\{hgK\in G/K\mid h\in H\}$ and $|\varphi^{-1}(HgK)|=|H:H\cap gKg^{-1}|=|HgK/K|$. 

In particular, if $(G,H,K)$ has polynomial double-coset growth then  
\begin{equation}\label{eq:zetas+1}
    \zeta_{G,H,K}(s)=\sum_{n=1}^\infty b_n(G,H,K)\cdot n^{-s-1}=\sum_{gK\in G/K}|HgK/K|^{-s-1}.
\end{equation}

\begin{rem}\label{rem:redUD}
    Let $\Delta$ be a local action diagram and $\mathbb{T}=(T,\pi,\caL)$ be $\Delta$-tree.
    Consider an inversion $\iota$ in $\Delta$ and denote by $(T(\Delta,\iota,c_0),\pi_0,\caL_0)$ the standard $\Delta$-tree associated to $\iota$ and some $c_0\in V\Gamma$. By Theorem~\ref{thm:UD}(i),
    there is a graph isomorphism $\phi\colon T\to T(\Delta,\iota,c_0)$ such that $\pi=\pi_0\circ \phi$ and $U(\Delta,\iota,c_0)=\phi U(\Delta,\mathbb{T})\phi^{-1}$. Set $G=U(\Delta,\mathbb{T})$ and $H=U(\Delta,\iota,c_0)$. Then, for all $t\in T$, the following map is bijective:
\begin{equation*}
    G/G_t\longrightarrow H/H_{\phi(t)}, \quad gG_t\longmapsto \phi g\phi^{-1} H_{\phi(t)}.
\end{equation*}
Moreover, for all $t_1,t_2\in T$ and $g\in G$, provided $h:=\phi g\phi^{-1}$ we have
\begin{equation*}
    |G_{t_1}: G_{t_1}\cap gG_{t_2}g^{-1}|=|H_{\phi(t_1)}:H_{\phi(t_1)}\cap hH_{\phi(t_2)}h^{-1}|.
\end{equation*}
Hence,
\begin{equation*}
a_n(G,G_{t_1},G_{t_2})=a_n(H,H_{\phi(t_1)},H_{\phi(t_2)}),\quad\forall\,n\geq 1
\end{equation*}
as well as
\begin{equation*}
    \zeta_{G,G_{t_1},G_{t_2}}(s)=\zeta_{H,H_{\phi(t_1)},H_{\phi(t_2)}}(s)
\end{equation*}
when the series before are defined.

In particular, by Theorem~\ref{thm:UD}(i), given $v\in VT$ with $\pi(v)=:c_0$ one may take~$\phi$ so that~$\phi(v)$ is the root $v_0$ of $T(\Delta,\iota,c_0)$. Then, for all $t\in T$,
    \begin{equation*}
    \begin{split}
         a_n(G,G_v,G_t) & =a_n(H,H_{v_0}, H_{\phi(t)}),\quad\forall\,n\geq 1;\quad\text{and}\\
         \zeta_{G,G_v,G_t}(s) & =\zeta_{H,H_{v_0}, H_{\phi(t)}}(s).
    \end{split}
    \end{equation*}
  Moreover, for all $e\in o^{-1}(v)$, we have $\phi(e)\in o^{-1}(v_0)$ and then, for all $t\in T$,
     \begin{equation*}
     \begin{split}
     a_n(G,G_e,G_t) & =a_n(H,H_{\phi(e)}, H_{\phi(t)}),\quad\forall\,n\geq 1;\quad\text{and}\\
        \zeta_{G,G_e,G_t}(s) & =\zeta_{H,H_{\phi(e)}, H_{\phi(t)}}(s).
        \end{split}
    \end{equation*}
\end{rem}

\subsection{The property ($\ast_k$)}\label{sus:propast}
Let $k\geq 1$. A group action on a tree $(G,T)$ has \emph{property~$(\ast_k)$} if,
for every geodesic $(e_1,\ldots, e_{l+k})$ in~$T$ with $l\geq 1$, we have
\begin{equation}\label{eq:prop*}
    |G_{(e_1,\ldots, e_l)}\cdot (e_{l+1},\ldots, e_{l+k})|\geq 2.
\end{equation}
One checks that property ($\ast_k$) implies property ($\ast_{k+1}$) for every $k\geq 1$.

\begin{lem}\label{lem:prop*nondisc}
    Let $T$ be a tree without leaves and $G\leq \Aut(T)$ be a subgroup with the permutation topology induced by~$\Aut(T)$. If $(G,T)$ has property~($\ast_k$) for some $k\geq 1$, then $G$ is non-discrete. 
\end{lem}
    
    \begin{proof} 
    Assume that $(G,T)$ has property~$(\ast_k)$. If we prove that all vertex-stabilisers in $G$ are infinite, then~\cite[Lemma~2.1]{bew:simp} yields the claim.
    Let $v\in VT$. Since $T$ has no leaves, there is a ray $(e_i)_{i\in \Z_{\geq 1}}$ in~$T$ with $o(e_1)=v$. For every $h\in \Z_{\geq 1}$, set $\frp_h=(e_i)_{1\leq i\leq hk}$ and note that $|G_{\frp_{h}}\cdot (e_{hk+1}, \ldots, e_{hk+k})|=|G_{\frp_{h}}:G_{\frp_{h+1}}|\geq 2$. Therefore, 
    $$|G_v:G_{\frp_{h}}|=|G_v:G_{\frp_1}|\cdot\prod_{i=1}^{h-1}|G_{\frp_i}:G_{\frp_{i+1}}|\geq 2^{h-1}, \quad \forall\,h\geq 1$$
    and $G_{v}$ is infinite.
\end{proof}

\begin{lem}\label{lem:opends}
    Let~$T$ be a tree and $\xi$ be an end of~$T$. If there is $e\in ET$ such that $G_e\subseteq G_\xi$, then $(G,T)$ does not have property ($\ast_k$) for any~$k\geq 1$. 
\end{lem}

 \begin{proof}
 Let $r=(e_i)_{i\in\Z_{\geq 1}}$ be a ray in $\xi$ with $e_1\in \{e,\bar{e}\}$. Note that such a ray exists: indeed, we can always find a representative of~$\xi$ which is completely contained either in~$T_{\geq e}$ or in~$T_{\geq \bar{e}}$, and~$r$ is obtained by extending it to a ray starting from either~$e$ or~$\bar{e}$. 

Since $G_e\subseteq G_{\xi}$, for every $g\in G_e$ there is $i\geq 1$ such that $g\cdot e_j$ is an edge of~$r$ for every $j\geq i$. Since $T$ is uniquely geodesic, we conclude that every $g\in G_e$ fixes every edge of~$r$.
\end{proof}

\begin{prop}\label{prop:prop*W}
    Let $(G,T)$ be a weakly locally $\infty$-transitive action on a locally finite tree. Denote by $\omega$ the standard edge weight on $\Gamma=G\backslash T$, consider $\Ne=\Ne^\omega$ as in Definition~\ref{defn:N}, and let $k\geq 1$. Then $(G,T)$ has property~($\ast_k$) if and only if $\Ne(\rho)\geq 2$ for every path $\rho$ in~$\Gamma$ of length $k+1$ that can be lifted to a geodesic in~$T$.
\end{prop}

\begin{proof}
    Let $k\geq 1$ and $\frp=(e_1,\ldots, e_{l+k})$ be a geodesic in~$T$ with $l\geq 1$. Denote by $\pi\colon T\to\Gamma$ the quotient map of $(G,T)$ and let $\pi(e_i)=a_i$ for every $1\leq i\leq l+k$. By Remark~\ref{rem:WLIT}(ii), $G_{(e_1,\ldots,e_l)}$ acts transitively on 
    $$\{\frq=(f_1,\ldots, f_{l+k})\in \Geod_T(e_1\to T)\,:\,\pi(\frq)=\pi(\frp)\text{ and }\forall\,i\leq l,\,f_i=e_i\}.$$ 
    Hence, by Remark~\ref{rem:Nint},
    \begin{gather}\label{eq:prop*W}
    \begin{split}
       & |G_{(e_1,\ldots,e_l)}\cdot (e_{l+1},\ldots, e_{l+k})|=\\
       & =|G_{(e_1,\ldots, e_l)}\cdot\frp|=\\
       & =\Big|\Big\{\frq=(f_1,\ldots,f_{l+k})\in \Geod_T(e_1\to T): \pi(\frq)=\pi(\frp);\forall\,i\leq l,\,f_i=e_i\Big\}\Big|\\
       & =\prod_{i=l}^{l+k-1}\Ne(a_i,a_{i+1})=\Ne(a_l,\ldots, a_{l+k}).
    \end{split}
    \end{gather}
    This yields the "if" part of the statement. For the "only if" part, let $\rho=(a_1,\ldots, a_{k+1})$ be an arbitrary path in~$\Gamma$ that can be lifted to a geodesic $\frp=(e_1,\ldots, e_{k+1})$ in~$T$. Arguing as in~\eqref{eq:prop*W}, we conclude that $|G_{e_1}\cdot \frp|=\Ne(\rho)\geq 2$.
\end{proof}

\begin{cor}\label{cor:*WLIT}
    Let $(G,T)$ be a weakly locally $\infty$-transitive group action on a locally finite tree with quotient graph $\Gamma$ and standard edge weight $\omega$. Assume that $\omega(E\Gamma)\subseteq \Z_{\geq 2}$. Then the following are equivalent:
    \begin{itemize}
        \item[(i)] $(G,T)$ has property~$(\ast_k)$ for some $k\geq 2$;
        \item[(ii)] $(G,T)$ has property $(\ast_2)$;
        \item[(iii)] $\omega(a)\geq 3$ or $\omega(\bar{a})\geq 3$ for every $a\in E\Gamma$.
    \end{itemize}
    Moreover, $(G,T)$ has property $(\ast_1)$ if and only if $\omega(a)\geq 3$ for every $a\in E\Gamma$.
\end{cor}
By Remark~\ref{rem:lift}, the hypothesis that $\omega(E\Gamma)\subseteq \Z_{\geq 2}$ guarantees that all paths in $\Gamma$ can be lifted to a geodesic. 
\begin{proof}
Given $a,b\in E\Gamma$ with $t(a)=o(b)$, note that
\begin{equation}\label{eq:Ne=1}
\Ne(a,b)=\omega(b)-\mathbbm{1}_{\{\bar{a}\}}(b)=1\,\Longleftrightarrow\,b=\bar{a}\text{ and }\omega(b)=2.
\end{equation}
The equivalences (i)$\Leftrightarrow$(iii) and (ii)$\Leftrightarrow$(iii) now follow from~\eqref{eq:Ne=1} and
Proposition~\ref{prop:prop*W}. 
\end{proof}

\begin{prop}\label{prop:prop*P}
    Let $(G,T)$ be a (P)-closed action on a tree with associated local action diagram $\Delta$, and let $\iota$ be an inversion on $\Delta$. Assume that the standard weight $\caW$ on $\Delta$ takes values in~$\Z_{\geq 1}$, and that $|X_a|\geq 2$ for every $a\in E\Gamma$. Let also $k\geq 1$. Then $(G,T)$ has property ($\ast_k$) if, and only if, every reduced path $\xi$ in~$(\Delta,\iota)$ of length $k+1$ has $\caW(\xi)\geq 2$.
\end{prop}

\begin{proof}
    Let $k\geq 1$ and consider a geodesic $\frp=(e_1,\ldots, e_{l+k})$ in~$T$ with $l\geq 1$. By Lemma~\ref{lem:propPbis} and Corollary~\ref{cor:Wgeod}, there is~$g\in G_{o(e_l)}$ such that
    \begin{equation*}
    \begin{split}
        |G_{(e_1,\ldots, e_l)}\cdot (e_{l+1},\ldots, e_{l+k})| & =|G_{e_l}\cdot (e_{l+1},\ldots, e_{l+k})|\\
        & = \caW(\caL(g\cdot (e_l,\ldots, e_{l+k}))).
    \end{split}
    \end{equation*}
    Since $\caL(g\cdot (e_l,\ldots, e_{l+k}))$ is a reduced path in $(\Delta,\iota)$ (cf.~Lemma~\ref{lem:geodTD}), this yields the "if" part of the statement. For the "only if" part, let $\xi=(x_1,\ldots, x_{k+1})\in \calP_{(\Delta,\iota)}$ and $e\in ET^+$ such that $\caL(e)=x_1$. By~Lemma~\ref{lem:geodTD}, there is a geodesic $\frp=(e_1,\ldots, e_{k+1})$ in~$T$ with $e_1=e$ such that $\caL(\frp)=\xi$. By Proposition~\ref{prop:subsizeP} we conclude that
    \begin{equation*}
        \caW(\xi)=|G_{e_1}\cdot (e_2,\ldots, e_k)|\geq 2.\qedhere
    \end{equation*}
\end{proof}

\subsection{Convergence properties}\label{sus:conv}
The main goal of what follows is to study the double-coset property and the polynomial double-coset growth of triples $(G,G_{t_1},G_{t_2})$, where $G$ is a group acting on a tree~$T$ as in Convention~\ref{conv:action} and $t_1,t_2\in T$.
\begin{lem}\label{lem:preconv}
    Consider a group action on a tree $(G,T)$.
   \begin{itemize}
       \item[(i)] Let $C:=\sup_{e,f\in ET\,:\,t(e)=o(f)}|G_e\cdot f|$ be finite. Then, for every geodesic $\frp=(e_1,\ldots, e_l)$ in~$T$ of length $l\geq 1$, we have
    $|G_{e_1}:G_{\frp}|\leq C^{l-1}$.
       \item[(ii)] Let $(G,T)$ have property $(\ast_k)$ for some $k\geq 1$. Then, for every geodesic $\frp=(e_1,\ldots, e_l)$ in~$T$ of length $l\geq 1$, we have $|G_{e_1}:G_{\frp}|\geq 2^{\frac{l-k}{k}}$.
   \end{itemize}
\end{lem}

\begin{proof}
  Let $\frp=(e_1,\ldots,e_l)$ be a geodesic in~$T$ of length $l\geq 1$. Arguing as for~\eqref{eq:orb2} we have
    \begin{equation}\label{eq:splitGe}
        |G_{e_{1}}:G_{\frp}|=\prod_{i=1}^{l-1}|G_{(e_1,\ldots,e_i)}\cdot e_{i+1}|.
    \end{equation}
    Since $|G_{(e_1,\ldots, e_i)}\cdot e_{i+1}|\leq |G_{e_i}\cdot e_{i+1}|$ for every $1\leq i\leq l-1$, we obtain~(i).
    
    To prove~(ii), we may assume that $l\geq k+1$.
    We claim that
    \begin{equation}\label{eq:splitGe2}
       |G_{e_1}:G_{\frp}|^{k} \stackrel{\eqref{eq:splitGe}}{=}\prod_{i=1}^{l-1}|G_{(e_1,\ldots, e_i)}\cdot e_{i+1}|^{k}\geq \prod_{i=1}^{l-k}\prod_{j=i}^{k+i-1}|G_{(e_1,\ldots, e_{j})}\cdot e_{j+1}|.
    \end{equation}
    To prove the latter inequality in~\eqref{eq:splitGe2}, set $A_j=|G_{(e_1,\ldots, e_j)}\cdot e_{j+1}|$ for every $1\leq j\leq l-1$. Then the product on the right-hand side of~\eqref{eq:splitGe2} becomes 
    \begin{equation*}
        \prod_{i=1}^{l-k}\prod_{j=i}^{k+i-1}A_j=\prod_{j=1}^{l-1}\prod_{i=\max\{1,j-k+1\}}^{\min\{j,l-k\}}A_j=\prod_{j=1}^{l-1}A_j^{\alpha_j},
    \end{equation*}
    where, for every $1\leq j\leq l-1$, 
    \begin{equation*}
        \alpha_j=|\{i\in \Z_{\geq 0}\,:\, \max\{1,j-k+1\}\leq i\leq\min\{j,l-k\}\}|.
    \end{equation*}
    It remains to show that $\alpha_j\leq k$ for every $1\leq j\leq l-1$.
    If $j\leq k-1$, then 
    $$\alpha_j=|\{i\in\Z_{\geq 0}\,:\, 1\leq i\leq \min\{j,l-k\}\}|\leq j<k.$$ If $k\leq j\leq l-k$, then $\alpha_j=|\{i\,:\,j-k+1\leq i\leq j\}|=k$. Finally, if $j\geq k$ and $j\geq l-k$ then $\alpha_j=|\{i\,:\, j-k+1\leq j\leq l-k\}|=l-j\leq k$. Hence~\eqref{eq:splitGe2} holds. By~\eqref{eq:splitGe2}, the orbit-stabiliser theorem and the fact that $(G,T)$ has property~$(\ast_k)$, we conclude that
    \begin{equation}\label{eq:splitGe3}
    \begin{split}
       |G_{e_1}:G_{\frp}|^{k} & \geq \prod_{i=1}^{l-k}\prod_{j=i}^{k+i-1}|G_{(e_1,\ldots, e_j)}:G_{(e_1,\ldots, e_{j+1})}|=\\
       & =\prod_{i=1}^{l-k}|G_{(e_1,\ldots, e_i)}\cdot (e_{i+1},\ldots, e_{i+k})|\geq 2^{l-k}.\qedhere
       \end{split}
    \end{equation}
\end{proof}

Let $(G,T)$ be a group action on a tree and consider $t_1,t_2\in T$ such that $|G_{t_1}gG_{t_2}/G_{t_2}|<\infty$ for every $g\in G$. For $i\in\{1,2\}$, set $T_i=\{t_i\}$ if $t_i\in VT$ and $T_i=\{t_i,\bar{t}_i\}$ if $t_i\in ET$. By~Fact~\ref{fact:cosets} and Lemma~\ref{lem:subsize}, we have
  \begin{equation}\label{eq:bngeo}
      b_n(G,G_{t_1},G_{t_2})=\Big|\Big\{\frp\in \Geod_T(T_1\to G\cdot T_2)\,:\,|G_{t_1}:G_{\frp}|=n\Big\}\Big|,\quad\forall\,n\geq 1.
  \end{equation}

\begin{prop}\label{prop:conv*}
    Let $(G,T)$ be a group action on a locally finite leafless tree with finite quotient graph, and set $M:=\sup_{v\in VT}|o^{-1}(v)|$. If $(G,T)$ has property~$(\ast_k)$ for some $k\geq 1$, then $3\leq M<\infty$ and, for all~$t_1,t_2\in T$ and~$n\geq 1$, 
    \begin{equation*}
        a_n(G,G_{t_1}, G_{t_2})=O(n^{k\cdot\log(M-1)-1}).
    \end{equation*}
    In particular, for all $t_1,t_2\in T$, the triple $(G,G_{t_1}, G_{t_2})$ has polynomial double-coset growth.
\end{prop}

\begin{proof}
    Let $t_1, t_2\in T$ and consider $T_1$ and $T_2$ as defined before~\eqref{eq:bngeo}. We first prove that $3\leq M<\infty$. Indeed, $|o^{-1}(v)|=|o^{-1}(g\cdot v)|$ for all $v\in VT$ and $g\in G$. Hence, since $T$ is locally finite and $G\backslash VT$ is finite, we have $M<\infty$. Moreover, $M\geq 2$ because $T$ has no leaves. Since $(G,T)$ has property~$(\ast_k)$, then $T$ cannot be a bi-infinite line and then $M\geq 3$.
    
Given $l\geq 0$ and $t\in T$, notice that the number of geodesics $\frp$ from $t$ in $T$ with $\ell(\frp)=l$ is $\leq 1$ if $l=0$, and it is $\leq M(M-1)^{l-1}$ otherwise.
 Hence, by~\eqref{eq:bngeo} and~Lemma~\ref{lem:preconv}(ii), the following holds for every $n\geq 1$:
  \begin{equation*}\label{eq:estbn}
  \begin{split}
    b_n(G,G_{t_1}, G_{t_2}) & =|\{\frp\in \Geod_T(T_1\to T)\,:\, |G_{t_1}:G_\frp|=n\}|\\
    & \leq |\{\frp\in \Geod_T(T_1\to T)\,:\, \ell(\frp)\leq \lfloor k\cdot \log_2 n\rfloor + k\}|\\
      & =\sum_{t\in T_1}\hspace{-0.1cm}|\{\frp\in \Geod_T(t\to T)\,:\, \ell(\frp)\leq \lfloor k\cdot \log_2 n\rfloor +k\}|\\
      & \leq \sum_{t\in T_1}\Bigg(1+\sum_{l=1}^{\lfloor k\cdot \log_2 n\rfloor + k}M(M-1)^{l-1}\Bigg)\\
       & =|T_1|\cdot \frac{M(M-1)^{\lfloor k\cdot \log_2 n\rfloor+k}-2}{M-2}.
  \end{split}
  \end{equation*}
  Hence, 
  $$b_n(G,G_{t_1}, G_{t_2})=O((M-1)^{k\cdot \log n})=O(n^{k\cdot \log(M-1)}).$$
  The latter claim of the statement now follows from~\eqref{eq:bnan}.
\end{proof}

\begin{thm}\label{thm:convWLIT}
    Let $(G,T)$ be a weakly locally $\infty$-transitive group action on a locally finite tree with finite quotient graph $\Gamma$. Assume that the standard edge weight $\omega$ on~$\Gamma$ takes values in~$\Z_{\geq 2}$. Then the following conditions are equivalent for all $t_1,t_2\in T$:
    \begin{itemize}
        \item[(i)] $(G,G_{t_1},G_{t_2})$ has the double-coset property;
        \item[(ii)] $(G,G_{t_1},G_{t_2})$ has polynomial double-coset growth;
        \item[(iii)] $(G,T)$ has property~$(\ast_k)$ for some $k\geq 1$;
        \item[(iv)] for every $a\in E\Gamma$, one has $\omega(a)\geq 3$ or $\omega(\bar{a})\geq 3$;
        \item[(v)] for every $e\in ET$ and every end~$\xi$ of~$T$, one has $G_e\not\subseteq G_\xi$.
    \end{itemize}
\end{thm}
\begin{proof}
    Clearly, (ii)$\Rightarrow$(i). Since~$\omega(E\Gamma)\subseteq \Z_{\geq 2}$, the tree~$T$ has no leaves. Hence, Proposition~\ref{prop:conv*} and Corollary~\ref{cor:*WLIT} imply (iii)$\Rightarrow$(ii) and (iii)$\Leftrightarrow$(iv), respectively. Moreover, Lemma~\ref{lem:opends} yields (iii)$\Rightarrow$(v). It remains to prove that $\lnot$(iv) implies both $\lnot$(i) and $\lnot$(v).

    Assume that there is $a\in E\Gamma$ such that $\omega(a)=\omega(\bar{a})=2$, and let $\frp=(a,\bar{a})$. Then $\Ne(\frp^d)=(\omega(\bar{a})-1)^d(\omega(a)-1)^{d-1}=1$ for every $d\geq 1$. In the following, $\pi\colon T\to \Gamma$ denotes the quotient map.

    We first prove $\lnot$(v). Let $e\in ET$ with $\pi(e)=a$. By Remark~\ref{rem:lift}, for every $d\geq 1$ there is a unique geodesic $\tilde{\frp}_d$ starting from~$e$ and with~$\pi(\tilde{\frp}_d)=\frp^d$. By uniqueness, for all $d'<d$, the geodesic $\tilde{\frp}_{d'}$ determines the first~$2d'$ edges of~$\tilde{\frp}_d$. Since~$\omega(E\Gamma)\subseteq \Z_{\geq 2}$, there is a ray~$r$ whose first $2d$ edges are the ones of $\tilde{\frp}_d$, for all~$d\geq 1$. Since $\Ne(\frp^d)=1$, $G_e$ fixes $\tilde{\frp}_d$ for every $d\geq 1$. Hence, $G_e$ fixes~$r$ pointwise. As a ray in~$T$ belongs to the end~$\xi$ precisely when it 
 eventually equals~$r$, we have~$G_e\subseteq G_\xi$. 
    
    To prove $\lnot$(i), consider two paths $\frq_1=(a_1,\ldots, a_h)$ and $\frq_2=(b_1,\ldots, b_k)$ of positive length in~$\Gamma$ from $\pi(t_1)$ to $o(a)$ and from $o(a)$ to $\pi(t_2)$, respectively. For every $d\geq 1$ we have
    \begin{equation*}
    \begin{split}
        \Ne(\frq_1\cdot \frp^d\cdot \frq_2)& =\Ne(\frq_1)\Ne(a_h,a)\Ne(\frp^d) \Ne(\bar{a},b_1)\Ne(\frq_2)\\
        & =\Ne(\frq_1)\Ne(a_h,a)\Ne(\bar{a},b_1)\Ne(\frq_2)=:N\geq 1.
        \end{split}
    \end{equation*}
   Since $\omega$ takes values in~$\Z_{\geq 2}$, for every $d\geq 1$ there is $\tilde{\frq}_d\in \Geod_T(T_1\to G\cdot T_2)$ satisfying $\pi(\tilde{\frq}_d)=\frq_1\cdot \frp^d\cdot \frq_2$ (cf.~Remark~\ref{rem:lift}). By~Proposition~\ref{prop:subWLIT}, for every $d\geq 1$ we have
    $$|G_{t_1}:G_{\tilde{\frq}_d}|=\left\{
    \begin{array}{cl}
    \Ne(\frq_1\cdot \frp^d\cdot \frq_2)=N, & \text{if }t_1\in ET;\\
    \Nv(\frq_1\cdot \frp^d\cdot \frq_2)=\omega(a_1)\cdot N=:N', & \text{if }t_1\in VT.
    \end{array}
    \right.
    $$
    Since $\tilde{\frq}_d\neq \tilde{\frq}_{d'}$ for all $d\neq d'$, by~\eqref{eq:bngeo} we conclude that
     $b_N(G,G_{t_1},G_{t_2})=\infty$ if $t_1\in ET$, and $b_{N'}(G,G_{t_1}, G_{t_2})=\infty$ if $t_1\in VT$.
\end{proof}

\begin{thm}\label{thm:convP}
    Let $(G,T)$ be a (P)-closed group action on a locally finite tree. Assume that the quotient graph is finite and its standard edge weight takes values in~$\Z_{\geq 2}$. Then the following conditions are equivalent for all $t_1,t_2\in T$:
    \begin{itemize}
        \item[(i)] $(G,G_{t_1}, G_{t_2})$ has the double-coset property;
        \item[(ii)] $(G,G_{t_1}, G_{t_2})$ has polynomial double-coset growth;
        \item[(iii)] $(G,T)$ has property~$(\ast_k)$ for some $k\geq 1$;
        \item[(iv)] for every $e\in ET$ and every end~$\xi$ in $T$, one has $G_e\not\subseteq G_\xi$.
    \end{itemize} 
\end{thm}
\begin{proof}
    The implication (ii)$\Rightarrow$(i) is immediate. Since the standard edge weight on~$G\backslash T$ takes values in~$\Z_{\geq 2}$, the tree~$T$ has no leaves.
    Hence, (iii)$\Rightarrow$(ii) follows from  Proposition~\ref{prop:conv*}, and (iii)$\Rightarrow$(iv) from Lemma~\ref{lem:opends}. It remains to prove that $\lnot$(iii) implies both $\lnot$(i) and $\lnot$(iv).
    
    Assume that $(G,T)$ does not have property~$(\ast_k)$ for every $k\geq 1$. Let $\Delta$ be the local action diagram associated to $(G,T)$ and consider an inversion~$\iota$ in~$\Delta$. By~Proposition~\ref{prop:prop*P}, there is a reduced path $(x_1,\ldots, x_k)$ in~$(\Delta,\iota)$ of length $k\geq |X|^2+2$ such that $\caW(x_1,\ldots, x_k)=1$, i.e., $\caW(x_i,x_{i+1})=1$ for every $1\leq i\leq k-1$. Since $k-1\geq |X|^2+1$, we have $(x_i,x_{i+1})=(x_{i+l},x_{i+1+l})$ for some $i,l\geq 1$. Set $\eta:=(x_j)_{i\leq j\leq i+l-1}$. For every $d\geq 1$ the power $\eta^d$ is a reduced path in~$(\Delta,\iota)$ satisfying
    \begin{equation*}
        \caW(\eta^d)=\caW(\eta)^d\cdot \caW(x_{i+l-1},x_i)^{d-1}=\caW(\eta)^d\cdot \caW(x_{i+l-1},x_{i+l})^{d-1}=1.
    \end{equation*}

    We first prove $\lnot$(iv). 
    Without loss of generality (cf.~Remark~\ref{rem:redUD}), $G=U(\Delta,\iota,c_0)$ and $T=T(\Delta,\iota,c_0)$, where $c_0\in V\Gamma$ is such that~$x_1\in X_{c_0}$. 
    Then $e=((x_j)_{1\leq j\leq i-1},(x_{j})_{1\leq j\leq i})\in ET^+$. For every $d\geq 1$, denote by $\tilde{\frp}_d$ the unique geodesic in~$T$ starting at~$e$ and such that~$\caL(\tilde{\frp}_d)=\eta^d$ (it exists by Lemma~\ref{lem:geodTD}). By uniqueness, for every~$d\geq 2$, $\tilde{\frp}_d$ prolongs~$\tilde{\frp}_{d-1}$.
   Since~$\caW(\eta^d)=1$ and the edges of~$\tilde{\frp}_d$ belong to~$ET^+$, $G_e=G_{\tilde{\frp}_d}$ for every~$d\geq 1$ (cf.~Proposition~\ref{prop:subsizeP})
   and there is an end~$\xi$ of~$T$ such that $G_e\subseteq G_\xi$. 
    
    To prove $\lnot$(i), 
    without loss of generality (cf.~Remark~\ref{rem:redUD}), we may assume that $G=U(\Delta,\iota,c_0)$, $T=T(\Delta,\iota,c_0)$, either~$t_1$ is the root of~$T$ (if $t_1\in VT$) or~the origin of~$t_1$ is the root of~$T$ (if~$t_1\in ET$), and $\Geod_T(t_1\to G\cdot t_2)\neq \emptyset$.

Choose arbitrary reduced paths of positive length in~$(\Delta,\iota)$, say $\xi_1=(y_1,\ldots, y_h)$ and $\xi_2=(z_1,\ldots, z_r)$, such that $\xi_1\cdot x_1\in \calP_{(\Delta,\iota)}(X_{\pi(t_1)}\to x_1)$ and $x_{i+l-1}\cdot \xi_2\in\calP_{(\Delta,\iota)}(x_{i+l-1}\to X_{\pi(t_2)})$. Such reduced paths exist because the standard edge weight of $G\backslash T$ takes values in~$\Z_{\geq 2}$ (cf.~Remark~\ref{rem:lift}).
    For every $d\geq 1$, the path $\xi_1\cdot \eta^d\cdot \xi_2$ is reduced in~$(\Delta,\iota)$ and 
    \begin{equation*}
    \begin{split}
        \caW(\xi_1\cdot \eta^d\cdot \xi_2)& =\caW(\xi_1)\caW(y_h,x_i)\caW(\eta^d)\caW(x_{i+l-1},z_1)\caW(\xi_2)\\
        & =\caW(\xi_1)\caW(y_h,x_i)\caW(x_{i+l-1},z_1)\caW(\xi_2)=:N.
    \end{split}
    \end{equation*}
    
    By~Lemma~\ref{lem:geodTD}, for every~$d\geq 1$ there is a unique geodesic~$\tilde{\frq}_d\in \Geod_T(t_1\to G\cdot t_2)$ with $\caL(\tilde{\frq}_d)=\xi_1\cdot \eta^d\cdot \xi_2$ and such that all the~$\tilde{\frq}_d$'s have a common first edge, say~$e_1\in ET^+$.
  By Remark~\ref{rem:geoE+}, all edges of every~$\tilde{\frq}_d$ are in~$ET^+$. 
     By~Proposition~\ref{prop:subsizeP}, for every $d\geq 1$ we have
    $$|G_{t_1}:G_{\tilde{\frq}_d}|=\left\{
    \begin{array}{cl}
        \caW(\caL(\tilde{\frq}_d))=N, & \text{if }t_1\in ET;\\
        \omega(\pi(e_1))\cdot\caW(\caL(\tilde{\frq}_d))=:N', & \text{if }t_1\in VT.
    \end{array}
    \right.$$
    Since $\tilde{\frq}_d\neq\tilde{\frq}_{d'}$ for all $d\neq d'$, from~\eqref{eq:bngeo} we conclude that $b_N(G,G_{t_1},G_{t_2})=\infty$ if $t_1\in ET$, and $b_{N'}(G,G_{t_1}, G_{t_2})=\infty$ if $t_1\in VT$.
\end{proof}

\subsection{Explicit formulae: the weakly locally $\infty$-transitive case}\label{sus:DCwlit}
\begin{settWLIT}\label{settWLIT}
    Let $(G,T)$ be a weakly locally $\infty$-transitive group action on a locally finite tree as in Convention~\ref{conv:action} and with quotient map $\pi\colon T\to \Gamma=G\backslash T$. Assume that~$\Gamma$ is finite, and that its standard edge weight $\omega$ takes values in~$\Z_{\geq 2}$ (so $T$ has no leaves) and satisfies $\omega(a)\geq 3$ or $\omega(\bar{a})\geq 3$ for every $a\in E\Gamma$. Let also $\Ne=\Ne^\omega$ and $\Nv=\Nv^\omega$ be as in Definition~\ref{defn:N}.
\end{settWLIT}

Setting~\hyperref[settWLIT]{[WLIT]} guarantees that the series defining $\zeta_{G,G_{t_1},G_{t_2}}(s)$ converges at some $s\in \C$, for all $t_1,t_2\in T$ (cf.~Theorem~\ref{thm:convWLIT}).

\begin{prop}\label{prop:zetaWLIT}
    Suppose Setting~\hyperref[settWLIT]{\emph{[WLIT]}}, and let $t\in T$ with $\pi(t)=u$. 
    Then, for every $v\in VT$ with $\pi(v)=c$, we have
    \begin{equation*}
            \zeta_{G,G_v, G_t}(s)=\sum_{\frp\in \calP_\Gamma(c\to U)}\Nv(\frp)^{-s}.
        \end{equation*}
   Moreover, for every $e\in ET$ with $\pi(e)=a$, we have
        \begin{equation*}
            \zeta_{G,G_e, G_t}(s)=\varepsilon_a(u)+\sum_{\frp\in \calP_\Gamma(A\to U),\atop \ell(\frp)\geq 2}\Ne(\frp)^{-s},
        \end{equation*}
    where $\varepsilon_a(u)=\mathbbm{1}_{\{o(a),t(a)\}}(u)$ if $u\in V\Gamma$ and $\varepsilon_a(u)=\mathbbm{1}_{A}(u)$ if $u\in E\Gamma$.
\end{prop}

\begin{proof}
  It is a direct consequence of Proposition~\ref{prop:dcgeoWLIT} (recalling~Remark~\ref{rem:lift}) and Proposition~\ref{prop:subWLIT}.
\end{proof}

Proposition~\ref{prop:zetaWLIT} suggests the following generalisation.
\begin{defn}\label{defn:Zu}
    Let $\Gamma$ be a non-empty graph with an edge weight $\omega\colon E\Gamma\to\Z_{\geq 2}$, and let $c\in V\Gamma$, $a\in E\Gamma$ and $u\in \Gamma$. Define the following formal Dirichlet series:
  \begin{equation*}
  \begin{split}
  \caZ_{\Gamma,c\to u}(s) &:=\sum_{\frp\in \calP_\Gamma(c\to U)}\Nv(\frp)^{-s},\\
   \caZ_{\Gamma,a\to u}(s) &:=\varepsilon_a(u)+\sum_{\frp\in \calP_\Gamma(A\to U),\atop \ell(\frp)\geq 2}\Ne(\frp)^{-s},
   \end{split}
  \end{equation*}
  where $\varepsilon_a(u)=\mathbbm{1}_{\{o(a),t(a)\}}(u)$ if $u\in V\Gamma$ and $\varepsilon_a(u)=\mathbbm{1}_{A}(u)$ if $u\in E\Gamma$.
\end{defn}

\begin{rem}
    In~Definition~\ref{defn:Zu}, we may assume that $\Gamma$ is connected. Indeed, given $u_1,u_2\in \Gamma$, if there is a connected component $\Lambda$ of $\Gamma$ containing both $u_1$ ad $u_2$, then
    $\caZ_{\Gamma,u_1\to u_2}(s)=\caZ_{\Lambda,u_1\to u_2}(s)$.
    If such a connected component does not exist, the function $\caZ_{\Gamma,u_1\to u_2}(s)$ is identically zero.
\end{rem}

\begin{rem}\label{rem:zetaZ}  
By~Proposition~\ref{prop:zetaWLIT}, for all $t_1, t_2\in T$ with $\pi(t_1)=u_1$ and $\pi(t_2)=u_2$ we have
    \begin{equation*}
        \zeta_{G,G_{t_1}, G_{t_2}}(s)=\caZ_{\Gamma,u_1\to u_2}(s).
    \end{equation*}
\end{rem}

In view of an explicit formula for $\caZ_{\Gamma,u_1\to u_2}(s)$, we introduce the following linear operator.

\begin{defn}\label{defn:Es}
   Let $\Gamma$ be a non-empty graph with an edge weight $\omega\colon E\Gamma\to\Z_{\geq 2}$. Let $\C\llbracket E\Gamma\rrbracket$ be the complex vector space of all formal sums $\sum_{a\in E\Gamma}\gamma_a a$, where $\gamma_a\in \C$ for every $a\in E\Gamma$. For each $u\in \Gamma$, define $e_u\in \C\llbracket E\Gamma\rrbracket$ as follows:
    \begin{equation}\label{eq:eu}
        e_u:=\left\{
        \begin{array}{cl}
            u, & \text{if }u\in E\Gamma; \\
            \sum_{a\in t^{-1}(u)}a, & \text{if }u\in V\Gamma.
        \end{array}
        \right.
    \end{equation}

    For every $s\in \C$, the \emph{Bass operator} $\caE(s)=\caE^{(\Gamma,\omega)}(s)\colon \C\llbracket E\Gamma\rrbracket\to\C\llbracket E\Gamma\rrbracket$ \emph{of $\Gamma$ at $s\in \C$} is the linear extension of the following assignment:
    \begin{equation}\label{eq:Es1}
        \caE(s)(a):=\sum_{b\in E\Gamma}\caE(s)(a,b)b,\quad\forall\,a\in E\Gamma
    \end{equation}
    where, for all $a,b\in E\Gamma$, 
    \begin{equation}\label{eq:Es2}
        \caE(s)(a,b):=\left\{
        \begin{array}{cl}
            \Ne(a,b)^{-s}, &  \text{if }t(a)=o(b);\\
            0, & \text{otherwise.}
        \end{array}
        \right.
    \end{equation}
\end{defn}
\begin{notat}\label{notat:Es}
\begin{itemize}
    \item[(i)] We will usually write $\caE(s)$ instead of $\caE^{(\Gamma,\omega)}(s)$. If we want to specify $\Gamma$ (but $\omega$ is clear from the context), we write $\caE^\Gamma(s)$ instead of $\caE(s)$ or $\caE^{(\Gamma,\omega)}(s)$, and $e^\Gamma_u$ instead of $e_u$, for all $u\in \Gamma$. 
    \item[(ii)]  We implicitly set a total order on $E\Gamma$. Thus, provided $|E\Gamma|<\infty$, we can regard $\caE(s)$ as a $|E\Gamma|$-dimensional matrix $[\caE(s)(a,b)]_{a,b\in E\Gamma}$ with complex entries, and the $e_u$'s in~\eqref{eq:eu} as row vectors in~$\C^{|E\Gamma|}$. For all $a,b\in E\Gamma$, note that $e_a\caE(s)e_b^t=\caE(s)(a,b)$.
\end{itemize}
\end{notat}

The term "Bass operator" is taken after~\cite[Definition~3.10]{deit:prim}. The reader is referred to Section~\ref{sus:ih-1} for further connections with~\cite{deit:prim}.

\begin{rem}\label{rem:Es}
    Let $\caE(s)$ be as in Definition~\ref{defn:Es} and assume that $\Gamma$ is finite.
    For every $n\geq 1$, let $\caE(s)^n$ be the $n$-th power of $\caE(s)$, and $\caE(s)^0$ be the identity operator on $\C\llbracket E\Gamma\rrbracket$. Then, for all $n\geq 0$ and $a,b\in E\Gamma$, we observe that
    \begin{equation*}
        \caE(s)^n(a,b)=\sum_{\frp\in \calP_\Gamma(a\to b)\atop \ell(\frp)=n+1}\Ne(\frp)^{-s}.
    \end{equation*}
    If $n\leq 1$, it is clear. For every $n\geq 2$, it suffices to observe that
    \begin{equation*}
    \caE(s)^n(a,b)=\sum_{a_2,\ldots,\, a_{n}\in E\Gamma}\caE(s)(a, a_2)\cdot \ldots \cdot \caE(s)(a_{n}, b).
    \end{equation*}
\end{rem}
As we did in Setting~\hyperref[settWLIT]{[WLIT]}, we fix a setting which ensures that the series $\caZ_{\Gamma,u_1\to u_2}(s)$ as in Definition~\ref{defn:Zu} converges at some $s\in \C$, for all $u_1,u_2\in \Gamma$.
\begin{settG}\label{settG}
    Let $\Gamma$ be a finite connected graph with~$E\Gamma\neq\emptyset$ and with an edge weight $\omega\colon E\Gamma\to \Z_{\geq 2}$ satisfying $\omega(a)\geq 3$ or $\omega(\bar{a})\geq 3$, for every $a\in E\Gamma$. \\
    \emph{Convention}: Every subgraph $\Gamma'\subseteq \Gamma$ is endowed "by default" with the restricted edge weight from $\omega$.
\end{settG}

\begin{thm}\label{thm:meroWLIT}
    Let $(\Gamma,\omega)$ be an edge-weighted graph satisfying Setting~\hyperref[settG]{\emph{[$\Gamma$]}}. Then, for all $u,w\in \Gamma$,
    \begin{equation}\label{eq:Gammero}
        \caZ_{\Gamma,u\to w}(s)=\frac{\det(I-\caE(s)+\caU_{u,w}(s))}{\det(I-\caE(s))}+\epsilon_u(w),
    \end{equation}
    where $I$ is the identity matrix in~$\Mat_{|E\Gamma|}(\C)$ and
    \begin{equation*}
        \caU_{u,w}(s):=\left\{
        \begin{array}{ll}
           \sum_{a\in o^{-1}(u)}\omega(a)^{-s}e_w^t\cdot e_a,  & \text{if }u,w\in V\Gamma;\\
           \sum_{a\in o^{-1}(u)}\omega(a)^{-s}(e_w+e_{\bar{w}})^t\cdot e_a\caE(s), & \text{if }u\in V\Gamma,\,w\in E\Gamma;\\
           e_w^t\cdot (e_u+e_{\bar{u}})\caE(s), & \text{if }u\in E\Gamma,\,w\in V\Gamma;\\
           (e_w+e_{\bar{w}})^t\cdot(e_u+e_{\bar{u}})\caE(s), & \text{if }u,w\in E\Gamma;
        \end{array}
        \right.
    \end{equation*}
    \begin{equation*}
        \epsilon_u(w):=\left\{
        \begin{array}{ll}
           \mathbbm{1}_{\{u\}}(w)-1,  & \text{if }u,w\in V\Gamma;\\
           \mathbbm{1}_{o^{-1}(u)}(w)\cdot \omega(w)-1, & \text{if }u\in V\Gamma,\,w\in E\Gamma;\\
           \mathbbm{1}_{\{o(u),t(u)\}}(w)-1, & \text{if }u\in E\Gamma,\,w\in V\Gamma;\\
           \mathbbm{1}_{\{u, \bar{u}\}}(w)-1, & \text{if }u,w\in E\Gamma.
        \end{array}
        \right.
    \end{equation*}
In particular, for all $u,w\in \Gamma$ the function $\caZ_{\Gamma,u\to w}(s)$ is meromorphic over~$\C$.
\end{thm}

The proof of Theorem~\ref{thm:meroWLIT} makes use of the following fact.
\begin{fact}[Matrix Determinant Lemma, \protect{cf.~\cite{matdet}}]\label{fact:mdl}
    Consider $A\in \mathrm{GL}_n(\C)$ with adjugate matrix $\mathrm{adj}(A)$, and let $u,v\in\C^n$ be row vectors. Then,
    \begin{equation*}\label{eq:mld2}
        \frac{\det(A+u^t\cdot v)}{\det(A)}=1+vA^{-1}u^t.
    \end{equation*}
\end{fact}

\begin{proof}[Proof of Theorem~\ref{thm:meroWLIT}] 
   Let $s\in \C$ be such that $\sum_{n=0}^\infty\caE(s)^n$ converges. Recall that $\sum_{n=0}^\infty\caE(s)^n=(I-\caE(s))^{-1}$.
  By Remark~\ref{rem:Es}, if $u\in V\Gamma$ then
  \begin{equation*}
  \begin{split}
      & \caZ_{\Gamma,u\to w}(s)=\\
      & =\left\{
      \begin{array}{ll}
          \displaystyle{\mathbbm{1}_{\{u\}}(w)+\sum_{n=0}^{\infty}\sum_{a\in o^{-1}(u)}\hspace{-0.2cm}\omega(a)^{-s}e_a\caE(s)^ne_w^{t}}, & \hspace{-0.2cm}\text{if }w\in V\Gamma; \\
          \displaystyle{\mathbbm{1}_{o^{-1}(u)}(w)\cdot\omega(w)^{-s}+\sum_{n=1}^{\infty}\sum_{a\in o^{-1}(u)}\hspace{-0.2cm}\omega(a)^{-s}e_a\caE(s)^n(e_w+e_{\bar{w}})^t}, & \hspace{-0.2cm}\text{if }w\in E\Gamma.
      \end{array}
      \right.
      \end{split}
  \end{equation*}
  Similarly, if $u\in E\Gamma$ then
  \begin{equation*}
  \begin{split}
      \caZ_{\Gamma,u\to w}(s)=\left\{
      \begin{array}{ll}
          \displaystyle{\mathbbm{1}_{\{o(u),t(u)\}}(w)+\sum_{n=1}^{\infty}(e_u+e_{\bar{u}})\caE(s)^ne_w^t}, & \text{if }w\in V\Gamma;\\
          \displaystyle{\mathbbm{1}_{\{u,\bar{u}\}}(w)+\sum_{n=1}^{\infty}(e_u+e_{\bar{u}})\caE(s)^n(e_w^t+e_{\bar{w}}^t)}, & \text{if }w\in E\Gamma.
      \end{array}
      \right.
    \end{split}
  \end{equation*}
  We now focus on the case in which $u,w\in V\Gamma$, as the other cases are analogous. Namely, if $u,w\in V\Gamma$ then
   \begin{equation*}
      \begin{split}
         \caZ_{\Gamma,u\to w}(s)& = \mathbbm{1}_{\{u\}}(w)+\sum_{a\in o^{-1}(u)}\omega(a)^{-s}e_a \Bigg(\sum_{n=0}^{\infty}\caE(s)^n\Bigg)e_w^t\\
            & = \mathbbm{1}_{\{u\}}(w)+\Bigg(\sum_{a\in o^{-1}(u)}\omega(a)^{-s}e_a\Bigg)(I-\caE(s))^{-1}e_w^t
      \end{split}
   \end{equation*}
  and Fact~\ref{fact:mdl} applies.
\end{proof}
In view of~Section~\ref{s:euler}, we provide some explicit formulae for $\caZ_{\Gamma,u\to u}(s)$ in case that $\Gamma$ has one edge-pair. 
\begin{ex}\label{ex:1segmWLIT}
Let $\Gamma$ be a $1$-segment graph with $E\Gamma=\{a,\bar{a}\}$, $c=o(a)$ and $d=t(a)$. Set $\omega(a):=\alpha+1$ and $\omega(\bar{a}):=\beta+1$, where $\alpha,\beta\in \Z_{\geq 1}$ with $\alpha\geq 2$ or $\beta\geq 2$. 
Set the order $\leq$ on $E\Gamma$ such that $a<\bar{a}$, and identify $\C\llbracket E\Gamma\rrbracket$ with $\C^2$, $e_a=e_d$ with $(1,0)$ and $e_{\bar{a}}=e_c=(0,1)$.
Then, for every $s\in \C$,
\begin{equation*}
\begin{split}
    \caE(s) & =\begin{bmatrix}
        \caE(s)(a,a) & \caE(s)(a,\bar{a})\\
        \caE(s)(\bar{a},a) & \caE(s)(\bar{a},\bar{a})
    \end{bmatrix}=\begin{bmatrix}
        0 & \beta^{-s}\\
        \alpha^{-s} & 0
    \end{bmatrix};\\
    \caU_{c,c}(s)& =(\alpha+1)^{-s}e_c^t\cdot e_a=\begin{bmatrix}
        0 & 0\\
        (\alpha+1)^{-s} & 0
    \end{bmatrix};\\
    \caU_{a,a}(s) & =(e_a+e_{\bar{a}})^t\cdot (e_a+e_{\bar{a}})\cdot \caE(s)=\begin{bmatrix}
        \alpha^{-s} & \beta^{-s}\\
        \alpha^{-s} & \beta^{-s}
    \end{bmatrix}.
\end{split}
\end{equation*}
Let $I$ be the identity matrix in~$\Mat_2(\C)$. By~Theorem~\ref{thm:meroWLIT}, 
\begin{equation}\label{eq:Z1seg}
    \begin{split}
        \caZ_{\Gamma,c\to c}(s) & =\frac{1+((\alpha+1)^{-s}-\alpha^{-s})\cdot\beta^{-s}}{1-\alpha^{-s}\beta^{-s}};\\
        \caZ_{\Gamma,a\to a}(s) & =\frac{(1+\alpha^{-s})(1+\beta^{-s})}{1-\alpha^{-s}\beta^{-s}}.
    \end{split}
\end{equation}
In particular, let $(G,T)$ be a locally $\infty$-transitive action on a locally finite tree with quotient graph $\Gamma$ and standard edge weight $\omega$ as before. By~Remark~\ref{rem:zetaZ} and~\eqref{eq:Z1seg}, we have explicit formulae for $\zeta_{G,G_v,G_v}(s)$ and $\zeta_{G,G_e,G_e}(s)$ for all $v\in VT$ with $G\cdot v=c$ and $e\in ET$ with $G\cdot e=a$.  For instance, one may take $G=\SL_2(\Q_p)$ and let~$T$ be the Bruhat--Tits tree of $G$ (cf.~Example~\ref{ex:lads}(ii)). In this case $\alpha=\beta=p$. Let $v\in VT$ be the vertex with $G_v=\SL_2(\Z_p)$, and $e\in ET$ be the edge whose pointwise stabiliser is the standard Iwahori subgroup. Then, 
    \begin{equation*}
        \zeta_{G,G_v,G_v}(s)=\frac{1+((p+1)^{-s}-p^{-s})\cdot p^{-s}}{1-p^{-2s}}\quad\text{and}\quad\zeta_{G,G_e,G_e}(s)=\frac{1+p^{-s}}{1-p^{-s}}.
    \end{equation*}
Other examples can be obtained from Example~\ref{ex:WLIT}.
Note that the formulae before agree with~\cite[Example~1.7]{ccw:zeta} in case that $G$ is the group of automorphisms of a bi-coloured tree $T$.
\end{ex}

\begin{ex}\label{ex:1loopWLIT}
Let $\Gamma$ be a $1$-bouquet of loops with $E\Gamma=\{a,\bar{a}\}$ and $c=o(a)=t(a)$. Set $\omega(a):=\alpha+1$ and $\omega(\bar{a}):=\beta+1$, for some $\alpha,\beta\in\Z_{\geq 1}$ with $\alpha\geq 2$ or $\beta\geq 2$. 
Consider the order $\leq$ on $E\Gamma$ such that $a<\bar{a}$, and identify $\C\llbracket E\Gamma\rrbracket$ with $\C^2$, $e_a$ with the vector $(1,0)$, $e_{\bar{a}}$ with $(0,1)$ and $e_c=e_a+e_{\bar{a}}$ with $(1,1)$.
Then, for every $s\in \C$,
\begin{equation*}
\begin{split}
    \caE(s) & =\begin{bmatrix}
        \caE(s)(a,a) & \caE(s)(a,\bar{a})\\
        \caE(s)(\bar{a},a) & \caE(s)(\bar{a},\bar{a})
    \end{bmatrix}=\begin{bmatrix}
        (\alpha+1)^{-s} & \beta^{-s}\\
        \alpha^{-s} & (\beta+1)^{-s}
    \end{bmatrix};\\
    \caU_{c,c}(s)& =(\alpha+1)^{-s}e_c^t\cdot e_a+(\beta+1)^{-s}e_c^t \cdot e_{\bar{a}}=\begin{bmatrix}
        (\alpha+1)^{-s} & (\beta+1)^{-s}\\
        (\alpha+1)^{-s} & (\beta+1)^{-s}
    \end{bmatrix};\\
    \caU_{a,a}(s) & =(e_a+e_{\bar{a}})^t\cdot(e_a+e_{\bar{a}})\caE(s)=\begin{bmatrix}
        (\alpha+1)^{-s}+\alpha^{-s} & (\beta+1)^{-s}+\beta^{-s}\\
        (\alpha+1)^{-s}+\alpha^{-s} & (\beta+1)^{-s}+\beta^{-s}
    \end{bmatrix}.
\end{split}
\end{equation*}
Let $I$ be the identity matrix in~$\Mat_2(\C)$. By~Theorem~\ref{thm:meroWLIT}, we have
\begin{equation}\label{eq:Z1loop}
    \begin{split}
        \caZ_{\Gamma,c\to c}(s) &=\frac{1-\Big((\alpha+1)^{-s}-\alpha^{-s}\Big)\cdot\Big((\beta+1)^{-s}-\beta^{-s}\Big)}{\Big(1-(\alpha+1)^{-s}\Big)\cdot\Big(1-(\beta+1)^{-s}\Big)-\alpha^{-s}\beta^{-s}};\\
        \caZ_{\Gamma,a\to a}(s) & =\frac{(\alpha^{-s}+1)(\beta^{-s}+1)-(\alpha+1)^{-s}(\beta+1)^{-s}}{\Big(1-(\alpha+1)^{-s}\Big)\cdot\Big(1-(\beta+1)^{-s}\Big)-\alpha^{-s}\beta^{-s}}.
    \end{split}
\end{equation}
If $\alpha=\beta$, after basic algebraic manipulations, the formulae in~\eqref{eq:Z1loop} become
\begin{equation}\label{eq:Z1loopA=B}
        \caZ_{\Gamma,c\to c}(s)=\frac{1-\alpha^{-s}+(\alpha+1)^{-s}}{1-\alpha^{-s}-(\alpha+1)^{-s}} \quad\text{and}\quad
        \caZ_{\Gamma,a\to a}(s)=\frac{1+\alpha^{-s}+(\alpha+1)^{-s}}{1-\alpha^{-s}-(\alpha+1)^{-s}}.
\end{equation}

Consider a weakly locally $\infty$-transitive group action on a locally finite tree $(G,T)$ with quotient graph $\Gamma$ and standard edge weight $\omega$. For explicit examples, see~Example~\ref{ex:WLIT}(i). By Remark~\ref{rem:zetaZ}, the computations in~\eqref{eq:Z1loop} provide explicit formulae for $\zeta_{G,G_v,G_v}(s)$ and $\zeta_{G,G_e,G_e}(s)$ whenever $v\in VT$ and $e\in ET$ satisfy $G\cdot v=c$ and $G\cdot e=a$ (or $G\cdot e=\bar{a}$), respectively. 
\end{ex}

\subsection{Explicit formulae: the (P)-closed case}\label{sus:DCPcl}
\begin{settP}\label{settP}
    Let~$\lad$ be a local action diagram based satisfying Convention~\ref{conv:lad}. Choose an inversion~$\iota$ in~$\Delta$ and $c_0\in V\Gamma$. Denote by $(T=T(\Delta,\iota,c_0),\pi,\caL)$ the standard $\Delta$-tree associated to $\iota$ and $c_0$, let $v_0$ be the root of $T$ (cf.~Section~\ref{ssus:TD}), and set $G=U(\Delta,\iota,c_0)$. Assume that the standard edge weight~$\omega$ on~$\Gamma$ and the standard edge weight~$\caW$ on~$\Delta$ take values in~$\Z_{\geq 2}$ and $\Z_{\geq 0}$, respectively. In particular, $T$ has no leaves. Finally, assume that $(G,T)$ has property~$(\ast_k)$ for some $k\geq 1$.
\end{settP}

Setting~\hyperref[settP]{[(P)-cl]} guarantees that the series defining $\zeta_{G,G_{t_1}, G_{t_2}}(s)$ converges at some $s\in \C$, for all $t_1,t_2\in T$ (cf.~Theorem~\ref{thm:convP}). 
Thanks to Remark~\ref{rem:redUD}, from now on we may focus only on the case in which $t_1\in \{v_0\}\cup o^{-1}(v_0)$ while studying $\zeta_{G,G_{t_1}, G_{t_2}}(s)$.

\smallskip 

The analogue of Proposition~\ref{prop:zetaWLIT} for (P)-closed actions is the following.

\begin{prop}\label{prop:zetageo}
    Let $(G,T)$ be as in Setting~\hyperref[settP]{\emph{[(P)-cl]}}, and let $t\in T$ with $\pi(t)=u$. Then
   \begin{equation}\label{eq:Zvpath}
       \zeta_{G,G_{v_0},G_t}(s)=\mathbbm{1}_{\{c_0\}}(u)+\hspace{-0.3cm}\sum_{a\in o^{-1}(c_0), \atop \xi\in \calP_{(\Delta, \iota)}(X_a\to X_U)}\hspace{-0.3cm}\omega(a)^{-s-1}\caW(\xi)^{-s-1}.
   \end{equation}
   Moreover, for every $e\in o^{-1}(v_0)$ with $\pi(e)=a$ and $\caL(e)=x\in X_{a}$, we have 
    \begin{equation}\label{eq:Zepath}
            \zeta_{G, G_e, G_t}(s)=\eta_a(u)+\hspace{-0.4cm}\sum_{\xi\in \calP_{(\Delta, \iota)}(x\to X_U), \atop \ell(\xi)\geq 2}\hspace{-0.6cm}\caW(\xi)^{-s-1}+\hspace{-0.6cm}\sum_{y\in X_{c_0}\setminus\{x\}, \atop \xi\in \calP_{(\Delta,\iota)}(y\to X_U),\,\ell(\xi)\geq 1}\hspace{-1cm}\caWr(x,y)^{-s-1}\caW(\xi)^{-s-1},
        \end{equation}
    where $\eta_a(u)=\mathbbm{1}_{\{o(a),t(a)\}}(u)$ if $u\in V\Gamma$ and $\eta_a(u)=\mathbbm{1}_{A}(u)$ if $u\in E\Gamma$.
\end{prop}

\begin{proof}
The statement is a direct consequence of~\eqref{eq:zetas+1}, Proposition~\ref{prop:cosPcl}, Proposition~\ref{prop:subsizeP} and~Corollary~\ref{cor:subsizP}.
\end{proof}

As in Section~\ref{sus:DCwlit}, we introduce a linear operator to express the series defining $\zeta_{G,G_{t_1}, G_{t_2}}(s)$ within Setting~\hyperref[settP]{[(P)-cl]}.

\begin{defn}\label{defn:Fs}
    Let $\lad$ be a local action diagram as in Convention~\ref{conv:lad} together with a function $\caW\colon X\times X\to \Z_{\geq 0}$ (recall that $X=\bigsqcup_{a\in E\Gamma}X_a$). Let $\C\llbracket X\rrbracket$ be the complex vector space of all formal sums $\sum_{x\in X}\gamma_x x$, where $\gamma_x\in \C$ for every $x\in X$. For every non-empty set $S\subseteq X$, set 
    $$f_S:=\sum_{x\in S}x\in \C\llbracket X\rrbracket.$$ 
    Given $s\in \C$ and for all $x\in X_a, y\in X_b$ with $a,b\in E\Gamma$, define
    \begin{equation*}
        \caF(s)(x, y):=\left\{
        \begin{array}{cl}
           \begin{array}{cl}
          \caW(x,y)^{-s}, & \text{if }t(a)=o(b)\text{ and }y\neq\iota(x); \\
          0, & \text{otherwise.}
      \end{array}  
        \end{array}
        \right.
    \end{equation*}
   The \emph{Bass operator} $\caF(s)\colon\C\llbracket X\rrbracket\longrightarrow \C\llbracket X\rrbracket$ \emph{of $(\Delta,\caW)$ at $s\in \C$} is defined by linearly extending the following assignment, for all $x\in X$:
    \begin{equation*}
        \caF(x):=\sum_{y\in X}\caF(s)(x,y)y.
    \end{equation*}
\end{defn}
\begin{notat}
    Technically, $\caF(s)$ depends on $\Delta$ and $\caW$. In our case, since $\Delta$ and $\caW$ will be always clear from the context (in particular, $\caW$ will be always the standard weight on $\Delta$), we avoid underlying this dependence. 
    
   In what follows, $X$ will be always finite and we implicitly fix a total order on $X$. In this way we can regard $\caF(s)$ and the $f_x$'s as a $|X|$-dimensional matrix $[\caF(s)(x,y)]_{x,y\in X}$ and as $|X|$-dimensional row vectors with complex entries, respectively. For all $x,y\in X$, note that $f_x\caF(s)f_y^t=\caF(s)(x,y)$.
\end{notat}

Continuing the analogy with Section~\ref{sus:subWLIT}, we observe the following.

\begin{rem}\label{rem:Fs}
    Let $\caF(s)$ be as in Definition~\ref{defn:Fs} and let~$X$ be finite. For every $n\geq 1$, let $\caF(s)^n$ be the $n$-th power of $\caF(s)$, and $\caF(s)^0$ be the identity matrix in~$\Mat_{|X|}(\C)$. For $n\geq 0$ and for all $x,y\in X$, we claim that
    \begin{equation}\label{eq:Fssum}
  \caF(s)^n(x,y)=\sum_{\xi\in \calP_{(\Delta, \iota)}(x\to y):\, \ell(\xi)=n+1}\caW(\xi)^{-s},
  \end{equation}
  where $\caW(\xi)=1$ if $\ell(\xi)=1$, and $\caW(\xi)=\prod_{i=1}^{l-1}\caW(x_i,x_{i+1})$ if $\xi=(x_1,\ldots, x_l)$ for some $l\geq 2$.
   Indeed, \eqref{eq:Fssum} is immediate if $n\leq 1$. For $n\geq 2$, one argues as in Remark~\ref{rem:Es}.
\end{rem}

\begin{thm}\label{thm:meroP}
    Let $G=U(\Delta,\iota,c_0)$ and $T=T(\Delta,\iota,c_0)$ be as in Setting~\hyperref[settP]{\emph{[(P)-cl]}}. Let $t\in T$ with $\pi(t)=u$ and let $e\in o^{-1}(v_0)$ have $\pi(e)=a$ and $\caL(e)=x$. Then, for every $r\in \{v_0,e\}$, we have
    \begin{equation*}
        \zeta_{G,G_r,G_t}(s)=\frac{\det(I-\caF(s+1)+\caY_{\pi(r),u}(s+1))}{\det(I-\caF(s+1))}+\kappa_{\pi(r)}(u),
    \end{equation*}
    where $I$ is the identity matrix in~$\Mat_{|X|}(\C)$,
    \begin{equation*}\label{eq:Y}
    \caY_{\pi(r),u}(s)=\left\{
        \begin{array}{ll}
           \sum_{a\in o^{-1}(c_0)}\omega(a)^{-s}f_{X_U}^tf_{X_a},  &  \text{if }r=v_0;\\
           f_{X_U}^t\Big(f_x\caF(s)+\sum_{y\in X_{c_0}\setminus \{x\}}\caWr(x,y)^{-s}f_y\Big), & \text{if }r=e;
        \end{array}
        \right.
    \end{equation*}
    and
    \begin{equation*}\label{eq:kappa}
        \kappa_{\pi(r)}(u)=\left\{
         \begin{array}{cl}
           \mathbbm{1}_{\{c_0\}}(u)-1,  &  \text{if }r=v_0;\\
            \mathbbm{1}_{A}(u)-1, &  \text{if }r=e.
        \end{array}
        \right.
    \end{equation*}
\end{thm}

\begin{proof}
  One proceeds analogously as in the proof of Theorem~\ref{thm:meroWLIT}. Let $s\in \C$ such that $\sum_{n=0}^\infty \caF(s+1)^n$ converges. 
   By Proposition~\ref{prop:zetageo} and Remark~\ref{rem:Fs}, we deduce what follows:
   \begin{equation*}
   \begin{split}
        \zeta_{G,G_v,G_t}(s) & =\mathbbm{1}_{\{c_0\}}(u)+\sum_{a\in o^{-1}(c_0)}\omega(a)^{-s-1}f_{X_a}\Bigg(\sum_{n=0}^\infty\caF(s+1)^n\Bigg)f_{X_U}^t;\\
        \zeta_{G,G_e,G_t}(s) & =\eta_{a}(u)+f_x\Bigg(\sum_{n=1}^\infty\caF(s+1)^n\Bigg)f_{X_U}^t+\\
        & \quad +\sum_{y\in X_{c_0}\setminus\{x\}}\caWr(x,y)^{-s-1}f_y\Bigg(\sum_{n=0}^\infty\caF(s+1)^n\Bigg)f_{X_U}^t=\\
=\mathbbm{1}_{A}(u)&+\Big(f_x\caF(s+1)+\hspace{-0.4cm}\sum_{y\in X_{c_0}\setminus\{x\}}\hspace{-0.3cm}\caWr(x,y)^{-s-1}f_y\Big)\Bigg(\sum_{n=0}^\infty\caF(s+1)^n\Bigg)f_{X_U}^t.
\end{split}
    \end{equation*}
    Since $\sum_{n=0}^\infty \caF(s+1)^n=(I-\caF(s+1))^{-1}$, Fact~\ref{fact:mdl} yields the claim.
\end{proof}

\section{The reciprocal of $\caZ_{\Gamma,u\to u}(s)$}\label{s:split}
In view of Section~\ref{sus:eul-1}, we present some formulae involving the reciprocal of the function $\caZ_{\Gamma,u\to u}(s)$, for $u\in \Gamma$, introduced in Definition~\ref{defn:Zu}. Recall that this function is a generalisation of $\zeta_{G,G_{t}, G_{t}}(s)$, where $(G,T)$ is a weakly locally $\infty$-transitive group action on a locally finite tree and $t\in T$ (cf.~Remark~\ref{rem:zetaZ}).

\begin{defn}\label{defn:Gs}
    Let $\Gamma$ be a non-empty graph with an edge weight $\omega\colon E\Gamma\to \Z_{\geq 1}$. Consider $\caE(s)$, for $s\in \C$, and $(e_u)_{u\in \Gamma}$ as in Definition~\ref{defn:Es}. For all $c\in V\Gamma$ and $a\in E\Gamma$, define
\begin{gather*}\label{eq:Gca}
\caG_{c}(s):=\caE(s)-\caU_{c,c}(s) \quad \text{and}\quad \caG_{a}(s):=\caE(s)-\caU_{a,a}(s),
\end{gather*}
where $\caU_{c,c}(s)=\sum_{a\in o^{-1}(c)}\omega(a)^{-s}e_c^te_a$ and $\caU_{a,a}(s)=(e_a+e_{\bar{a}})^t\cdot(e_a+e_{\bar{a}})\caE(s)$ (cf.~Theorem~\ref{thm:meroWLIT}).
\end{defn}

\begin{notat}
   If necessary, we write $\caG_\bullet^\Gamma(s)$, $\caE^\Gamma(s)$, $\caU_{\bullet,\bullet}^\Gamma(s)$, $e^\Gamma_\bullet$ and $I^\Gamma$ instead of $\caG_\bullet(s)$, $\caE(s)$, $\caU_{\bullet,\bullet}(s)$, $e_\bullet$ and the identity matrix in~$\Mat_{|E\Gamma|}(\C)$, respectively.   
\end{notat}

\begin{lem}\label{lem:ZWLITinv}
  Let $(\Gamma, \omega)$ satisfy Setting~\hyperref[settG]{\emph{[$\Gamma$]}}, and denote by $I$ the identity matrix on $\Mat_{|E\Gamma|}(\C)$.
  Then, for all $c\in V\Gamma$ and $s\in \C$ such that $I-\caG_c(s)$ is invertible, we have
    \begin{equation*}\label{eq:ZrecV}
\caZ_{\Gamma, c\to c}(s)^{-1}=1-\sum_{a\in o^{-1}(c)}\omega(a)^{-s}e_a(I-\caG_c(s))^{-1}e_c^t.
    \end{equation*}
    Moreover, for all $a\in E\Gamma$ and $s\in \C$ such that $I-\caG_a(s)$ is invertible, we have
    \begin{equation*}
        \caZ_{\Gamma,a\to a}(s)^{-1}=1-(e_a+e_{\bar{a}})\caE(s)(I-\caG_a(s))^{-1}(e_a+e_{\bar{a}})^t.
    \end{equation*}
\end{lem}

\begin{proof}
    By Theorem~\ref{thm:meroWLIT}, we deduce that
    \begin{equation*}
    \begin{array}{l}
        \displaystyle{\caZ_{\Gamma, c\to c}(s)^{-1}=\frac{\det(I-\caG_c(s)-\caU_{c,c}(s))}{\det(I-\caG_{c}(s))};}\\
        \displaystyle{\caZ_{\Gamma,a\to a}(s)^{-1}=\frac{\det(I-\caG_a(s)-\caU_{a,a}(s))}{\det(I-\caG_a(s))}}.
    \end{array}
    \end{equation*}
    The statements now follow from Fact~\ref{fact:mdl}.
\end{proof}

\begin{prop}\label{prop:Vsplit}
    Let $(\Gamma,\omega)$ satisfy Setting~\hyperref[settG]{\emph{[$\Gamma$]}}. Consider two subgraphs $\Gamma_1$ and $\Gamma_2$ of $\Gamma$ such that $\Gamma=\Gamma_1\cup \Gamma_2$ and $\Gamma_1\cap \Gamma_2=\{c\}$, for some $c\in V\Gamma$. Then, 
    \begin{equation*}
     \caZ_{\Gamma, c\to c}(s)^{-1}=\caZ_{\Gamma_1, c\to c}(s)^{-1}+\caZ_{\Gamma_2, c\to c}(s)^{-1}-1.
    \end{equation*}
\end{prop}

\begin{proof}
Let $s\in \C$ such that $I^\Gamma-\caG_c^\Gamma(s)$ is invertible, and set $\Gamma_3=\Gamma_1\cap\Gamma_2$. 
By Definition~\ref{defn:Gs}, for all $a,b\in E\Gamma$ we have
  \begin{equation}\label{eq:Gc1}
  \begin{split}
     (I^\Gamma-\caG_c^\Gamma(s))(a,b) & =e_a^\Gamma\cdot \big(I^\Gamma-\caG_c^\Gamma(s)\big)\cdot(e_b^\Gamma)^t\\
     &=(I^\Gamma-\caE^\Gamma(s))(a,b)+\sum_{a'\in o^{-1}(c)}\omega(a')^{-s}(e_{a}^\Gamma (e_c^\Gamma)^t)\cdot (e_{a'}^\Gamma (e_{b}^\Gamma)^t)\\
       & = (I^\Gamma-\caE^\Gamma(s))(a,b)+\mathbbm{1}_{t^{-1}(c)}(a)\mathbbm{1}_{o^{-1}(c)}(b)\omega(b)^{-s}.
  \end{split}
  \end{equation}
Similarly, for $1\leq i\leq 3$ and for all $a,b\in E\Gamma_i$ we have
\begin{equation}\label{eq:GcGi}
   (I^{\Gamma_i}-\caG_c^{\Gamma_i}(s))(a,b)=(I^{\Gamma_i}-\caE^{\Gamma_i}(s))(a,b)+\mathbbm{1}_{t^{-1}(c)\cap E\Gamma_i}(a)\mathbbm{1}_{o^{-1}(c)\cap E\Gamma_i}(b)\omega(b)^{-s}.
\end{equation}
Combining~\eqref{eq:Gc1} and~\eqref{eq:GcGi}, for every $1\leq i\leq 3$ we deduce that
\begin{equation}\label{eq:Gc2}
    (I^\Gamma-\caG_c^\Gamma(s))(a,b)=(I^{\Gamma_i}-\caG_c^{\Gamma_i})(a,b),\quad\forall\,a,b\in E\Gamma_i
\end{equation}
and Lemma~\ref{lem:ZWLITinv} implies that
\begin{equation}\label{eq:ZrecGammai}
    \caZ_{\Gamma_i,c\to c}(s)^{-1}=1-\sum_{a\in o^{-1}(c)\cap E\Gamma_i,\atop b\in t^{-1}(c)\cap E\Gamma_i}\omega(a)^{-s}(I^{\Gamma}-\caG_c^{\Gamma}(s))^{-1}(a,b).
\end{equation}
Moreover, for all $a,b\in E\Gamma$ with $t(a)=c=o(b)$, by~\eqref{eq:Gc1} we have $(I^\Gamma-\caG_c^\Gamma(s))(a,b)=\mathbbm{1}_{\{a\}}(b)-(\omega(b)-\mathbbm{1}_{\{\bar{a}\}}(b))^{-s}+\omega(b)^{-s}$. Hence,
\begin{equation}\label{eq:Gc2}
    (I^\Gamma-\caG_c^\Gamma(s))(a,b)= 0,\quad\forall\,b\in o^{-1}(c)\setminus\{a,\bar{a}\}.
\end{equation}
We claim that 
\begin{equation*}
    (I^\Gamma-\caG_c^\Gamma(s))(a,b)=0,\quad\forall\,(a,b)\in (E\Gamma_1\times E\Gamma_2)\cup (E\Gamma_2\times E\Gamma_1).
\end{equation*}
Indeed, recall that $E\Gamma_1\cap E\Gamma_2=\emptyset$ and $V\Gamma_1\cap V\Gamma_2=\{c\}$. Hence, for all $a\in E\Gamma_1$ and $b\in E\Gamma_2$, we have $b\not\in\{a,\bar{a}\}$ and either $t(a)\neq o(b)$ or $t(a)=c=o(b)$. Now~\eqref{eq:Gc1} and~\eqref{eq:Gc2} apply. A similar argument holds for all $a\in E\Gamma_2$ and $b\in E\Gamma_1$.

Therefore, once fixed a total order $\leq$ on $E\Gamma$ so that $a<b$ for all $a\in E\Gamma_1$ and $b\in E\Gamma_2$, we have the following decomposition in diagonal blocks:
\begin{gather}\label{eq:block1}
    I^\Gamma-\caG_c^\Gamma(s)=\begin{bmatrix}
        I^{\Gamma_1}-\caG_c^{\Gamma_1}(s) & 0\\
        0 & I^{\Gamma_2}-\caG_c^{\Gamma_2}(s)
    \end{bmatrix}.
\end{gather}
Since
\begin{equation*}\label{eq:o-1c}
    o^{-1}(c)=(o^{-1}(c)\cap E\Gamma_1)\sqcup (o^{-1}(c)\cap E\Gamma_2),
\end{equation*}
by Lemma\til\ref{lem:ZWLITinv}, \eqref{eq:block1} and then ~\eqref{eq:ZrecGammai}, we conclude that
\begin{gather*}\label{eq:presplit}
\begin{array}{ccl}
   \caZ_{\Gamma,c\to c}(s)^{-1} & = & \displaystyle{1-\sum_{a,b\in o^{-1}(c)}\omega(a)^{-s}(I^\Gamma-\caG_c^\Gamma(s))^{-1}(a,\bar{b})}\\
      & = & \displaystyle{1-\sum_{a,b\in o^{-1}(c)\cap E\Gamma_1}\omega(a)^{-s}(I^{\Gamma_1}-\caG_c^{\Gamma_1}(s))^{-1}(a,\bar{b})+}\\
      &  & \displaystyle{-\sum_{a,b\in o^{-1}(c)\cap E\Gamma_2}\omega(a)^{-s}(I^{\Gamma_2}-\caG_c^{\Gamma_2}(s))^{-1}(a,\bar{b})}\\
      & = & \caZ_{\Gamma_1,c\to c}(s)^{-1}+\caZ_{\Gamma_2,c\to c}(s)^{-1}-1.\qedhere
\end{array}
\end{gather*}
\end{proof}
\begin{cor}\label{cor:Vsplit}
    Let $(\Gamma,\omega)$ satisfy Setting~\hyperref[settG]{\emph{[$\Gamma$]}}. Assume that there are subgraphs $\Lambda_1$ and $\Lambda_2$ of $\Gamma$ such that $\Gamma=\Lambda_1\cup\Lambda_2$ and $\Lambda_1\cap\Lambda_2=\{c\}$, for some vertex $c\in V\Gamma$. Then, for all subgraphs $\Gamma_1$ and $\Gamma_2$ of $\Gamma$ satisfying $\Gamma_i\supseteq \Lambda_i$ for every $i\in\{1,2\}$, we have
    \begin{equation*}
     \caZ_{\Gamma, c\to c}(s)^{-1}=\caZ_{\Gamma_1, c\to c}(s)^{-1}+\caZ_{\Gamma_2, c\to c}(s)^{-1}-\caZ_{\Gamma_1\cap\Gamma_2, c\to c}(s)^{-1}.   
    \end{equation*}
\end{cor}

\begin{proof}
   Let $\Gamma_1$ and $\Gamma_2$ be as in the statement, and set $\Gamma_3:=\Gamma_1\cap\Gamma_2$. Note that $\Lambda_1\cap \Gamma_3=\Lambda_1\cap \Gamma_2$ and $\Lambda_2\cap \Gamma_3=\Lambda_2\cap \Gamma_1$.
    Therefore,
    \begin{gather}\label{eq:grdecV}
        \begin{array}{lcl}
           \Gamma=\Lambda_1\cup \Lambda_2 & \text{and} & \Lambda_1\cap \Lambda_2=\{c\};\\
           \Gamma_1=\Lambda_1\cup (\Lambda_2\cap \Gamma_1) & \text{and} & \Lambda_1\cap (\Lambda_2\cap\Gamma_1)=\{c\};\\
           \Gamma_2=(\Lambda_1\cap \Gamma_2)\cup \Lambda_2 & \text{and} & (\Lambda_1\cap\Gamma_2)\cap \Lambda_2=\{c\};\\
           \Gamma_3=(\Lambda_1\cap \Gamma_2)\cup (\Lambda_2\cap\Gamma_1) & \text{and} & (\Lambda_1\cap\Gamma_2)\cap (\Lambda_2\cap \Gamma_1)=\{c\}.
        \end{array}
    \end{gather}
    Applying Proposition~\ref{prop:Vsplit} to each decomposition in~\eqref{eq:grdecV} yields the claim.
\end{proof}

The following lemma shows that the hypotheses of Proposition~\ref{prop:Vsplit} are satisfied for every $c\in V\Gamma$ if, for instance, $\Gamma$ is a connected graph without $n$-cycles for every $n\geq 2$.
\begin{lem}\label{lem:GSplit}
    Let $\Gamma$ be a connected graph without $n$-cycles, for every $n\geq 2$. For every $c\in V\Gamma$, there are connected subgraphs $\Gamma_1$ and $\Gamma_2$ of $\Gamma$ such that $\Gamma=\Gamma_1\cup\Gamma_2$ and $\Gamma_1\cap\Gamma_2=\{c\}$. If in particular $|o^{-1}(c)|\geq 2$, one may take $\Gamma_1$ and $\Gamma_2$ to be proper subgraphs of $\Gamma$.
\end{lem}
\begin{proof}
Let $\{\Xi_i\}_{i\in I}$ be the collection of all the connected components of the graph $\Gamma\setminus (o^{-1}(c)\cup \overline{o^{-1}(c)}\cup\{c\})$. Recall that $\Gamma\setminus (o^{-1}(c)\cup \overline{o^{-1}(c)}\cup\{c\})=\bigsqcup_{i\in I}\Xi_i$.
For every $i\in I$, there is exactly one edge $a_i\in o^{-1}(c)$ such that $t(a_i)\in \Xi_i$. In fact, assume that there are $a,b\in o^{-1}(c)$ with $a\neq b$ such that $x=t(a),y=t(b)\in V\Xi_i$. Then the reduced path $[x,y]$ as in Remark~\ref{rem:Gnocyc} is contained in~$\Xi_i$. Hence $a\cdot [x,y]\cdot \bar{b}$ is a cycle of length $\geq 2$ in~$\Gamma$, impossible.

Consider subsets $\caE_1,\caE_2$ of $o^{-1}(c)$ with the following properties: $\caE_1\cap\caE_2=\emptyset$, $o^{-1}(c)=\caE_1\cup \caE_2$ and, for every $k\in\{1,2\}$, every $1$-loop $a$ in $\caE_k$ satisfies $\bar{a}\in\caE_k$. 
Provided $\overline{\caE}_k=\{\bar{a}\mid a\in\caE_k\}$ for every $k\in\{1,2\}$, note that $(\caE_1\cup\overline{\caE}_1)\cap (\caE_2\cup\overline{\caE}_2)=\emptyset$.
Moreover, if $|o^{-1}(c)|\geq 2$, take $\caE_1$ and $\caE_2$ so that $\caE_1\neq\emptyset$ and $\caE_2\neq\emptyset$. 
For $k\in\{1,2\}$, set also $I_k:=\{i\in I\mid a_i\in \caE_k\}$. Note that $I_1\cap I_2=\emptyset$ and $I=I_1\cup I_2$. For $k\in\{1,2\}$, define the following subgraph of $\Gamma$:
\begin{equation*}
   \Gamma_k:=\{c\}\cup \caE_k\cup \overline{\caE}_k\cup\bigsqcup_{i\in I_k}\Xi_i.
\end{equation*}

One checks that $\Gamma=\Gamma_1\cup\Gamma_2$ and $\Gamma_1\cap\Gamma_2=\{c\}$. If in particular $|o^{-1}(c)|\geq 2$, then $\Gamma_1\setminus\Gamma_2\supseteq \caE_1\neq\emptyset$ and $\Gamma_2\setminus\Gamma_1\supseteq \caE_2\neq\emptyset$. Therefore, both $\Gamma_1$ and $\Gamma_2$ are proper subgraphs of $\Gamma$.
\end{proof}

\begin{prop}\label{prop:Esplit}
  Let $(\Gamma,\omega)$ satisfy Setting~\hyperref[settG]{\emph{[$\Gamma$]}}. Consider subgraphs $\Gamma_1$ and $\Gamma_2$ of $\Gamma$ satisfying $\Gamma=\Gamma_1\cup \Gamma_2$ and such that $\Gamma_3:=\Gamma_1\cap \Gamma_2$ is a $1$-segment with edge set $\{a,\bar{a}\}$. Then, 
    \begin{gather*}\label{eq:Esplit}
    \caZ_{\Gamma,a\to a}(s)^{-1}=\caZ_{\Gamma_1,a\to a}(s)^{-1}+\caZ_{\Gamma_2,a\to a}(s)^{-1}-\caZ_{\Gamma_3, a\to a}(s)^{-1}.
    \end{gather*}
\end{prop}

\begin{proof} 
Fix $s\in \C$ such that $I^\Gamma-\caG_a^\Gamma(s)$ is invertible, and set $c=o(a)$, $d=t(a)$. By Lemma~\ref{lem:ZWLITinv}, we have
\begin{equation}\label{eq:preEsplit1}
\begin{array}{c}
    \caZ_{\Gamma,a\to a}(s)^{-1}=\\
    \displaystyle{=1-\Bigg(\sum_{b\in E\Gamma, \atop o(b)=d}\caE^\Gamma(s)(a,b)e_b^\Gamma+\sum_{b\in E\Gamma, \atop o(b)=c}\caE^\Gamma(s)(\bar{a},b)e_b^\Gamma\Bigg)(I^\Gamma-\caG_a^\Gamma(s))^{-1}(e_a^\Gamma+e_{\bar{a}}^\Gamma)^t.}
    \end{array}
\end{equation}
By Definition~\ref{defn:Gs}, for all $b_1,b_2\in E\Gamma$ we observe that
\begin{gather}\label{eq:preEsplit3}
\begin{array}{c}
(I^\Gamma-\caG_a^\Gamma(s))(b_1,b_2)=\\
\begin{array}{rl}
    = &\mathbbm{1}_{\{b_1\}}(b_2)-\caE^\Gamma(s)(b_1, b_2)+e_{b_1}^\Gamma(e_a^\Gamma+e_{\bar{a}}^\Gamma)^t(e_a^\Gamma+e_{\bar{a}}^\Gamma)\caE^\Gamma(s)(e_{b_2}^\Gamma)^t\\
    = & \mathbbm{1}_{\{b_1\}}(b_2)-\caE^\Gamma(s)(b_1, b_2)+\mathbbm{1}_{\{a,\bar{a}\}}(b_1)\cdot\Big (\caE^\Gamma(s)(a,b_2)+\caE^\Gamma(s)(\bar{a},b_2)\Big). 
\end{array}
\end{array}
\end{gather}
Hence, for every $i\in\{1,2\}$,
\begin{gather}\label{eq:preEsplit4}
   \begin{split}
     (I^{\Gamma_i}-\caG_a^{\Gamma_i}(s))(b_1,b_2) & =(I^\Gamma-\caG_a^\Gamma(s))(b_1,b_2),\quad\forall\,b_1,b_2\in E\Gamma_i;\\
    (I^{\Gamma_i}-\caG_a^{\Gamma_i}(s))(b_1,b_2) & =(I^{\Gamma_3}-\caG_a^{\Gamma_3}(s))(b_1,b_2),\quad \forall\,b_1,b_2\in\{a,\bar{a}\}.     
   \end{split}
\end{gather}
Let $\Lambda_1$ (resp.~$\Lambda_2$) be the graph obtained from $\Gamma_1$ (resp.~$\Gamma_2$) by removing $a,\bar{a}$ and $d$ (resp.~$a,\bar{a}$ and $c$). The following picture sketches an example of $\Lambda_1$ (with wavy edges) and $\Lambda_2$ (with dashed edges).
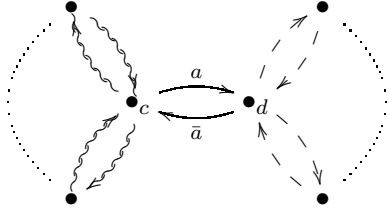
\begin{equation*}
    {\xymatrix@C=1em{\bullet\ar@{~>}@/^/[rd]\ar@/_2.0pc/@{..}[dd]& &&&\bullet\ar@{-->}@/^/[ld]\ar@/^2.0pc/@{..}[dd]\\
 &\bullet_c\ar@{~>}@/^/[lu]\ar@{~>}@/^/[ld]\ar@{->}@/^/[rr]^{a}&&\bullet_{d}\ar@{->}@/^/[ll]^{\bar{a}}\ar@{-->}@/^/[ru]\ar@{-->}@/^/[rd]&\\
 \bullet\ar@{~>}@/^/[ru]&&&&\bullet\ar@{-->}@/^/[lu]}}
\end{equation*}
Note that $c\in V\Lambda_1$, $d\in V\Lambda_2$ and $V\Lambda_1\cap V\Lambda_2=\emptyset$. 
In particular, no edges of~$\Lambda_1$ (resp.~$\Lambda_2$) end in a vertex of~$\Lambda_2$ (resp.~$\Lambda_1$).
We claim that
\begin{equation}\label{eq:outdiag1}
\begin{array}{c}
     (I^\Gamma-\caG_a^\Gamma(s))(b_1,b_2)=0,\quad\forall\,(b_1,b_2)\in (E\Lambda_1\sqcup \{a\})\times(E\Lambda_2\sqcup\{\bar{a}\});\\
     (I^\Gamma-\caG_a^\Gamma(s))(b_1,b_2)=0,\quad\forall\,(b_1,b_2)\in (E\Lambda_2\sqcup\{\bar{a}\})\times (E\Lambda_1\sqcup \{a\}).
\end{array}
\end{equation}
Indeed, let $b_1\in E\Lambda_1\sqcup \{a\}$ and $b_2\in E\Lambda_2\sqcup\{\bar{a}\}$. If $b_1\in E\Lambda_1$, then $V\Lambda_1\ni t(b_1)\neq o(b_2)\in V\Lambda_2$ and~\eqref{eq:preEsplit3} implies that $I^\Gamma-\caG^\Gamma_a(s)=-\caE^\Gamma(b_1,b_2)=0$. If $b_1=a$, then $c=t(\bar{a})\neq o(b_2)\in E\Lambda_2$ and~\eqref{eq:preEsplit3} implies that $I^\Gamma-\caG_a^\Gamma(s)=\caE^\Gamma(s)(\bar{a},b_2)=0$. The second line of~\eqref{eq:outdiag1} can be proved analogously.

Fix a total order $\leq$ on $E\Gamma=E\Lambda_1\sqcup \{a,\bar{a}\}\sqcup E\Lambda_2$ so that $b_1<a<\bar{a}<b_2$ for all $b_1\in E\Lambda_1$ and $b_2\in E\Lambda_2$. Set also
\begin{equation*}
\begin{split}
    A & :=[(I^{\Gamma_1}-\caG_a^{\Gamma_1}(s))(b_1,b_2)]_{b_1,b_2\in E\Lambda_1\sqcup\{a\}}; \\
    B & :=[(I^{\Gamma_2}-\caG_a^{\Gamma_2}(s))(b_1,b_2)]_{b_1,b_2\in \{\bar{a}\}\sqcup E\Lambda_2};\\
    \alpha & :=\omega(a)-1 \quad\text{and}\quad \beta:=\omega(\bar{a})-1.
    \end{split}
\end{equation*}
From~\eqref{eq:preEsplit4} (and~\eqref{eq:preEsplit3}) we observe that 
$A(a,a)=1+\alpha^{-s}$ and $B(\bar{a},\bar{a})=1+\beta^{-s}$.
Moreover, by~\eqref{eq:preEsplit4} and~\eqref{eq:outdiag1}, we have the following decompositions in diagonal blocks:
\begin{equation}\label{eq:preEsplit5}
\begin{split}
    I^\Gamma-\caG_a^\Gamma(s)=\begin{bmatrix}
        A & 0\\
        0 & B
    \end{bmatrix} & \quad\text{and}\quad    
    (I^\Gamma-\caG_a^\Gamma(s))^{-1}=\begin{bmatrix}
        A^{-1} & 0\\
        0 & B^{-1}
    \end{bmatrix}; \\
    I^{\Gamma_1}-\caG_a^{\Gamma_1}(s)=\begin{bmatrix}
        A & 0\\
        0 & B(\bar{a},\bar{a})
    \end{bmatrix} & \quad\text{and}\quad
     (I^{\Gamma_1}-\caG_a^{\Gamma_1}(s))^{-1}=\begin{bmatrix}
        A^{-1} & 0\\
        0 & B(\bar{a},\bar{a})^{-1}
    \end{bmatrix};
    \\
    I^{\Gamma_2}-\caG_a^{\Gamma_2}(s)=\begin{bmatrix}
        A(a,a) & 0\\
        0 & B
    \end{bmatrix} & \quad\text{and}\quad
    (I^{\Gamma_2}-\caG_a^{\Gamma_2}(s))^{-1}=\begin{bmatrix}
        A(a,a)^{-1} & 0\\
        0 & B^{-1}
    \end{bmatrix}
    .
\end{split}
\end{equation}
Note that
\begin{gather}\label{eq:ocd}
       o^{-1}(c)=\{a\}\sqcup(o^{-1}(c)\cap E\Lambda_1)\,\text{ and }\,
    o^{-1}(d)=\{\bar{a}\}\sqcup (o^{-1}(d)\cap E\Lambda_2).
\end{gather}
Therefore, by~\eqref{eq:preEsplit5} and~\eqref{eq:ocd}, we rewrite~\eqref{eq:preEsplit1} as follows:
\begin{equation}\label{eq:preEsplit6}
\begin{split}
     \caZ_{\Gamma, a\to a}(s)^{-1} & = 1-\beta^{-s}B^{-1}(\bar{a},\bar{a})-\sum_{b\in o^{-1}(d)\cap E\Lambda_2}\caE^{\Gamma_2}(s)(a,b)\cdot B^{-1}(b,\bar{a})+\\
       &  \quad -a^{-s}A^{-1}(a,a)-\sum_{b\in o^{-1}(c)\cap E\Lambda_1}\caE^{\Gamma_1}(s)(\bar{a},b)\cdot A^{-1}(b,a).
\end{split}
\end{equation}
A formula analogous to~\eqref{eq:preEsplit1} holds for $\Gamma_1$, namely
\begin{gather*}
    \caZ_{\Gamma_1,a\to a}(s)^{-1}=\\
    1-\Bigg(\caE^{\Gamma_1}(s)(a,\bar{a})e_{\bar{a}}^{\Gamma_1}+\caE^{\Gamma_1}(s)(\bar{a},a)e_a^{\Gamma_1}+\sum_{b\in o^{-1}(c)\cap E\Lambda_1}\caE^{\Gamma_1}(s)(\bar{a},b)e_b^{\Gamma_1}\Bigg)\cdot\nonumber\\
    \cdot\Big(I^{\Gamma_1}-\caG_a^{\Gamma_1}(s)\Big)^{-1}\cdot (e_a^{\Gamma_1}+e_{\bar{a}}^{\Gamma_1})^t\nonumber
\end{gather*}
which, by~\eqref{eq:preEsplit5}, yields
\begin{equation}\label{eq:ZaG1}
    \caZ_{\Gamma_1,a\to a}(s)^{-1}\hspace{-0.1cm}=1-\beta^{-s}B(\bar{a},\bar{a})^{-1}-\alpha^{-s}A^{-1}(a,a)-\hspace{-0.7cm}\sum_{b\in o^{-1}(c)\cap E\Lambda_1}\hspace{-0.5cm}\caE^{\Gamma_1}(s)(\bar{a},b)\cdot A^{-1}(b,a).
\end{equation}
In a similar manner, we deduce that
\begin{equation}\label{eq:ZaG2}
   \caZ_{\Gamma_2,a\to a}(s)^{-1}
   \hspace{-0.1cm}=1-\alpha^{-s}A(a,a)^{-1}-\beta^{-s}B^{-1}(\bar{a},\bar{a})-\hspace{-0.7cm}\sum_{b\in o^{-1}(d)\cap E\Lambda_2}\hspace{-0.5cm}\caE^{\Gamma_2}(s)(a,b)\cdot B^{-1}(b,\bar{a}).  
\end{equation}
By Example~\ref{ex:1segmWLIT}, one also checks that
\begin{equation}\label{eq:ZaGa}
    \caZ_{\Gamma_3, a\to a}(s)^{-1}=\frac{1-\alpha^{-s}\beta^{-s}}{(1+\alpha^{-s})(1+\beta^{-s})}=1-\alpha^{-s}A(a,a)^{-1}-\beta^{-s}B(\bar{a},\bar{a})^{-1}.
\end{equation}
Combining~\eqref{eq:preEsplit6}, \eqref{eq:ZaG1}, \eqref{eq:ZaG2} and~\eqref{eq:ZaGa}, we conclude the claim.
\end{proof}
Note that the strategy to prove Proposition~\ref{prop:Esplit} strictly depends on the fact that $a$ has distinct endpoints (cf.~\eqref{eq:outdiag1}).

\begin{cor}\label{cor:Esplit}
   Let $(\Gamma,\omega)$ satisfy Setting~\hyperref[settG]{\emph{[$\Gamma$]}}. Assume that there are subgraphs~$\Lambda_1$ and~$\Lambda_2$ of~$\Gamma$ satisfying $\Gamma=\Lambda_1\cup\Lambda_2$ and such that $\Lambda_1\cap\Lambda_2$ is a $1$-segment graph with edge set $\{a,\bar{a}\}$. Then, for all subgraphs $\Gamma_1$ and~$\Gamma_2$ of~$\Gamma$ such that $\Gamma_i\supseteq\Lambda_i$ for every $i\in\{1,2\}$, we have
   \begin{equation*}
       \caZ_{\Gamma,a\to a}(s)^{-1}=\caZ_{\Gamma_1,a\to a}(s)^{-1}+\caZ_{\Gamma_2,a\to a}(s)^{-1}-\caZ_{\Gamma_1\cap\Gamma_2,a\to a}(s)^{-1}.
   \end{equation*}
\end{cor}

\begin{proof}
    Let $\Gamma_1$ and $\Gamma_2$ be as in the statement. For simplicity, set $\Gamma_3=\Gamma_1\cap\Gamma_2$ and $\Gamma_a=\Lambda_1\cap\Lambda_2$. Since $\Lambda_1\subseteq \Gamma_1$ and $\Lambda_2\subseteq \Gamma_2$, we have $\Lambda_1\cap \Gamma_2=\Lambda_1\cap \Gamma_3$ and $\Lambda_2\cap \Gamma_1=\Lambda_2\cap \Gamma_3$.
    Therefore, the following decompositions hold:
    \begin{gather}\label{eq:grdecE}
        \begin{array}{lcl}
           \Gamma=\Lambda_1\cup \Lambda_2 & \text{and} & \Lambda_1\cap \Lambda_2=\Gamma_a;\\
           \Gamma_1=\Lambda_1\cup (\Lambda_2\cap \Gamma_1) & \text{and} & \Lambda_1\cap (\Lambda_2\cap\Gamma_1)=\Gamma_a;\\
           \Gamma_2=(\Lambda_1\cap \Gamma_2)\cup\Lambda_2 & \text{and} &  (\Lambda_1\cap\Gamma_2)\cap \Lambda_2=\Gamma_a;\\
           \Gamma_3=(\Lambda_1\cap \Gamma_2)\cup (\Lambda_2\cap\Gamma_1) & \text{and} & (\Lambda_1\cap\Gamma_2)\cap (\Lambda_2\cap \Gamma_1)=\Gamma_a.
        \end{array}
    \end{gather}
    Applying Proposition~\ref{prop:Esplit} to each decomposition in \eqref{eq:grdecE} yields the claim.
\end{proof}
For completeness, in analogy to Lemma~\ref{lem:GSplit} we observe the following.
\begin{rem}\label{rem:GEsplit}
    Let $\Gamma$ be a connected graph without $n$-cycles for every $n\geq 2$, and let $a\in E\Gamma$. By~Lemma~\ref{lem:GSplit}, there are connected subgraphs $\Lambda_1$ and $\Lambda_2$ of $\Gamma$ such that $\Gamma=\Lambda_1\cup\Lambda_2$ and $\Lambda_1\cap\Lambda_2=\{o(a)\}$. Denote by~$\Gamma_1$ and~$\Gamma_2$ the smallest subgraphs of $\Gamma$ containing $\Lambda_1\cup\{a,\bar{a}\}$ and $\Lambda_2\cup\{a,\bar{a}\}$, respectively. Then $\Gamma=\Gamma_1\cup\Gamma_2$ and $\Gamma_1\cap\Gamma_2$ is the subgraph of $\Gamma$ with edge set $\{a,\bar{a}\}$.
\end{rem}

\smallskip

In view of the next proofs, it might be useful to recall the following well-known fact~\cite{tab:mat}. Given a $2\times 2$ block-matrix $M=[M_{ij}]_{1\leq i,j\leq 2}\in \Mat_n(\C)$ with $M_{22}$ invertible, one has
  \begin{equation}\label{eq:blockdet}
     \det(M)=\det(M_{22})\cdot \det(M_{11}-M_{12}M_{22}^{-1}M_{21}).
  \end{equation}
  
\begin{prop}\label{prop:Zasegm}
     Let $(\Gamma,\omega)$ satisfy Setting~\hyperref[settG]{\emph{[$\Gamma$]}}. Consider $a\in E\Gamma$ with $o(a)=:c\neq d:=t(a)$, and assume that $c$ is a terminal vertex in~$\Gamma$. Put $\omega(a)=\alpha+1$, $\omega(\bar{a})=\beta+1$, and denote by $\Lambda$ the graph obtained from $\Gamma$ by removing $a$, $\bar{a}$ and $c$. Then, 
     \begin{equation*}
     \begin{split}
         \caZ_{\Gamma,c\to c}(s)^{-1} &=\frac{(1+\alpha^{-s}\xi(\beta,s))\cdot\caZ_{\Lambda,d\to d}(s)^{-1}-\alpha^{-s}(\beta+1)^{-s}}{(1-\xi(\alpha,s)\xi(\beta,s))\cdot\caZ_{\Lambda,d\to d}(s)^{-1}+\xi(\alpha,s)(\beta+1)^{-s}};\\
        \caZ_{\Gamma,a\to a}(s)^{-1} &=\frac{(1+\alpha^{-s}\xi(\beta,s))\cdot\caZ_{\Lambda,d\to d}(s)^{-1}-\alpha^{-s}(\beta+1)^{-s}}{(1+\alpha^{-s}) \Big((1-\xi(\beta,s))\cdot \caZ_{\Lambda,d\to d}(s)^{-1}+(\beta+1)^{-s}\Big)},
        \end{split}
    \end{equation*}
    where $\xi(\alpha,s)=(\alpha+1)^{-s}-\alpha^{-s}$ and $\xi(\beta,s)=(\beta+1)^{-s}-\beta^{-s}$.
\end{prop}
The picture below sketches a possible setting for Proposition~\ref{prop:Zasegm}. The edges of $\Lambda$ are dashed.
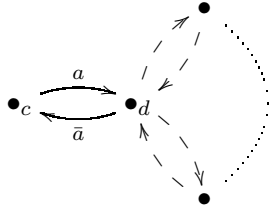
\begin{equation*}
{\xymatrix@C=1em{&&&\bullet\ar@{-->}@/^/[ld]\ar@/^2.0pc/@{..}[dd]\\
 \bullet_c\ar@{->}@/^/[rr]^{a}&&\bullet_{d}\ar@{->}@/^/[ll]^{\bar{a}}\ar@{-->}@/^/[ru]\ar@{-->}@/^/[rd]&\\
 &&&\bullet\ar@{-->}@/^/[lu]}}
\end{equation*}

\begin{proof}
    Let $s\in \C$ be such that $\caZ_{\Gamma,a\to a}(s)\neq 0$, and consider a total order $\leq$ on $E\Gamma$ such that $a<\bar{a}<b$ for all $b\in E\Lambda$. Then $\caE^\Gamma(s)$ admits the following block decomposition:
    \begin{equation}\label{eq:blockEs}
        \caE^\Gamma(s)=\begin{bmatrix}
        A & B\\
        C & D
        \end{bmatrix},
    \end{equation}
    where $A=[\caE^\Gamma(s)(b_1,b_2)]_{b_1,b_2\in \{a,\bar{a}\}}$, $B=[\caE^\Gamma(s)(b_1,b_2)]_{b_1\in \{a,\bar{a}\},\,b_2\in E\Lambda}$, $C=[\caE^\Gamma(s)(b_1,b_2)]_{b_1\in E\Lambda,\,b_2\in \{a,\bar{a}\}}$ and $D=[\caE^\Gamma(s)(b_1,b_2)]_{b_1,b_2\in E\Lambda}=\caE^\Lambda(s)$.
    For $i\in\{1,2\}$, let $B_i$ and $C^i$ denote the $i$-th row of $B$ and the $i$-th column of $C$, respectively. One checks that
   \begin{equation}\label{eq:BC12}
       B_1=\sum_{b\in o^{-1}(d)\cap E\Lambda}\omega(b)^{-s}e_b^\Lambda;\quad B_2=\underline{0};\quad C^1=\underline{0}^t\quad\text{and}\quad C^2=(\beta+1)^{-s}(e_d^\Lambda)^t,
   \end{equation}
    where $\underline{0}$ denotes the row zero vector in~$\C^{|E\Lambda|}$. Moreover, denote by~$E_{21}$ and~$\underline{1}$ the elementary matrix associated to $(2,1)$  in~$\Mat_2(\C)$ and the matrix in~$\Mat_2(\C)$ with all entries equal to~$1$, respectively. Hence,
    \begin{equation*}
    \begin{split}
    \caU_{cc}^\Gamma(s) & =(\alpha+1)^{-1}(e_{\bar{a}}^\Gamma)^t(e_a^\Gamma)=\begin{bmatrix}
        (\alpha+1)^{-s}E_{21} & 0_{2\times |E\Lambda|}\\
         0_{|E\Lambda|\times 2} & 0_{|E\Lambda|\times |E\Lambda|}
    \end{bmatrix}
    ;\\
        \caU_{aa}^\Gamma(s) & =(e_a^\Gamma+e_{\bar{a}}^\Gamma)^t(e_a^\Gamma+e_{\bar{a}}^\Gamma)\caE^\Gamma(s)=\begin{bmatrix}
            \underline{1}\cdot A & \underline{1}\cdot B\\
            0_{|E\Lambda|\times 2} & 0_{|E\Lambda|\times |E\Lambda|}
        \end{bmatrix}.
    \end{split}
    \end{equation*}
    Denoting by $I_2$ the identity matrix in $\Mat_2(\C)$, the following holds:
    \begin{equation}\label{eq:Z1segm}
    \begin{split}
        I^\Gamma-\caE^\Gamma(s) & =\begin{bmatrix}
            I_2-A & -B\\
            -C & I_\Lambda-D
        \end{bmatrix};\\
        I^\Gamma-\caE^\Gamma(s)+\caU_{c,c}^\Gamma(s) & =\begin{bmatrix}
            I_2-A+(\alpha+1)^{-s}E_{21} & -B\\
            -C & I_\Lambda-D
        \end{bmatrix};\\
        I^\Gamma-\caE^\Gamma(s)+\caU_{a,a}^\Gamma(s) & =\begin{bmatrix}
            I_2+(\underline{1}-I_2)A & (\underline{1}-I_2)B\\
            -C & I_\Lambda-D
        \end{bmatrix};\\
        I_2-A=\begin{bmatrix}
            1 & -\beta^{-s}\\
            -\alpha^{-s} & 1
        \end{bmatrix}; & \qquad
        I_2+(\underline{1}-I_2)A=\begin{bmatrix}
        1+\alpha^{-s} & 0\\
        0 & 1+\beta^{-s}
        \end{bmatrix}.
        \end{split}
    \end{equation}
   By Theorem~\ref{thm:meroWLIT}, \eqref{eq:blockdet} and~\eqref{eq:Z1segm}, we deduce that 
   \begin{equation}\label{eq:Z1segm.c}
       \begin{split}
           \caZ_{\Gamma,c\to c}(s)^{-1} & = \frac{\det(I^\Gamma-\caE^\Gamma(s))}{\det(I-\caE^\Gamma(s)+\caU_{c,c}^\Gamma(s))}\\
            & \hspace{-0.7cm}=\frac{\det\Big(I_2-A-\overbrace{B(I_\Lambda-D)^{-1}C}^{=:X}\Big)}{\det\Big(I_2-A+(\alpha+1)^{-s}E_{21}-\underbrace{B(I_\Lambda-D)^{-1}C}_{=X}\Big)}\\
            & \hspace{-0.7cm}=\frac{(1-X(a,a))(1-X(\bar{a},\bar{a}))-(X(a,\bar{a})+\beta^{-s})(X(\bar{a},a)+\alpha^{-s})}{(1-X(a,a))(1-X(\bar{a},\bar{a}))+(\xi(\alpha,s)-X(\bar{a},a))(\beta^{-s}+X(a,\bar{a}))}.
       \end{split}
   \end{equation}
   Similarly,
   \begin{equation}\label{eq:Z1segm.2}
       \begin{split}
           \caZ_{\Gamma,a\to a}(s)^{-1} & = \frac{\det(I^\Gamma-D)}{\det(I-D+\caU_{a,a}^\Gamma(s))}\\
            &\hspace{-0.7cm}=\frac{\det\Big(I_2-A-\overbrace{B(I_\Lambda-D)^{-1}C}^{=X}\Big)}{\det\Big(I_2+(\underline{1}-I_2)A+\underbrace{(\underline{1}-I_2)B(I_\Lambda-D)^{-1}C}_{=:Y}\Big)}\\
            &\hspace{-0.7cm}=\frac{(1-X(a,a))(1-X(\bar{a},\bar{a}))-(X(a,\bar{a})+\beta^{-s})(X(\bar{a},a)+\alpha^{-s})}{(1+\alpha^{-s}+Y(a,a))(1+\beta^{-s}+Y(\bar{a},\bar{a}))-Y(a,\bar{a})Y(\bar{a},a)}.
       \end{split}
   \end{equation}
    It remains to study the entries of $X$ and $Y$. Observe that $(\underline{1}-I_2)B$ is the matrix obtained from $B$ by interchanging its two rows. Since $B_2=\underline{0}$ and $C^1=\underline{0}^t$, we deduce the following:
    \begin{equation*}
        \begin{split}
            X(a,a)& =Y(\bar{a},a)=B_1(I_\Lambda-\caE^\Lambda(s))^{-1}C^1=0;\\ 
            X(\bar{a},a)& =Y(a,a)=B_2(I_\Lambda-\caE^\Lambda(s))^{-1}C^1=0;\\
           X(\bar{a},\bar{a}) & =Y(a,\bar{a})=B_2(I_\Lambda-\caE^\Lambda(s))^{-1}C^2=0.
        \end{split}
    \end{equation*}
    Moreover, by Theorem~\ref{thm:meroWLIT}, Fact~\ref{fact:mdl} and since $\caE^\Lambda(s)=D$, 
    \begin{equation*}
    \begin{split}
        X(a,\bar{a})=Y(\bar{a},\bar{a})& =B_1(I_\Lambda-\caE^\Lambda(s))^{-1}C^2\\
         & =(\beta+1)^{-s}\cdot\hspace{-0.5cm}\sum_{b\in o^{-1}(d)\cap E\Lambda}\hspace{-0.5cm}\omega(b)^{-s}e_b^\Lambda(I_\Lambda-\caE^\Lambda(s))^{-1}(e_d^{\Lambda})^t\\
         & =(\beta+1)^{-s}(\caZ_{\Lambda,d\to d}(s)-1).\qedhere
        \end{split}
    \end{equation*}
    The claim now follows by substitution and elementary algebraic manipulations.
\end{proof}

\begin{prop}\label{prop:Zaloop} 
    Let $(\Gamma,\omega)$ satisfy Setting~\hyperref[settG]{\emph{[$\Gamma$]}}, and consider $a\in E\Gamma$ with $o(a)=t(a)=c$. Set $\omega(a)=\alpha+1$, $\omega(\bar{a})=\beta+1$, and consider the subgraph 
 of $\Gamma$ given by $\Lambda:=\Gamma\setminus\{a,\bar{a}\}$. Then, 
    \begin{equation*}
        \caZ_{\Gamma,a\to a}(s)^{-1}=\frac{\xi_1(\alpha,\beta)\cdot \caZ_{\Lambda,c\to c}(s)^{-1}-\eta(\alpha,\beta)}{\xi_2(\alpha,\beta)\cdot \caZ_{\Lambda,c\to c}(s)^{-1}+\eta(\alpha,\beta)},
    \end{equation*}
    where 
    \begin{equation*}
        \begin{split}
            \xi_1(\alpha,\beta) & =1-\Big(\alpha^{-s}-(\alpha+1)^{-s}\Big)\Big(\beta^{-s}-(\beta+1)^{-s}\Big);\\
            \xi_2(\alpha,\beta) & =\Big(1+\alpha^{-s}-(\alpha+1)^{-s}\Big)\Big(1+\beta^{-s}-(\beta+1)^{-s}\Big);\\
            \eta(\alpha,\beta) & =(\alpha^{-s}+1)(\beta+1)^{-s}+(\alpha+1)^{-s}(\beta^{-s}+1)-2(\alpha+1)^{-s}(\beta+1)^{-s}.
        \end{split}
    \end{equation*}
\end{prop}

The picture below sketches a possible setting for Proposition~\ref{prop:Zaloop}. The edges of~$\Lambda$ are dashed.
\begin{equation*}
{\xymatrix@C=1em{ &\bullet\ar@{-->}@/^/[ld]\ar@/^2.0pc/@{..}[dd]\\
 \bullet_c\ar@(l,u)^{a}\ar@(l,d)_{\bar{a}}\ar@{-->}@/^/[ru]\ar@{-->}@/^/[rd]&\\
 &\bullet\ar@{-->}@/^/[lu]}}
\end{equation*}

\begin{rem}\label{rem:Zaloop}
    In Proposition~\ref{prop:Zaloop}, if $\alpha=\beta$, after elementary manipulations the given formula becomes
    \begin{equation*}
        \caZ_{\Gamma,a\to a}(s)^{-1}=\frac{(1-\alpha^{-s}+(\alpha+1)^{-s})\cdot \caZ_{\Lambda,c\to c}(s)^{-1}-2(\alpha+1)^{-s}}{(1+\alpha^{-s}-(\alpha+1)^{-s})\cdot \caZ_{\Lambda,c\to c}(s)^{-1}+2(\alpha+1)^{-s}}.
    \end{equation*}
\end{rem}

\begin{proof}[Proof of Proposition~\ref{prop:Zaloop}]
  The strategy of the proof is analogous to the one of Proposition~\ref{prop:Zasegm}. Thus, we keep the same notation and proof structure, and we only specify what needs to be changed.
  First, instead of~\eqref{eq:BC12}, the rows $B_1$ and $B_2$ of $B$ and the columns $C^1$ and $C^2$ of $C$ are the following:
   \begin{equation}\label{eq:1L1}
    \begin{split}
    B_1 &=B_2=\sum_{b\in o^{-1}(c)\cap E\Lambda}\omega(b)^{-s}e_b^{\Lambda};\\
    C^1 &=\sum_{b\in t^{-1}(c)\cap E\Lambda}(\alpha+1)^{-s}(e_b^{\Lambda})^t=(\alpha+1)^{-s}(e_c^\Lambda)^t;\\
    C^2&=\sum_{b\in t^{-1}(c)\cap E\Lambda}(\beta+1)^{-s}(e_b^{\Lambda})^t=(\beta+1)^{-s}(e_c^\Lambda)^t.
    \end{split}
 \end{equation}
 Moreover, in~\eqref{eq:Z1segm} the only matrices that change are the following:
 \begin{equation}\label{eq:Z1loop.1}
 \begin{split}
     I_2-A & =\begin{bmatrix}
         1-(\alpha+1)^{-s} & -\beta^{-s}\\
         -\alpha^{-s} & 1-(\beta+1)^{-s}
     \end{bmatrix};\\
     I_2+(\underline{1}-I_2)A & =\begin{bmatrix}
         1+\alpha^{-s} & (\beta+1)^{-s}\\
         (\alpha+1)^{-s} & 1+\beta^{-s}
     \end{bmatrix}.
    \end{split}
 \end{equation}
 Analogously to~\eqref{eq:Z1segm.2}, we deduce that
 \begin{equation}\label{eq:Z1loop.2}
 \begin{split}
     \caZ_{\Gamma,a\to a}(s)^{-1} & =\frac{\det\Big(I_2-A-\overbrace{B(I_\Lambda-D)^{-1}C}^{=X}\Big)}{\det\Big(I_2+(\underline{1}-I_2)A+\underbrace{(\underline{1}-I_2)B(I_\Lambda-D)^{-1}C}_{=Y}\Big)}.
 \end{split}
 \end{equation}
 In this case $X=Y$, because $B_1=B_2$ and then $(\underline{1}-I_2)B=B$. Moreover, recalling Theorem~\ref{thm:meroWLIT} and Fact~\ref{fact:mdl}, we have
 \begin{equation}\label{eq:Z1loop.3}
     \begin{split}
         X(a,a)& =X(\bar{a},a)=B_1(I_\Lambda-D)^{-1}C^1\\
               & =(\alpha+1)^{-s}\cdot\sum_{b\in o^{-1}(c)\cap E\Lambda}\omega(b)^{-s}e_b^\Lambda(I_\Lambda-\caE^\Lambda(s))^{-1}(e_c^\Lambda)^t\\
               & =(\alpha+1)^{-s}(\caZ_{\Lambda,c\to c}(s)-1).
    \end{split}
 \end{equation}
 Similarly, 
 \begin{equation}\label{eq:Z1loop.4}
     X(a,\bar{a})=X(\bar{a},\bar{a})=B_1(I_\Lambda-D)^{-1}C^2=(\beta+1)^{-s}(\caZ_{\Lambda,c\to c}(s)-1).
 \end{equation}
 The statement now follows by~\eqref{eq:Z1loop.3}, \eqref{eq:Z1loop.4} and elementary algebraic manipulations.
\end{proof}

\section{The behaviour at $s=-1$}\label{s:euler}
The main goal of this section is to prove that the Euler--Poincaré identity~\eqref{eq:chi-1} holds for a unimodular t.d.l.c.~group $G$ having an action on a tree as prescribed by Corollary~\hyperref[corF]{F}. To achieve this, we first give a formula of the relevant Euler--Poincaré characteristic in terms of local data of the action (cf.~Proposition~\ref{prop:chitree}) and then use the splitting formulae of Section~\ref{s:split} to prove Theorem~\hyperref[thmE]{E} and Corollary~\hyperref[corF]{F}. 
Finally, given an edge-weighted graph $(\Gamma,\omega)$, we exploit the machinery introduced in this paper to relate the behaviour of $\caZ_{\Gamma,u\to u}(-1)$ with the behaviour at $1$ of a suitable weighted Ihara zeta function associated to $\Gamma$ (cf.~Section~\ref{sus:ih-1}).

\subsection{The Euler--Poincaré characteristic}\label{sus:eul}
According to~\cite[\S5]{ccw:zeta}, every unimodular t.d.l.c.~group $G$ of type $\mathrm{FP}$ (with respect to the category of discrete left $\QG$-modules) admits an \emph{Euler--Poincaré characteristic} $\tchi_G$. For every compact open subgroup $K\leq G$, this invariant is a determined rational multiple of the Haar measure $\mu_K$ on $G$ normalised with respect to $K$, written $\tchi_G=\chi(G,\mu_K)\cdot \mu_K$. By the uniqueness of the Haar measure on $G$ up to a positive real rescaling, for every Haar measure $\mu$ on $G$ there is a unique real number $\chi(G,\mu)$ such that
\begin{equation*}
    \widetilde{\chi}_G=\chi(G,\mu)\cdot \mu.
\end{equation*}
Moreover, if $\mu$ and $\mu'=c\cdot \mu$ ($c\in \R_{>0}$) are Haar measures on $G$, then
\begin{equation}\label{eq:mumu'}
    \chi(G,\mu)=c\cdot\chi(G,\mu').
\end{equation} 

In the present note, we focus on the case of t.d.l.c.~groups acting on a tree with compact open vertex stabilisers and finite quotient graph. For those groups, their unimodularity and their Euler--Poincaré characteristic can be characterised in terms of local data of the action as shown in Proposition~\ref{prop:unimod} and Proposition~\ref{prop:chitree}, respectively.

\begin{prop}[\protect{\cite[Propositions~1.2~and~3.6]{baku}, \cite[\S 3.6]{car}}]\label{prop:unimod}
    Let $G$ be a t.d.l.c. group acting on a tree $T$ with compact open vertex stabilisers and without edge inversions. Let $\Gamma$ be the quotient graph of $(G,T)$ and denote by $\omega$ its standard edge weight. Then $G$ is unimodular if and only if, for every closed path $(a_1, \ldots, a_n)$ in~$\Gamma$,
    \begin{equation}\label{eq:uni}
        \prod_{i=1}^n\omega(a_i)=\prod_{i=1}^n\omega(\bar{a}_i).
    \end{equation}
\end{prop}

\begin{rem}\label{rem:Nuni}
    Let $\frp=(a_1,\ldots, a_m)$ and $\frq=(b_1,\ldots, b_n)$  be reduced paths in~$\Gamma$ with $o(a_1)=o(b_1)$ and $t(a_m)=t(b_n)$. Hence $\Nv(\frp)=\prod_{i=1}^m\omega(a_i)$. Moreover, $\Ne(\frp)=1$ if $m=1$ and $\Ne(\frp)=\prod_{i=2}^m\omega(a_i)$ if $m\geq 2$. Similar observations hold for $\overline{\frp}$, $\frq$ and $\overline{\frq}$.
    By~Proposition~\ref{prop:unimod}, we deduce that
    \begin{equation*}
    \Nv(\frp)\Nv(\overline{\frq})=\Nv(\frq)\Nv(\overline{\frp}).
    \end{equation*}
    If in particular $a_m=b_n$, we also have
    \begin{equation*}
    \Nv(\frp)\Ne(\overline{\frq})=\Nv(\frq)\Ne(\overline{\frp}).
    \end{equation*}
    Moreover, if $a_1=b_1$ and $a_m=b_n$ then
    \begin{equation*}
    \Ne(\frp)\Ne(\overline{\frq})=\Ne(\frq)\Ne(\overline{\frp}).
    \end{equation*}
\end{rem}

\begin{thm}[\protect{\cite[Theorem~5.6]{ccw:zeta}}]\label{thm:chitree}
    Let~$G$ be a unimodular t.d.l.c.~group acting on a tree $T$ without edge inversions, with compact open vertex stabilisers and with finite quotient graph. Let~$\caV\subseteq VT$ and~$\caE^+\subseteq ET$ be sets of representatives for the $G$-orbits on~$VT$ and on a fixed orientation~$ET^+$ in~$T$, respectively. Then, for every Haar measure $\mu$ on~$G$, 
    \begin{equation*}
        \chi(G,\mu)=\sum_{v\in \caV}\frac{1}{\mu(G_v)}-\sum_{e\in \caE^+}\frac{1}{\mu(G_e)}.
    \end{equation*}
 \end{thm}
The hypotheses of Theorem~\ref{thm:chitree} imply that $T$ is locally finite. In particular, in this case the quantity $\chi(G,\mu)$ coincides with the Euler--Poincaré characteristic of $G$ with respect to $\mu$ as defined in~\cite[Definition~4.8]{pesato}.

\begin{prop}\label{prop:chitree}
Let $G$ be a unimodular t.d.l.c.~group acting on a tree~$T$ without edge inversions, with compact open vertex stabilisers, with finite quotient graph $\Gamma$, and such that $(G,T)$ is weakly locally $\infty$-transitive or (P)-closed. 
Let $\omega$ be the standard edge weight on $\Gamma$, and let $\Nv=\Nv^\omega$, $\Ne=\Ne^\omega$ be as in Definition~\ref{defn:N}. Let $c\in V\Gamma$ and $\Lambda\subseteq \Gamma$ be a maximal subtree, and consider an orientation $E\Lambda^+$ in $\Lambda$ such that the restricted origin map $o\colon E\Lambda^+\to V\Lambda\setminus \{c\}$ is a bijection. Let also $E\Gamma^+$ be an arbitrary orientation in $\Gamma$ such that $E\Gamma^+\cap \Lambda=E\Lambda^+$.

Then, for every $v\in VT$ with $\pi(v)=c$, we have
\begin{equation}\label{eq:chitree}
        \chi(G, \mu_{G_v})=1+\sum_{a\in E\Lambda^+}(1-\omega(a))\frac{\Nv(\frp_{c,o(a)})}{\Nv(\overline{\frp_{c,o(a)}})}-\sum_{b\in E\Gamma^+\setminus E\Lambda}\frac{\Nv(\frq_{c,b})}{\Ne(\overline{\frq_{c,b}})},
    \end{equation}
    for arbitrary reduced paths $\frp_{c,o(a)}\in\calP_{\Gamma,v}(c\to o(a))$ and $\frq_{c,b}\in\calP_{\Gamma,v}(c\to b)$, for all $a\in E\Lambda^+$ and $b\in E\Gamma^+\setminus E\Lambda$.
\end{prop}

 \begin{rem}\label{rem:indepchiT}
In the following, we comment on the choices made in the statement of Proposition~\ref{prop:chitree}.
\begin{itemize}
    \item[(i)] Let $\Lambda$ be a finite tree. For every $c\in V\Lambda$, there is an orientation $E\Lambda^+$ for which the origin map restricts to a bijection $o\colon E\Lambda^+\to V\Lambda\setminus\{c\}$.
    
    Indeed, let $\caE^+$ be an arbitrary orientation in~$\Lambda$ and set
    \begin{equation*}
    E\Lambda^+:=\{a\mid a\in \caE^+,\,o(a)\neq c\}\sqcup\{\bar{a}\mid a\in\caE^+,\,o(a)=c\}.
    \end{equation*}
    Then $E\Lambda^+$ is an orientation. Moreover, since $o(a)\neq t(a)$ for every $a\in E\Lambda$, the origin map in~$\Lambda$ restricts to a map $o\colon E\Lambda^+\to V\Lambda\setminus\{c\}$. By~\cite[\S I.2, Proposition~12]{ser:trees} we have $|E\Lambda^+|=|V\Lambda|-1$ and then $o\colon E\Lambda^+\to V\Lambda\setminus\{c\}$ is bijective.
    \item[(ii)] By~Remark~\ref{rem:Nuni}, the right-hand side of~\eqref{eq:chitree} does not depend on the choice of specific reduced paths $\frp_{c,o(a)}$ and $\frq_{c,b}$ from $c$ to $o(a)$ and from $c$ to $b$, respectively. Being~$\Gamma$ connected, note that there is always a reduced path connecting any two vertices of~$\Gamma$. Reduced paths in~$\Gamma$ are always liftable to geodesics in~$T$ (even if~$T$ has leaves).
    \item[(iii)] A formula analogous to the one in~\eqref{eq:chitree} holds for $\chi(G,\mu_{G_e})$, $e\in ET$. Indeed, since $\mu_{G_e}=|G_v:G_e|\cdot\mu_{G_v}=\omega(\pi(e))\cdot \mu_{G_v}$, by~\eqref{eq:mumu'} we have
$$\chi(G,\mu_{G_e})=\omega(\pi(e))^{-1}\chi(G,\mu_{G_{o(e)}}).$$
   Then Proposition~\ref{prop:chitree} applies. 
\end{itemize}
\end{rem}

\begin{proof}[Proof of Proposition~\ref{prop:chitree}]
 Let $\pi\colon T\to \Gamma$ be the quotient map and consider a set of representatives $\caE^+\subseteq ET$ for $E\Gamma^+$. Up to replace elements of $E\Gamma^+\setminus E\Lambda$ with their reverse, we may assume that for every $e\in \caE^+$ with $\pi(e)\not\in E\Lambda$ the geodesic from $v$ to $e$ is defined in~$T$. Since $o\colon E\Lambda^+\to V\Gamma\setminus\{c\}$ is bijective, notice that $\caV:=\{v\}\sqcup \{o(e)\mid e\in \caE^+\text{ and }\pi(e)\in E\Lambda\}$ is a set of representatives for $V\Gamma$. 
 Moreover, $\mu_{G_v}(G_v)=1$ and, for every $e\in \caE^+$,  $$\mu_{G_v}(G_{o(e)})=|G_{o(e)}:G_e|\cdot\mu_{G_v}(G_e)=\omega(\pi(e))\cdot\mu_{G_v}(G_e).$$ 
 By Theorem~\ref{thm:chitree}, we have
 \begin{equation*}\label{eq:ct1}
 \begin{split}
    \chi(G, \mu_{G_v}) & =1+\sum_{e\in\caE^+: \atop \pi(e)\in E\Lambda^+}\Bigg(\frac{1}{\mu_{G_v}(G_{o(e)})}-\frac{1}{\mu_{G_v}(G_e)}\Bigg)-\sum_{e\in \caE^+: \atop \pi(e)\in E\Gamma^+\setminus\Lambda}\frac{1}{\mu_{G_v}(G_e)}\\
    & = 1+\sum_{e\in\caE^+: \atop \pi(e)\in E\Lambda^+}(1-\omega(\pi(e)))\frac{1}{\mu_{G_v}(G_{o(e)})}-\sum_{e\in \caE^+: \atop \pi(e)\in E\Gamma^+\setminus\Lambda}\frac{1}{\mu_{G_v}(G_e)}.
 \end{split} 
 \end{equation*}
 Let $e\in \caE^+$ and~$\pi(e)=a$. For $t\in\{o(e),e\}$, consider the geodesic $[v,t]=(e_1, \dots, e_n)$ in~$T$ lifting  $\frp_{c,o(a)}$ if $t=o(e)$, and lifting $\frq_{c,a}$ if $t=e$. 
 Then, 
 \begin{equation}\label{eq:ct2}
 \begin{split}
     \frac{1}{\mu_{G_v}(G_{t})}& =\frac{\overbrace{|G_v:G_v\cap G_t|\cdot\mu_{G_v}(G_v\cap G_t)}^{=\mu_{G_v}(G_v)=1}}{\underbrace{|G_t:G_v\cap G_t|\cdot\mu_{G_v}(G_v\cap G_t)}_{=\mu_{G_v}(G_t)}}=\frac{|G_v:G_{[v,t]}|}{|G_t: G_{\overline{[v,t]}}|}\\
     & =\frac{|G_v\cdot e_1|}{|G_t\cdot \bar{e}_n|}\prod_{k=1}^{n-1}\frac{|G_{(e_1, \ldots, e_k)}\cdot e_{k+1}|}{|G_{(\bar{e}_n, \ldots, \bar{e}_{k+1})}\cdot \bar{e}_k|}.
     \end{split}
 \end{equation}
 For the latter equality in~\eqref{eq:ct2}, see~\eqref{eq:orb1} and~\eqref{eq:orb2}.
 Note that $|G_v\cdot e_1|=\omega(\pi(e_1))$. Moreover, $|G_t\cdot \bar{e}_n|=1$ if $t=e$ (because $e_n=t$) and $|G_t\cdot \bar{e}_n|=\omega(\pi(\bar{e}_n))$ if $t=o(e)$ (because $o(\bar{e}_n)=t$). For $1\leq k\leq n-1$ we claim that 
 \begin{equation}\label{eq:ct2'}
 |G_{(e_1, \ldots, e_k)}\cdot e_{k+1}|=|G_{e_k}\cdot e_{k+1}|\quad\text{and}\quad|G_{(\bar{e}_n, \ldots, \bar{e}_{k+1})}\cdot \bar{e}_k|=|G_{\bar{e}_{k+1}}\cdot \bar{e}_k|.
 \end{equation}
 If $(G,T)$ is (P)-closed, \eqref{eq:ct2'} follows from~\eqref{eq:orbEdge}. If $(G,T)$ is weakly locally $\infty$-transitive, Remark~\ref{rem:WLIT}(ii) and Proposition~\ref{prop:subWLIT} yield
 \begin{equation*}
     |G_{(e_1,\ldots, e_k)}\cdot e_{k+1}|=\Ne(\pi(e_k),\pi(e_{k+1}))=|G_{e_k}\cdot e_{k+1}|.
 \end{equation*}
 A similar argument holds for $|G_{(\bar{e}_n, \ldots, \bar{e}_{k+1})}\cdot \bar{e}_k|$.
 
 For $1\leq k\leq n-1$, we now prove that
 \begin{equation}\label{eq:ct3}
     \frac{|G_{e_k}\cdot e_{k+1}|}{|G_{\bar{e}_{k+1}}\cdot \bar{e}_k|}=\frac{\omega(\pi(e_{k+1}))}{\omega(\pi(\bar{e}_k))}.
 \end{equation}
 To see this, set $v_k=t(e_k)$ and $H=G_{e_k}\cap G_{e_{k+1}}$. Since $G_{e_k}=G_{\bar e_k}$ and $G_{e_{k+1}}=G_{\bar{e}_{k+1}}$, we have
 \begin{equation*}
    |G_{v_k}:H|=|G_{v_k}:G_{e_{k+1}}|\cdot |G_{\bar e_{k+1}}:G_{(\bar e_{k+1},\bar{e}_k)}|=\omega(\pi(e_{k+1}))\cdot |G_{\bar e_{k+1}}\cdot \bar{e}_k|
 \end{equation*}
 and, at the same time, 
 \begin{equation*}
|G_{v_k}:H|=|G_{v_k}:G_{\bar{e}_k}|\cdot |G_{e_k}:G_{(e_k,e_{k+1})}|=\omega(\pi(\bar{e}_k))\cdot |G_{e_{k}}\cdot e_{k+1}|.
 \end{equation*}
 Combining~\eqref{eq:ct2}, \eqref{eq:ct2'} and~\eqref{eq:ct3}, we deduce that
 \begin{equation}\label{eq:ct4}
     \frac{1}{\mu_{G_v}(G_t)}=\frac{\omega(\pi(e_1))}{|G_t\cdot \bar{e}_n|}\prod_{k=1}^{n-1}\frac{\omega(\pi(e_{k+1}))}{\omega(\pi(\bar{e}_k))},
 \end{equation}
 where $|G_t\cdot \bar{e}_n|$ equals $1$ if $t=e$, and it equals $\omega(\pi(\bar{e}_n))$ if $t=o(e)$.
 By design, $\pi([v,t])=:(a_1, \ldots, a_n)$ is a reduced path in~$\Gamma$. Then
 \begin{gather}\label{eq:Ngeod}
     \Nv(\pi([v,t]))=\prod_{i=1}^n\omega(a_i)\,\text{ and }\,\Ne(\pi([v,t]))=\left\{
     \begin{array}{cl}
         1, & \text{if }n=1; \\
         \prod_{i=2}^n\omega(a_i), & \text{if }n\geq2.
     \end{array}
     \right.     
 \end{gather}

 By~\eqref{eq:ct4} and~\eqref{eq:Ngeod}, we conclude that
 \begin{equation*}
 \begin{split}
     \frac{1}{\mu_{G_v}(G_{o(e)})}& =\frac{\Nv(\frp_{c,o(a)})}{\Nv(\overline{\frp_{c,o(a)}})},\quad\forall\,e\in \caE^+\text{ with }\pi(e)\in E\Lambda;\\
     \frac{1}{\mu_{G_v}(G_{e})} &=\frac{\Nv(\frq_{c,a})}{\Ne(\overline{\frq_{c,a}})}, \quad\forall\,e\in \caE^+\text{ with }\pi(e)\not\in E\Lambda.\qedhere
    \end{split}
 \end{equation*}
\end{proof}

\begin{rem}\label{rem:chi=}
    Let $(G_1,T_1)$ and $(G_2,T_2)$ be group actions on trees that satisfy the hypotheses of Proposition~\ref{prop:chitree}. Let $(\Gamma_1,\omega_1)$ and $(\Gamma_2,\omega_2)$ be the quotient graphs of $(G_1,T_1)$ and $(G_2,T_2)$ endowed with their standard edge weights, respectively. Assume there is a graph isomorphism $\varphi\colon \Gamma_1\to\Gamma_2$ such that $\omega_2(\varphi(a))=\omega_1(a)$ for every $a\in E\Gamma_1$. Let $v_1\in VT_1$ and $v_2\in VT_2$ be vertices satisfying $G_1\cdot v_1=c_1$ and $G_2\cdot v_2=\varphi(c_1)$.
    By Proposition~\ref{prop:chitree},
    \begin{equation*}
        \chi(G_1, \mu_{(G_1)_{v_1}})=\chi(G_2,\mu_{(G_2)_{v_2}}),
    \end{equation*}
    where $\mu_{(G_i)_{v_i}}$ is the Haar measure of $G_i$ normalised with respect to~$(G_i)_{v_i}$.
\end{rem}

A notable consequence of Proposition~\ref{prop:chitree} is that the value $\chi(G,\mu_{G_v})$ 
depends only on $(\Gamma,\omega)$. 
This suggests the following definition.

\begin{defn}\label{defn:chiGamma}
    Let $\Gamma$ be a finite connected non-empty graph with an edge weight $\omega\colon E\Gamma\to \Z_{\geq 1}$. 
    Let $(T,\pi)$ be the universal cover of $(\Gamma,\omega)$, and set $G=\Aut_\pi(T)$ (cf.~Example~\ref{exintro}).
    The pair $(\Gamma,\omega)$ is said to be \emph{unimodular} if $\Aut_\pi(T)$ is unimodular. 
    
   Let $(\Gamma, \omega)$ be unimodular. For all $c\in V\Gamma$ and $a\in E\Gamma$ and given arbitrary $v\in VT$ and $e\in ET$ satisfying $\pi(v)=c$ and $\pi(e)=a$, define
    \begin{equation}\label{eq:chiG}
\chi(\Gamma,c):=\chi(G,\mu_{G_v})\quad\text{and}\quad\chi(\Gamma,a):=\chi(G,\mu_{G_e}).
    \end{equation}
\end{defn}
Since $G$ is unimodular, the assignments in~\eqref{eq:chiG} do not depend on the choice of $v\in \pi^{-1}(c)$ and of $e\in \pi^{-1}(a)$, respectively.

\begin{rem}\label{rem:chiGamma}
Let $G$ be a unimodular t.d.l.c.~group acting on a tree~$T$, without edge inversions, with compact open vertex-stabilisers and with finite quotient graph~$\Gamma$. Denote by $\omega$ the standard edge weight, and assume that $(G,T)$ is weakly locally $\infty$-transitive or (P)-closed. For every $t\in T$ with $G\cdot t=u$, from Proposition~\ref{prop:chitree} we have
\begin{equation*}
    \chi(\Gamma,u)=\chi(G,\mu_{G_{t}}).
\end{equation*}
\end{rem}

\begin{ex}\label{ex:chi}
    \begin{itemize}
        \item[(i)] Let $\Gamma$ be a $1$-segment with $E\Gamma=\{a,\bar{a}\}$. Since $\Gamma$ is a tree, $(\Gamma,\omega)$ is unimodular for every $\omega\colon  E\Gamma\to\Z_{\geq 1}$ (cf.~Proposition~\ref{prop:unimod}). Moreover, $\Gamma$ is its only maximal subtree. Consider an edge weight $\omega\colon E\Gamma\to \Z_{\geq 1}$. Set $c=o(a)$, $E\Gamma^+=\{\bar{a}\}$ and let $\frp_{c,o(\bar{a})}$ be the $1$-edge path $a$. Then Proposition~\ref{prop:chitree} implies that
        \begin{equation*}
            \chi(\Gamma,c)=1+(1-\omega(\bar{a}))\frac{\omega(a)}{\omega(\bar{a})}.
        \end{equation*}
        With a similar strategy one computes $\chi(\Gamma,t(a))$.
        \item[(ii)] Let $\Gamma$ be a $n$-bouquet of loops based on the vertex $c$. Note that the $1$-point subgraph is the only maximal subtree of $\Gamma$. By Proposition~\ref{prop:unimod}, for an edge weight $\omega\colon E\Gamma\to \Z_{\geq 1}$ the pair $(\Gamma,\omega)$ is unimodular if, and only if, $\omega(a)=\omega(\bar{a})$ for every $a\in E\Gamma$. Provided $(\Gamma,\omega)$ is unimodular and $E\Gamma=\{a_i,\bar{a}_i\mid 1\leq i\leq n\}$, from Proposition~\ref{prop:chitree} we deduce that 
        \begin{equation*}
            \chi(\Gamma,c)=1-\sum_{i=1}^n\omega(a_i).
        \end{equation*}
    \end{itemize}
\end{ex}

\begin{lem}\label{lem:chichange}
    Let $\Gamma$ be a finite connected non-empty graph, and let $\omega\colon E\Gamma\to \Z_{\geq 1}$ be such that $(\Gamma,\omega)$ is unimodular. Then, for every $a\in E\Gamma$,
    \begin{equation}\label{eq:chich1}
        \chi(\Gamma,o(a))=\omega(a)\cdot\chi(\Gamma,a).
    \end{equation}
    Moreover, for all $c,d\in V\Gamma$,
    \begin{equation}\label{eq:chichV}
        \chi(\Gamma,c)=\frac{\Nv(\frp)}{\Nv(\overline{\frp})} \chi(\Gamma,d), 
     \end{equation}
     where $\frp$ is any reduced path in~$\Gamma$ from $c$ to $d$. Similarly, for all $a,b\in E\Gamma$ for which there is a reduced path in~$\Gamma$ from $a$ to $b$, we have
    \begin{equation}\label{eq:chichE}
        \chi(\Gamma,a)=\frac{\Ne(\frq)}{\Ne(\overline{\frq})}\chi(\Gamma,b),
    \end{equation}
    where $\frq$ is any reduced path from $a$ to $b$ in $\Gamma$.
\end{lem}

In Lemma~\ref{lem:chichange}, since $\Gamma$ is connected, replacing $a$ with $\bar{a}$ or $b$ with $\bar{b}$ if necessary, we can always find a reduced path from $a$ to $b$ in $\Gamma$.

\begin{proof}
   First, \eqref{eq:chich1} follows from Remark~\ref{rem:indepchiT}(iii).
   By~Remark~\ref{rem:Nuni}, the ratios in~\eqref{eq:chichV} and~\eqref{eq:chichE} do not depend on the choices of $\frp$ and $\frq$, respectively. 
   Moreover, if we prove~\eqref{eq:chichV} and~\eqref{eq:chichE} for $\ell(\frp)=1$ and $\ell(\frq)=2$, the general statements follow iteratively.
   It remains to observe what follows. First, for every $a\in E\Gamma$ we have
    \begin{equation}\label{eq:chichangeV}
\chi(\Gamma,o(a))=\omega(a)\cdot\chi(\Gamma,a)=\frac{\omega(a)}{\omega(\bar{a})}\chi(\Gamma,t(a)).
    \end{equation}
  Moreover, let $(a,b)$ is a length-$2$ reduced path in~$\Gamma$ and set $t(a)=c=o(b)$. Then~\eqref{eq:chich1}~and~\eqref{eq:chichangeV} imply that
   \begin{equation*}
       \chi(\Gamma,a)=\chi(\Gamma,\bar{a})=\frac{1}{\omega(\bar{a})}\chi(\Gamma,c)=\frac{\omega(b)}{\omega(\bar{a})}\chi(\Gamma,b).\qedhere
   \end{equation*}
\end{proof}

\begin{lem}\label{lem:chiProp}
    Let $\Gamma$ be a finite connected non-empty graph with an edge weight $\omega\colon E\Gamma\to \Z_{\geq 1}$ such that $(\Gamma,\omega)$ is unimodular. 
     Suppose that there are connected subgraphs $\Gamma_1$ and $\Gamma_2$ of $\Gamma$ such that $\Gamma=\Gamma_1\cup \Gamma_2$ and $\Gamma_1\cap \Gamma_2=\{c\}$, for some $c\in V\Gamma$. Then $(\Gamma_i,\omega{\vert_{E\Gamma_i}})$ is unimodular for every $i\in\{1,2\}$, and 
    \begin{equation}
        \chi(\Gamma,c)=\chi(\Gamma_1,c)+\chi(\Gamma_2,c)-1.
    \end{equation}
\end{lem}

\begin{proof}
Let $\Lambda$ be a maximal subtree of $\Gamma$. We claim that $\Lambda_i:=\Gamma_i\cap \Lambda$ is a maximal subtree of $\Gamma_i$, for every $i\in\{1,2\}$. Clearly, both $\Lambda_1$ and $\Lambda_2$ are subtrees of $\Gamma_1$ and $\Gamma_2$, respectively. We prove the maximality for $i=1$, as for $i=2$ one may proceed analogously. For every subtree $\Xi_1\subseteq \Gamma_1$ with $\Xi_1\supseteq \Lambda_1$, we have $\Xi_1\cap\Lambda_2=\{c\}$ and thus $\Xi_1\cup\Lambda_2$ is a subtree of $\Gamma$ containing~$\Lambda$. Hence $\Lambda=\Xi_1\cup \Lambda_2$ and 
$$\Lambda_1=\Gamma_1\cap \Lambda=(\Gamma_1\cap \Xi_1)\cup (\Gamma_1\cap \Lambda_2)=\Xi_1\cup \{c\}=\Xi_1.$$  

Consider an orientation $E\Gamma^+$ in $E\Gamma$ such that the origin map in $\Gamma$ restricts to a bijection $o\colon E\Gamma^+\cap E\Lambda\to V\Gamma\setminus\{c\}$ (cf.~Remark~\ref{rem:indepchiT}). For every $i\in\{1,2\}$, the set $E\Gamma_i^+:=E\Gamma^+\cap E\Gamma_i$ is an orientation in $E\Gamma_i$ and the origin map in $\Gamma_i$ restricts to a bijection $o_i\colon E\Gamma_i^+\cap E\Lambda_i\to V\Gamma_i\setminus\{c\}$.
By~Proposition~\ref{prop:chitree}, we conclude that
    \begin{equation*}
    \begin{split}
        \chi(\Gamma,c) & =
    1+\sum_{i=1}^2\Bigg (\sum_{a\in E\Gamma_i^+\cap E\Lambda_i}\hspace{-0.2cm}(1-\omega(a))\frac{\Nv(\frp_{c,o(a)})}{\Nv(\overline{\frp_{c,o(a)}})}-\hspace{-0.4cm}\sum_{a\in E\Gamma_i^+\setminus E\Lambda_i}\hspace{-0.2cm}\frac{\Nv(\frq_{c,a})}{\Ne(\overline{\frq_{c,a}})}\Bigg )\\
    & =\chi(\Gamma_1,c)+\chi(\Gamma_2,c)-1.
    \end{split}
    \end{equation*}
\end{proof}

\subsection{The evaluation at $s=-1$ and the Euler--Poincaré characteristic}\label{sus:eul-1}

The goal of what follows is to prove Theorem~\hyperref[thmE]{E} and~Corollary~\hyperref[corF]{F}.
In view of Theorem~\hyperref[thmE]{E}, we first formulate a version of Lemma~\ref{lem:chichange} for $\caZ_{\Gamma,u\to u}(-1)^{-1}$. 

\begin{lem}\label{lem:chi-1VE}
    Let $(\Gamma,\omega)$ be a unimodular edge-weighted graph satisfying Setting~\hyperref[settG]{\emph{[$\Gamma$]}} and such that $\Gamma$ has no cycles of length $\geq 2$. Then, for every $a\in E\Gamma$,
    \begin{equation}\label{eq:change1}
        \caZ_{\Gamma,o(a)\to o(a)}(-1)^{-1}=\omega(a)\cdot \caZ_{\Gamma,a\to a}(-1)^{-1}.
    \end{equation}
    Moreover, for all $c,d\in V\Gamma$ and all $a,b\in E\Gamma$ such there is a reduced path from $a$ to $b$ in~$\Gamma$, we have
    \begin{gather}\label{eq:change2}
    \begin{array}{c}
         \displaystyle{\caZ_{\Gamma,c\to c}(-1)^{-1}=\frac{\Nv(\frp)}{\Nv(\overline{\frp})}\caZ_{\Gamma,d\to d}(-1)^{-1},}\\
        \displaystyle{\caZ_{\Gamma,a\to a}(-1)^{-1}=\frac{\Ne(\frq)}{\Ne(\overline{\frq})}\caZ_{\Gamma,b\to b}(-1)^{-1},}
    \end{array}
    \end{gather}
    where $\frp$ and $\frq$ are arbitrary reduced paths in~$\Gamma$ from $c$ to $d$ and from $a$ to $b$, respectively.
\end{lem}

\begin{proof}
    Once proved~\eqref{eq:change1} (which is analogous to~\eqref{eq:chich1}), arguing as in the proof of Lemma~\ref{lem:chichange} one can deduce~\eqref{eq:change2}.
    We first prove~\eqref{eq:change1} for every $1$-loop $a$. Namely, let $a\in E\Gamma$ with $o(a)=t(a)=c$. Since $(\Gamma,\omega)$ is unimodular, note that $\omega(a)=\omega(\bar{a})$. Let $\Lambda$ be the graph obtained from $\Gamma$ removing $a$ and $\bar{a}$, and let $L_a$ be the subgraph of $\Gamma$ with $VL_a=\{c\}$ and $EL_a=\{a,\bar{a}\}$.
    By Remark~\ref{rem:Zaloop}, we have
    \begin{equation}\label{eq:cha1}
        \caZ_{\Gamma,a\to a}(-1)^{-1}=\omega(a)^{-1}\cdot\Big(\caZ_{\Lambda,c\to c}(-1)^{-1}-\omega(a)\Big).
    \end{equation}
   Moreover, Proposition~\ref{prop:Vsplit} and~\eqref{eq:Z1loopA=B} yield
   \begin{equation}\label{eq:cha2}
       \caZ_{\Gamma,c\to c}(-1)^{-1}=\caZ_{\Lambda,c\to c}(-1)^{-1}-\omega(a).
   \end{equation}
   Combining~\eqref{eq:cha1} and~\eqref{eq:cha2}, we deduce~\eqref{eq:change1}.

   For all edges $a$ with $o(a)\neq t(a)$, the relation in~\eqref{eq:change1} is proved by induction on $|E\Gamma|/2=:k(\Gamma)\geq 1$. If $k(\Gamma)=1$, then $\Gamma$ is a $1$-segment and~\eqref{eq:change1} follows from~\eqref{eq:Z1seg}. 
    Let $k(\Gamma)\geq 2$ and assume that the claim holds for every graph $\Gamma'$ with $k(\Gamma')<k(\Gamma)$. Let $a\in E\Gamma$ be such that $o(a)=:c\neq d:=t(a)$. If $o^{-1}(c)=\{a\}$, then Proposition~\ref{prop:Zasegm} directly implies the claim. In case that $o^{-1}(d)=\{\bar{a}\}$, let $\Lambda$ be the graph obtained from $\Gamma$ by removing $a$ and $\bar{a}$. Then Proposition~\ref{prop:Zasegm} yields
    \begin{equation}\label{eq:1segm-1.1}
        \caZ_{\Gamma,a\to a}(-1)^{-1}=\caZ_{\Gamma,\bar{a}\to\bar{a}}(-1)^{-1}=\frac{\caZ_{\Lambda,c\to c}(-1)^{-1}}{\omega(a)}-\frac{\omega(\bar{a})-1}{\omega(\bar{a})}.
    \end{equation}
    On the other hand, let $\Gamma_a$ denote the $1$-segment subgraph of $\Gamma$ with $E\Gamma_a=\{a,\bar{a}\}$. By Proposition~\ref{prop:Vsplit} and~Example~\ref{ex:1segmWLIT}, 
    \begin{equation}\label{eq:1segm-1.2}
    \begin{split}
        \caZ_{\Gamma,c\to c}(-1)^{-1} & =\caZ_{\Lambda,c\to c}(-1)^{-1}+\caZ_{\Gamma_a,c\to c}(-1)^{-1}-1\\
         & =\caZ_{\Lambda,c\to c}(-1)^{-1}+\omega(a)\frac{\omega(\bar{a})-1}{\omega(\bar{a})}.
    \end{split}
    \end{equation}
    Hence~\eqref{eq:change1} follows from~\eqref{eq:1segm-1.1}~and~\eqref{eq:1segm-1.2}.
    Finally, assume that both $|o^{-1}(c)|\geq 2$ and $|o^{-1}(d)|\geq 2$. Denote by $\Xi_1$ and $\Xi_2$ the connected components of $\Gamma\setminus\{a,\bar{a}\}$ containing $c$ and $d$, respectively. There are exactly two connected components because $\Gamma$ has no cycles of length $\geq 2$. Since $|o^{-1}(c)|\geq 2$ and $|o^{-1}(d)|\geq 2$, both $E\Xi_1$
    and $E\Xi_2$ are non-empty. Moreover, $E\Gamma=E\Xi_1\sqcup \{a,\bar{a}\}\sqcup E\Xi_2$.
    For $i\in\{1,2\}$, let $\Gamma_i$ be the smallest subgraph of $\Gamma$ containing $\Xi_i\cup \{a,\bar{a}\}$, and note that $k(\Gamma_i)<k(\Gamma)$. Let also $\Gamma_a$ be the $1$-segment subgraph with edge set $\{a,\bar{a}\}$, and observe that $\Gamma_1\cap \Gamma_2=\Gamma_a$. Moreover, if $\Lambda_1:=\Xi_1$ and $\Lambda_2:=\Gamma_2$, we have $\Gamma=\Lambda_1\cup\Lambda_2$, $\Lambda_1\cap\Lambda_2=\{c\}$ and $\Lambda_i\subseteq \Gamma_i$ for every $i\in\{1,2\}$.
   Hence, Corollary~\ref{cor:Vsplit} and Corollary~\ref{cor:Esplit} imply
    \begin{equation*}
    \begin{split}
        \caZ_{\Gamma,c\to c}(-1)^{-1} & =\caZ_{\Gamma_1,c\to c}(-1)^{-1}+\caZ_{\Gamma_2,c\to c}(-1)^{-1}-\caZ_{\Gamma_a,c\to c}(-1)^{-1};\\
         \caZ_{\Gamma,a\to a}(-1)^{-1} & =\caZ_{\Gamma_1,a\to a}(-1)^{-1}+\caZ_{\Gamma_2,a\to a}(-1)^{-1}-\caZ_{\Gamma_a,a\to a}(-1)^{-1}.
    \end{split}
    \end{equation*}
    The induction hypothesis now yields~\eqref{eq:change1}.
\end{proof}

By Lemma~\ref{lem:chichange} and Lemma~\ref{lem:chi-1VE}, we deduce the following.
\begin{cor}\label{cor:Z-1indep}
    Let $(\Gamma,\omega)$ be a unimodular edge-weighted graph satisfying Setting~\hyperref[settG]{\emph{[$\Gamma$]}} and such that $\Gamma$ has no cycles of length $\geq 2$. If $\chi(\Gamma,u)=\caZ_{\Gamma,u\to u}(-1)^{-1}$ for some $u\in \Gamma$, then $\chi(\Gamma,u)=\caZ_{\Gamma,u\to u}(-1)^{-1}$ for every $u\in \Gamma$.
\end{cor}
\begin{proof}
     Let $u\in \Gamma$. Since $\Gamma$ is connected, for all $c\in V\Gamma$ and $a\in E\Gamma$ there are reduced paths $\frp\in\calP_\Gamma(U\to c)$ and $\frq\in\calP_\Gamma(U\to A)$ (cf.~Notation~\ref{notat:capitG}). Moreover, note that $\caZ_{\Gamma,b\to b}(s)=\caZ_{\Gamma,\bar{b}\to \bar{b}}(s)$ and $\chi(\Gamma,b)=\chi(\Gamma,\bar{b})$ for every $b\in E\Gamma$. Hence we may assume that $\frp\in\calP_\Gamma(u\to c)$ and $\frq\in\calP_\Gamma(u\to a)$. The statement follows from Lemma~\ref{lem:chichange} and Lemma~\ref{lem:chi-1VE}.
\end{proof}

\begin{proof}[\protect{Proof of Theorem~\hyperref[thmE]{E}}]
    By Corollary~\ref{cor:Z-1indep}, it suffices to prove that
    \begin{equation}\label{eq:chiZ-1.1}
    \caZ_{\Gamma,c\to c}(-1)^{-1}=\chi(\Gamma,c)
    \end{equation}
    for some vertex $c\in V\Gamma$. We prove~\eqref{eq:chiZ-1.1} by induction on $|E\Gamma|/2=:k(\Gamma)\geq 1$. Let first $k(\Gamma)=1$, i.e.,~$E\Gamma=\{a,\bar{a}\}$. If $o(a)\neq t(a)$, Example~\ref{ex:1segmWLIT} and Example~\ref{ex:chi}(i) yield
    \begin{equation*}
        \caZ_{\Gamma,c\to c}(-1)^{-1}=1+(1-\omega(\bar{a}))\frac{\omega(a)}{\omega(\bar{a})}=\chi(\Gamma,c).
    \end{equation*}
    If $o(a)=t(a)$, from~\eqref{eq:Z1loopA=B} and Example~\ref{ex:chi}(ii) we deduce that
    \begin{equation*}
        \caZ_{\Gamma,c\to c}(-1)^{-1}=1-\omega(a)=\chi(\Gamma,c).
    \end{equation*}
    
    Let now $k(\Gamma)\geq 2$ and assume that the statement holds for all graphs $\Gamma'$ with $k(\Gamma')<k(\Gamma)$. Without loss of generality (cf.~Corollary~\ref{cor:Z-1indep}), we may take $c\in V\Gamma$ such that $|o^{-1}(c)|\geq 2$. Note that this vertex exists because $k(\Gamma)\geq 2$. By~Lemma~\ref{lem:GSplit}, there are proper connected subgraphs~$\Gamma_1$ and~$\Gamma_2$ of~$\Gamma$ such that $\Gamma=\Gamma_1\cup\Gamma_2$ and $\Gamma_1\cap\Gamma_2=\{c\}$. Then Proposition~\ref{prop:Vsplit} and Lemma~\ref{lem:chiProp} yield the claim.
\end{proof}

In view of the proof of Corollary~\hyperref[corF]{F}, we observe what follows.
\begin{lem}\label{lem:noUD} 
  Let $\lad$ be a local action diagram. Let $G=U(\Delta,\iota,c_0)$ and $T=T(\Delta,\iota,c_0)$ be as in Setting~\hyperref[settP]{\emph{[(P)-cl]}}.
  Let $\tilde{\Delta}=(\Gamma, (X_a), (\tilde{G}(c)))$ be a local action diagram such that $\tilde{G}:=U(\tilde{\Delta},\iota,c_0)$ acts weakly locally $\infty$-transitively on $T$. Then, for all $t_1,t_2\in T$ we have
  \begin{equation*}
       \zeta_{G, G_{t_1},G_{t_2}}(-1)=\zeta_{\tilde{G},\tilde{G}_{t_1}, \tilde{G}_{t_2}}(-1).
  \end{equation*}
\end{lem}
By Example~\ref{ex:WLIT}(i), note that $\tilde\Delta$ as in Lemma~\ref{lem:noUD} exists.

\begin{proof}
    Both $(G,T)$ and $(\tilde{G},T)$ are (P)-closed actions on trees satisfying Setting~\hyperref[settP]{[(P)-cl]}. The statement now follows by applying Theorem~\ref{thm:meroP} to both $\zeta_{G,G_{t_1},G_{t_2}}(s)$ and $\zeta_{\tilde G,\tilde G_{t_1},\tilde G_{t_2}}(s)$. In detail, in both cases one checks that the matrices $\caF(0)$ and $\caY_{\pi(t_1),\pi(t_2)}(0)$ involved in the statement of Theorem~\ref{thm:meroP}, as well as the integer $\kappa_{\pi(t_1)}(\pi(t_2))$, depend only on $X=\bigsqcup_{a\in E\Gamma}X_a$, on $\Gamma$ and its standard edge weight $\omega$, and on the inversion map $\iota$. The latter quantities do not vary by passing from $\Delta$ to $\tilde{\Delta}$, and the statement follows.
\end{proof}
\begin{proof}[\protect{Proof of Corollary~\hyperref[corF]{F}}]
By~Remark~\ref{rem:redUD} and~Lemma~\ref{lem:noUD}, it suffices to prove the statement for $(G,T)$ being weakly locally $\infty$-transitive. In this case, the claim follows from Remark~\ref{rem:zetaZ}, Remark~\ref{rem:chiGamma} and~Theorem~\hyperref[thmE]{E}.
\end{proof}

\subsection{The behaviour at $s=-1$ and the Ihara zeta function of a weighted graph}\label{sus:ih-1}
In~\cite[\S3]{deit:prim}, a generalisation of the classical Ihara zeta function has been defined for every finite graph $\Gamma$ with a transition weight. Although it is not necessary here, we mention that the finiteness hypothesis on $\Gamma$ can be relaxed.
According to~\cite[Definition~3.3]{deit:prim}, a \emph{transition weight} on a finite graph $\Gamma$ is a map $W\colon E\Gamma\times E\Gamma\longrightarrow \R_{\geq 0}$ such that, whenever $W(a,b)\neq 0$, then $t(a)=o(b)$.

According to the definition of graph in~\cite{deit:prim} (cf.~\cite[Definition~3.1]{deit:prim}), every edge is supposed to be uniquely determined by its endpoints. However, for the results involved below, this hypothesis has no influence and thus we do not assume it.

\smallskip

The \emph{Ihara zeta function} $Z_{(\Gamma,W)}(x)$ of $(\Gamma,W)$ has been defined in~\cite[Definition~3.8]{deit:prim}
as a suitable infinite product of meromorphic functions on $\C$ converging for all $x\in \C$ with $|x|\ll 1$. Here we only need the following characterisation of the reciprocal of $Z_{(\Gamma,W)}(x)$, cf.~\cite[Theorem~3.11]{deit:prim}. Namely, 
\begin{equation}\label{eq:ih}
    Z_{(\Gamma,W)}(x)^{-1}=\det(I-xT),
\end{equation}
where $I$ is the identity matrix of dimension $|E\Gamma|$ and $T=[T(a,b)]_{a,b\in E\Gamma}\in \Mat_n(\R)$ is defined as $T(a,b)=W(a,b)$ for all $a,b\in E\Gamma$ (assuming to have set a total order on $E\Gamma$). The matrix $T$ is called the \emph{Bass operator} of $(\Gamma,W)$ (cf.~\cite[Definition~3.10]{deit:prim}). Note that~\eqref{eq:ih} gives a meromorphic continuation of $Z_{(\Gamma,W)}(x)$ to $\C$.

\begin{ex}\label{ex:ih}
    Let $\Gamma$ be a finite graph with an edge weight $\omega\colon E\Gamma\to\Z_{\geq 2}$. Let $\Ne=\Ne^\omega$ and $\caE(s)$ be as in Definition~\ref{defn:N} and Definition~\ref{defn:Es}, respectively. Consider the map $W=W_{(\Gamma,\omega)}\colon E\Gamma\times E\Gamma\to\Z_{\geq 0}$ defined by
    \begin{equation*}
        W(a,b):=\caE(-1)(a,b)=\left\{
        \begin{array}{cl}
           \Ne(a,b),  &  \text{if }t(a)=o(b);\\
           0,  & \text{otherwise}.
        \end{array}
        \right.
    \end{equation*}
    Then $W$ yields a transition weight on~$\Gamma$. Note that assuming that $\omega(a)\geq 2$ is necessary to have $W(\bar{a},a)\neq 0$, for all $a\in E\Gamma$.
    In particular, by~\eqref{eq:ih} we have
    \begin{equation*}
       Z_{(\Gamma,W)}(x)^{-1}=\det(I-x\caE(-1)).
    \end{equation*}
\end{ex}

\begin{thm}\label{thm:iha}
Let $(\Gamma,\omega)$ be an edge-weighted graph satisfying Setting~\emph{\hyperref[settG]{[$\Gamma$]}}. Let $\Gamma_1, \Gamma_2$ be subgraphs of~$\Gamma$ satisfying $\Gamma=\Gamma_1\cup \Gamma_2$ and such that $\Gamma_1\cap \Gamma_2$ is a $1$-segment with edge set $\{a,\bar{a}\}$. Assume also that $t(a)$ and $o(a)$ are leaves in~$\Gamma_1$ and~$\Gamma_2$, respectively.
Let $W$, $W_1$ and~$W_2$ be the transition weights defined in Example~\ref{ex:ih} on~$\Gamma$, $\Gamma_1$ and~$\Gamma_2$, respectively. Then,
\begin{equation*}
   \frac{\caZ_{\Gamma, a\to a}(-1)}{\caZ_{\Gamma_1, a\to a}(-1)\cdot \caZ_{\Gamma_2, a\to a}(-1)}=\frac{1}{\omega(a)\omega(\bar{a})}\cdot \frac{Z_{(\Gamma,W)}(1)}{Z_{(\Gamma_1,W_1)}(1)\cdot Z_{(\Gamma_2, W_2)}(1)}.
\end{equation*}
\end{thm}

\begin{proof}
Denote by $\caE(-1)$, $\caE_1(-1)$, $\caE_2(-1)$ the Bass operators of $\Gamma$, $\Gamma_1$ and~$\Gamma_2$ at~$-1$, respectively. Let also $I, I_1$ and $I_2$ denote the identity matrices with complex entries of dimension $|E\Gamma|$, $|E\Gamma_1|$ and $|E\Gamma_2|$, respectively.
    By Theorem~\ref{thm:meroWLIT} and Example~\ref{ex:ih}, 
    \begin{equation}\label{eq:ihGamma}
        \caZ_{\Gamma, a\to a}(-1)=Z_{(\Gamma,W)}(1)\cdot \det(I-M),   
    \end{equation}
    where $$M=[M(h,k)]_{h,k\in E\Gamma}:=\caE(-1)-\caU_{a,a}(-1)=(I-(e_a+e_{\bar{a}})^t (e_a+e_{\bar{a}}))\caE(-1).$$ 
     For all $h,k\in E\Gamma$, observe that
     \begin{equation}\label{eq:M}
    \begin{array}{ccl}
    M(h,k) & = & e_h\Big(I-(e_a+e_{\bar{a}})^t (e_a+e_{\bar{a}})\Big)\caE(-1) e_k^t\\
            & = & e_h \caE(-1)e_k^t - e_h(e_a^t+e_{\bar{a}}^t)\Big((e_a+e_{\bar{a}})\caE(-1) e_k^t\Big)\\
            & = & \caE(-1)(h,k)-\mathbbm{1}_{\{a,\bar{a}\}}(h)\Big(\caE(-1)(a,k)+\caE(-1)(\bar{a},k)\Big).
    \end{array}
    \end{equation}
   Similarly, for every $i\in\{1,2\}$ we have
    \begin{equation}\label{eq:ihGammai}
        \caZ_{\Gamma_i, a\to a}(-1)=Z_{(\Gamma_i,W_i)}(1)\cdot \det(I_i-M_i),   
    \end{equation}
    where $M_i=[M_i(h,k)]_{h,k\in E\Gamma_i}$ is the $|E\Gamma_i|$-dimensional given by
    \begin{equation}\label{eq:Mi}
        M_i(h,k)=\caE_i(-1)(h,k)-\mathbbm{1}_{\{a,\bar{a}\}}(h)\Big(\caE_i(-1)(a,k)+\caE_i(-1)(\bar{a},k)\Big).
    \end{equation}
    By~\eqref{eq:M}~and~\eqref{eq:Mi}, for every $i\in\{1,2\}$ we deduce that 
    \begin{equation}\label{eq:Mi'}
        M(h,k)=M_i(h,k), \quad\forall\,h,k\in E\Gamma_i.
    \end{equation}
    Let $\tilde{M}_1:=[M(h,k)]_{h,k\in E\Gamma_1\setminus\{\bar{a}\}}$ and $\tilde{M}_2:=[M(h,k)]_{h,k\in E\Gamma_2\setminus\{a\}}$. Set also $\tilde{I}_1$ and $\tilde{I}_2$ be the identity matrices in~$\Mat_{|E\Gamma_1|-1}(\C)$ and $\Mat_{|E\Gamma_2|-1}(\C)$, respectively.
    We claim that $M$, $M_1$ and $M_2$ have the following decompositions in diagonal blocks:
    \begin{gather}\label{eq:blockM}
    \begin{array}{c}
       M=\begin{bmatrix}
            \tilde{M}_1 & 0\\
            0 & \tilde{M}_2
        \end{bmatrix};\,
        M_1=\begin{bmatrix}
            \tilde{M}_1 & 0\\
            0 & M(\bar{a},\bar{a})
        \end{bmatrix};\,
        M_2=\begin{bmatrix}
            M(a,a) & 0\\
            0 & \tilde{M}_2
        \end{bmatrix}. 
    \end{array}
    \end{gather}
   Before proving~\eqref{eq:blockM}, we use it to conclude the argument. From~\eqref{eq:blockM} we deduce that
   \begin{equation}\label{eq:detM}
   \begin{split}
       \det(I-M) & = \det(\tilde{I}_1-\tilde{M}_1)\cdot \det(\tilde{I}_2-\tilde{M}_2)\\
        & =\frac{\det(I_1-M_1)\cdot \det(I_2-M_2)}{(1-M(\bar{a},\bar{a}))(1-M(a,a))}.
    \end{split}
   \end{equation}
  Moreover, \eqref{eq:M} yields $M(a,a)=-\caE(-1)(\bar{a},a)=1-\omega(a)$ and $M(\bar{a},\bar{a})=-\caE(-1)(a,\bar{a})=1-\omega(\bar{a})$. Combining \eqref{eq:ihGamma}, \eqref{eq:ihGammai} and~\eqref{eq:detM} we conclude the statement. 

   It remains to prove~\eqref{eq:blockM}. By~\eqref{eq:Mi'}, it suffices to show that $M(h,k)=0$ if either $(h,k)\in (E\Gamma_1\setminus\{\bar{a}\})\times (E\Gamma_2\setminus\{a\})$ or $(h,k)\in (E\Gamma_2\setminus\{a\})\times (E\Gamma_1\setminus\{\bar{a}\})$. 
   Recall that the only edge of $\Gamma_1$ (resp.~$\Gamma_2$) ending at $t(a)$ (resp.~$o(a)$) is $a$ (resp.~$\bar{a}$).
   Hence, if $h\in E\Gamma_1\setminus \{a,\bar{a}\}$ then $t(h)\in V\Gamma_1\setminus\{t(a)\}$ and
   every $k\in E\Gamma_2\setminus\{a\}$ satisfies $o(k)\in V\Gamma_2\setminus\{o(a)\}$. 
   Since $V\Gamma_1\setminus \{t(a)\}$ and $V\Gamma_2\setminus\{o(a)\}$ are disjoint, for such $h$ and $k$ we have $t(h)\neq o(k)$ and~\eqref{eq:M} implies that $M(h,k)=\caE(-1)(h,k)=0$. Similarly, if $h\in E\Gamma_2\setminus \{a,\bar{a}\}$ and $k\in E\Gamma_1\setminus\{\bar{a}\}$ we have $M(h,k)=\caE(-1)(h,k)=0$.
   Moreover, $M(a,k)=-\caE(-1)(\bar{a},k)=0$ for every $k\in E\Gamma_2\setminus\{a\}$ as $t(\bar{a})\neq o(k)$. Similarly, $M(\bar{a},k)=-\caE(-1)(a,k)=0$ for every $k\in E\Gamma_1\setminus\{\bar{a}\}$.
\end{proof}

\subsection*{Declaration of AI-assisted technologies in the writing process}
During the preparation of this work, the author used "Writefull" to check spelling and grammar. After using this tool, the author reviewed and edited the content as needed and takes full responsibility for the content of the published article.

\bibliography{DCzeta}

\providecommand{\bysame}{\leavevmode\hbox to3em{\hrulefill}\thinspace}
\providecommand{\MR}{\relax\ifhmode\unskip\space\fi MR }
\providecommand{\MRhref}[2]{%
  \href{http://www.ams.org/mathscinet-getitem?mr=#1}{#2}
}
\providecommand{\href}[2]{#2}
\begin{thebibliography}{10}

\bibitem{apv:build}
P.~Abramenko, J.~Parkinson, and H.~{Van Maldeghem}, \emph{Distance regularity in buildings and structure constants in {H}ecke algebras}, J.~Algebra \textbf{481} (2017), 158--187.

\bibitem{bew:simp}
C.~Banks, M.~Elder, and G.~A. Willis, \emph{Simple groups of automorphisms of trees determined by their actions on finite subtrees}, J.~Group Theory \textbf{18} (2015), no.~2, 235--261.

\bibitem{baku}
H.~Bass and R.~Kulkarni, \emph{Uniform tree lattices}, J.~Amer.~Math.~Soc. \textbf{3} (1990), no.~4, 843--902.

\bibitem{infperm}
M.~Bhattacharjee, D.~Macpherson, R.~M{\"o}ller, and {P.~M.} Neumann, \emph{Notes on infinite permutation groups}, Lecture notes in mathematics, Springer, 1998.

\bibitem{bumo}
M.~Burger and S.~Mozes, \emph{Groups acting on trees: From local to global structure}, Publ.~Math.~Inst.~Hautes {\'E}tudes Sci. \textbf{92} (2000), 113--150.

\bibitem{bumo2}
M.~Burger and S.~Mozes, \emph{Lattices in products of trees}, Publ.~Math.~IHES \textbf{92} (2001), 151--194.

\bibitem{cadm}
P.-E. Caprace and T.~De~Medts, \emph{Simple locally compact groups acting on trees and their germs of automorphisms}, Transform.~Groups \textbf{16} (2011), 375--411.

\bibitem{car}
M.~Carette, \emph{Commability of groups quasi-isometric to trees}, Ann.~Inst.~Fourier \textbf{68} (2018), no.~4, 1365--1398.

\bibitem{cawi20}
M.~Carter and G.~A. Willis, \emph{Decomposition theorems for automorphism groups of trees}, Bull.~Aust.~Math.~Soc. \textbf{103} (2020), 104--112.

\bibitem{ccw:zeta}
I.~Castellano, G.~Chinello, and Th. Weigel, \emph{The {H}attori--{S}tallings rank, the {E}uler--{P}oincar\'e characteristic and zeta functions of totally disconnected locally compact groups}, 2024, ArXiv preprint.

\bibitem{deit:prim}
A.~Deitmar, \emph{Weighted prime geodesic theorems}, Eur.~J.~Math. \textbf{9} (2019).

\bibitem{deika:ihara}
A.~Deitmar and M.-H. Kang, \emph{Tree-lattice zeta functions and class numbers}, Michigan Math.~J. \textbf{67} (2018), 617--645.

\bibitem{dym:thin}
J.~Dymara, \emph{Thin buildings}, Geometry \& Topology \textbf{10} (2006), 667–694.

\bibitem{ggt:aut}
A.~Garrido, Y.~Glasner, and S.~Tornier, \emph{Automorphism groups of trees: Generalities and prescribed local actions}, pp.~92--116, Cambridge University Press, 2018.

\bibitem{gri:lang}
R.~I. Grigorchuk, \emph{{G}rowth functions, rewriting systems, and the {E}uler characteristic}, Math.~Notes \textbf{58} (1995), 653–668.

\bibitem{gss:zeta}
F.~Grunewald, D.~Segal, and G.~C. Smith, \emph{Subgroups of finite index in nilpotent groups}, Invent.~Math. \textbf{93} (1988), 185--223.

\bibitem{hari}
G.~H. Hardy and M.~Riesz, \emph{{T}he {G}eneral {T}heory of {D}irichlet's {S}eries}, Cambridge Tracts in Mathematics, Cambridge University Press, 2015.

\bibitem{matdet}
D.~A. Harville, \emph{Matrix algebra from a statistician's perspective}, New York: Springer-Verlag, 1997.

\bibitem{ihara}
Y.~Ihara, \emph{On discrete subgroups of the two by two projective linear group over $p$-adic fields}, J.~Math.~Soc.~Japan \textbf{18} (1966), 219–235.

\bibitem{pesato}
H.~D. Petersen, R.~Sauer, and A.~Thom, \emph{{$L^2$}-{B}etti numbers of totally disconnected groups and their approximation by {B}etti numbers of lattices}, J.~Topol. \textbf{11} (2018), no.~1, 257--282.

\bibitem{radu}
N.~Radu, \emph{A classification theorem for boundary 2-transitive automorphism groups of trees}, Invent.~Math. \textbf{209} (2017), 1--60.

\bibitem{reid}
C.~D. Reid, \emph{Rigid stabilizers and local prosolubility for boundary-transitive actions on tree}, J.~Group Theory \textbf{28} (2025), no.~3, 493--561.

\bibitem{resm:motiv}
C.~D. Reid and S.~M. Smith, \emph{An introduction to the local-to-global behaviour of groups acting on trees and the theory of local action diagrams}, pp.~295 -- 340, Groups St Andrews 2022 in Newcastle, CUP, 2024.

\bibitem{resm:lad}
C.~D. Reid and S.~M.~Smith (with an appendix by S.~Tornier) Smith, \emph{{G}roups acting on trees with {T}its' independence property~{(P)}}, 2022, ArXiv preprint.

\bibitem{rovo}
T.~Rossmann and C.~Voll, \emph{Groups, graphs, and hypergraphs: Average sizes of kernels of generic matrices with support constraints}, Mem.~Amer.~Math.~Soc. \textbf{294} (2024), no.~1465, 1--132.

\bibitem{mds:intro}
M.~du Sautoy, \emph{Zeta functions of groups: the quest for order versus the flight from ennui}, London Mathematical Society Lecture Note Series, p.~150–189, Cambridge University Press, 2003.

\bibitem{ser:trees}
J-P. Serre, \emph{Trees}, Springer-Verlag, Berlin-New York, 1980, Translated from the French by J.~Stillwell.

\bibitem{sm:subsiz}
S.~M. Smith, \emph{Subdegree growth rates of infinite primitive permutation groups}, J.~Lond.~Math.~Soc. \textbf{82} (2006), 526–548.

\bibitem{sm:UFF}
\bysame, \emph{A product for permutation groups and topological groups}, Duke Math.~J. \textbf{166} (2017), 2965--2999.

\bibitem{smy:eul}
N.~Smythe, \emph{{G}rowth functions and {E}uler series}, Invent.~Math. \textbf{77} (1984), 517--531.

\bibitem{tab:mat}
M.~Taboga, \emph{Determinant of a block-matrix}, Lectures on matrix algebra (2021).

\bibitem{tit:tree}
J.~Tits, \emph{Sur le groupe des automorphismes d'un arbre}, pp.~188--211, Springer Berlin Heidelberg, 1970.

\bibitem{tutte}
W.~T. Tutte, \emph{A family of cubical graphs}, Proc.~Cambridge Phil.~Soc. \textbf{43} (1947), 459--474.

\end{thebibliography}
\bibliographystyle{amsplain}

\end{document}